\pdfoutput = 1

\documentclass[11pt]{article}
\usepackage{fullpage}

\usepackage[utf8]{inputenc} 
\usepackage[T1]{fontenc}    
\usepackage{hyperref}       
\usepackage{url}            
\usepackage{booktabs}       
\usepackage{amsfonts}       
\usepackage{nicefrac}       
\usepackage{microtype}      
\usepackage{xcolor}         

\usepackage[pdftex]{graphicx}
\usepackage{subcaption}

\usepackage{amsfonts}       
\usepackage{amssymb}
\usepackage{amsmath}
\usepackage{amsthm}
\usepackage{mathtools}
\usepackage{multicol}
\usepackage{multirow}

\usepackage{color}

\usepackage{array}
\usepackage{algorithm}
\usepackage{algorithmic}

\usepackage[shortlabels]{enumitem}

\usepackage[symbol]{footmisc} 

\def\M{\mathcal{M}}

\def\D{\mathrm{D}}

\def\sR{{\mathbb R}}
\def\sS{{\mathbb S}}

\def\gP{{\mathcal{P}}}

\def\gX{{\mathcal{X}}}

\def\gS{{\mathcal{S}}}

\def\gT{{\mathcal{T}}}
\def\gE{{\mathcal{E}}}
\def\gR{{\mathcal{R}}}
\def\gG{{\mathcal{G}}}

\def\id{{\mathrm{id}}}

\def\expm{{\mathrm{expm}}}
\def\logm{{\mathrm{logm}}}
\def\grad{{\mathrm{grad}}}
\def\hess{{\mathrm{Hess}}}

\def\Exp{{\mathrm{Exp}}}

\def\Riem{{\mathrm{Riem}}}
\def\Retr{{\mathrm{Retr}}}

\newcommand{\trace}{\mathrm{tr}}

\newtheorem{theorem}{Theorem}
\newtheorem{lemma}{Lemma}
\newtheorem{proposition}{Proposition}

\theoremstyle{definition}
\newtheorem{definition}{Definition}
\newtheorem{assumption}{Assumption}
\newtheorem{remark}{Remark}

\DeclareMathOperator*{\argmin}{arg\,min}

\makeatletter
\newcommand{\thickhline}{%
    \noalign {\ifnum 0=`}\fi \hrule height 1pt
    \futurelet \reserved@a \@xhline
}
\newcolumntype{"}{@{\hskip\tabcolsep\vrule width 1pt\hskip\tabcolsep}}
\makeatother

\title{\textbf{Riemannian accelerated gradient methods \\ via extrapolation}}

\author{Andi Han\footnote{University of Sydney 
  (\texttt{andi.han@sydney.edu.au}, \texttt{junbin.gao@sydney.edu.au}).}
\and Bamdev Mishra\footnote{Microsoft India.
  (\texttt{bamdevm@microsoft.com}, \texttt{pratik.jawanpuria@microsoft.com}).}
 \and Pratik Jawanpuria\footnotemark[2] 
\and Junbin Gao\footnotemark[1]}
\date{}

\usepackage{graphicx}

\begin{document}

\maketitle

\begin{abstract}
    In this paper, we propose a simple acceleration scheme for Riemannian gradient methods by extrapolating iterates on manifolds. We show when the iterates are generated from Riemannian gradient descent method, the accelerated scheme achieves the optimal convergence rate asymptotically and is computationally more favorable than the recently proposed Riemannian Nesterov accelerated gradient methods. Our experiments verify the practical benefit of the novel acceleration strategy. 
\end{abstract}

\section{Introduction}
In this paper, we consider the following optimization problem 
\begin{equation}
    \min_{x \in \M} f(x), \label{riem_opt_main}
\end{equation}
where $\M$ is a Riemannian manifold and $f: \M \xrightarrow{} \sR$ is a smooth, real-valued function. Optimization on a Riemannian manifold naturally appears in various fields of applications, including principal component analysis \cite{edelman1998geometry,zhang2016riemannian}, matrix completion and factorization \cite{keshavan2009gradient,vandereycken2013low,boumal2015low}, dictionary learning \cite{cherian2016riemannian,harandi2013dictionary}, optimal transport \cite{shi2021coupling,mishra2021manifold,han2022riemannian}, to name a few. Riemannian optimization \cite{absil2009optimization,boumal2020introduction} provides a universal and efficient framework for problem \eqref{riem_opt_main} that respects the intrinsic geometry of the constraint set. In addition, many non-convex problems turns out to be geodesic convex (a generalized notion of convexity) on the manifold, which yields better convergence guarantees for Riemannian optimization methods. 

One of the most fundamental solvers is the Riemannian gradient descent method \cite{udriste2013convex,zhang2016first,absil2009optimization,boumal2020introduction}, which generalizes the classical gradient descent method in the Euclidean space with intrinsic updates on manifolds. There also exist various advanced algorithms for Riemannian optimization that include stochastic and variance reduced methods \cite{bonnabel2013stochastic,zhang2016riemannian,kasai2018riemannian,han2021improved,hanmomentum2021}, adaptive gradient methods \cite{becigneul2018riemannian,kasai2019riemannian} quasi-Newton methods \cite{huang2015broyden,qi2010riemannian}, trust region methods \cite{absil2007trust}, and cubic regularized Newton methods \cite{agarwal2021adaptive}, among others.

Nevertheless, it remains unclear whether there exists a simple strategy to accelerate first-order algorithms on Riemannian manifolds.
Existing research on accelerated gradient methods focus primarily on generalizing Nesterov acceleration \cite{nesterov1983method} to Riemannian manifolds, including \cite{liu2017accelerated,ahn2020nesterov,zhang2018estimate,alimisis2020continuous,jin2022understanding,kim2022accelerated}. However, most of the algorithms are theoretic constructs and are usually less favourable in practice. This is due to the excessive use of exponential map, inverse exponential map, and parallel transport (formally defined in Section \ref{prelim_sect}) each iteration. In addition, the Nesterov acceleration based methods require the knowledge of smoothness and strong convexity constants, which are often unknown. Furthermore, recent studies \cite{hamilton2021no,criscitiello2022negative} show global acceleration cannot be achieved on manifolds in general, thus further questioning the utility of Nesterov type of acceleration for Riemannian optimization.

In this paper, we undertake a different approach for Riemannian acceleration compared to the aforementioned studies. The main idea is based on extrapolation, which aims to combine a (converging) sequence of iterates on manifolds to produce an accelerated sequence. We highlight that the iterate sequence can be given by any first-order solver, not necessarily the Riemannian gradient descent. In the Euclidean space, such an idea dates back to Aitken's $\Delta^2$ process \cite{aitken_1927} and Shanks sequence transformation \cite{shanks1955non,brezinski2018shanks}, which are collectively referred to as nonlinear acceleration. Similar methods are also known as $\epsilon$-algorithm \cite{wynn1956device}, minimal polynomial extrapolation \cite{sidi1986acceleration} and Anderson acceleration \cite{walker2011anderson}, just to name a few. The core idea of these methods is to compute extrapolation as a linear combination of the iterates where the weights depend nonlinearly on the iterates. 

Particularly, we focus on the formulation in \cite{scieur2020regularized}, which introduces regularization to nonlinear acceleration that shows better stability when optimizing non-quadratic functions. Given a sequence of iterates $x_0, ..., x_k$, \cite{scieur2020regularized} produces extrapolation $\sum_{i=0}^k c_i x_i$ where the weights $\{c_i\}$ minimize the regularized combination of residuals, i.e., $\| \sum_{i=0}^k c_i (x_{i+1} - x_i) \|_2^2 + \lambda \sum_{i=0}^k c_i^2$. It has been shown such a scheme achieves optimal convergence rate asymptotically without knowing the function-specific constants. 

A natural question is \textit{whether such simple idea of extrapolation can be generalized to Riemannian optimization so that we achieve acceleration on manifolds both theoretically and practically}. In this paper, we answer this question affirmatively. In particular, our contributions are as follows. 
\begin{itemize}
    \item We propose a simple acceleration strategy for Riemannian optimization based on the idea of extrapolation, which we call the Riemannian nonlinear acceleration (RiemNA) strategy. To extrapolate on manifolds, we develop several averaging schemes, which generalize weighted averaging in the Euclidean space from various perspectives.
    
    \item When the iterates are generated by the Riemannian gradient descent method, we show RiemNA leads to the same optimal convergence rate as the Riemannian Nesterov gradient methods \cite{liu2017accelerated,zhang2018estimate,ahn2020nesterov,kim2022accelerated} asymptotically, while being more computationally efficient. We also show the convergence is robust to the choice of different averaging schemes on manifolds.
    

    
    \item Different to most analysis for Riemannian accelerated gradient methods \cite{liu2017accelerated,zhang2018estimate,ahn2020nesterov,kim2022accelerated} that focus on exponential map and parallel transport, we show the use of general retraction and vector transport achieves the same convergence guarantees. 
    
    \item Various experiments demonstrate the superiority of RiemNA compared to both Riemannian gradient descent and Riemannian Nesterov accelerated gradient methods.
\end{itemize}

\section{Related work on Riemannian acceleration}

The first paper that generalizes the Nesterov acceleration strategy \cite{nesterov1983method} from the Euclidean space to the Riemannian manifold space is \cite{liu2017accelerated}. The algorithm however requires to compute the exact solution of a nonlinear equation each iteration and is usually infeasible in practice. In \cite{zhang2018estimate}, a tractable algorithm for Riemannian Nesterov acceleration, called RAGD, is proposed and is proved to achieve local acceleration around the optimal solution. Furthermore, \cite{ahn2020nesterov} shows global convergence and eventual acceleration of RAGD with carefully chosen parameters. Recently, \cite{kim2022accelerated} develops practical algorithms that achieve acceleration for both geodesic convex and strongly convex functions by assuming a bounded domain and curvature. Finally, \cite{jin2022understanding} proposes a general framework for analyzing acceleration on manifolds where RAGD is a special case.

Another direction of acceleration research on manifolds is inspired by the continuous dynamics formulation of Nesterov acceleration in the Euclidean space \cite{su2014differential,wibisono2016variational}. In \cite{alimisis2020continuous}, an ordinary differential equation (ODE) is introduced to model Nesterov acceleration on manifolds, and later \cite{duruisseaux2022variational} provides a variational framework for understanding Riemannian acceleration. Although the ODEs achieve acceleration, discretization of the continuous processes does not necessarily preserves the same accelerated rate of convergence. Lastly, there also exist acceleration results for specific manifolds, including sphere and hyperbolic manifolds \cite{martinez2022global}, Stiefel manifold \cite{siegel2019accelerated} and for nonconvex optimization \cite{criscitiello2022accelerated}. For more related works on Riemannian acceleration, we refer to \cite[Appendix A]{criscitiello2022negative}.

\section{Preliminaries}
\label{prelim_sect}

\paragraph{Basic Riemannian geometry.} A Riemannian manifold $\M$ is a smooth manifold endowed with a smooth inner product structure, namely a Riemannian metric on the tangent space $T_x\M$, for all $x \in \M$. The Riemannian inner product between any $u, v \in T_x\M$ is written as $\langle u,v \rangle_x$ and the induced norm of a tangent vector $u$ is $\| u \|_x = \sqrt{\langle u,u \rangle_x}$. A `straight' line on manifold is called a geodesic $\gamma: [0,1] \xrightarrow{} \M$, which is a locally distance minimizing curve with zero acceleration. Riemannian distance between $x, y \in \M$ is $d(x,y) = \inf_{\gamma} \int_{0}^1 \| \gamma'(t) \|_{\gamma(t)} dt$ where $\gamma(0) =x$, $\gamma(1) = y$. Exponential map ${\rm Exp}_{x}: T_x\M \xrightarrow[]{} \M$ maps a tangent vector $u \in T_x\M$ to $\gamma(1)$ with $\gamma(0) = x, \gamma'(0) = u$. If between $x,y \in\M$, there exists a unique geodesic connecting them, the exponential map has an inverse ${\rm Exp}_x^{-1}(y)$ and the distance can be computed as $d(x,y) = \| {\rm Exp}_x^{-1}(y) \|_x = \| {\rm Exp}_y^{-1}(x) \|_y$. In this work, we only consider a unique-geodesic subset $\gX$ of the manifold, which we explicitly assume in Section \ref{sect_convergence_rimena}.
Parallel transport $\Gamma_{x}^y: T_x\M \xrightarrow{} T_y\M$ allows tangent vectors to be transported along a geodesic that connects $x$ to $y$ such that the induced vector fields in \textit{parallel}. It is known that parallel transport is isometric, i.e., $\langle \Gamma_{x}^y u, \Gamma_{x}^y v \rangle_y = \langle u,v \rangle_x$ for any $u, v \in T_x\M$. 

In this paper, we also consider the more general retraction and vector transport that include exponential map and parallel transport as special cases. A retraction, $\Retr_x: T_x\M \xrightarrow{} \M$ is a first-order approximation to the exponential map and a vector transport $\gT_{x}^y: T_x\M \xrightarrow{} T_y \M$ is a linear map between tangent spaces that approximates the parallel transport.

\paragraph{Function classes on Riemannian manifolds.}
For a differentiable, real-valued function $f: \M \xrightarrow{} \sR$, its Riemannian gradient at $x$, $\grad f(x) \in T_x\M$ is the unique tangent vector that satisfies $\langle \grad f(x), u \rangle_{x} = \D f(x) [u] = \langle \nabla f(x),u \rangle_{2}$, where $\D f(x)[u]$ represents the directional derivative of $f$ along $u$ and $\nabla f(x)$ is the Euclidean gradient and $\langle \cdot, \cdot\rangle_{2}$ represents the Euclidean inner product. The Riemannian Hessian at $x$, $\hess f(x): T_x\M \xrightarrow{} T_x\M$ is a self-adjoint operator, defined as the covariant derivative of the Riemannian gradient.   On Riemannian manifolds, one can extend the notion of gradient and Hessian Lipschitzness in the Euclidean space to \textit{geodesic Lipschitz gradient and Hessian} associated with the exponential map and parallel transport. Some equivalent characterizations, including function smoothness and bounded Hessian norm also exist for the Riemannian counterparts. Furthermore, the notion of convexity can be similarly generalized to \textit{geodesic convexity} where the convex combination is defined along the geodesics. Similar notions are also properly defined with respect to general retractions. The formal definitions are deferred to Appendix \ref{function_class_appendix}. See also \cite{boumal2020introduction} for a more thorough treatment.


\paragraph{Metric distortion.}
Due to the curved geometry of Riemannian manifolds, many of the metric properties in the linear space are lost. To perform convergence analysis, we require the following geometric lemmas on manifolds that provide bounds on the metric distortion.

\begin{lemma}[\cite{ahn2020nesterov,sun2019escaping}]
\label{lemma_move_vectorsum}
Consider a compact subset $\gX \subseteq \M$ with unique geodesic.
Let $x, y = \Exp_{x}(u) \in \gX$ for some $u \in T_x\M$. Then for any $v \in T_x\M$, we have $d( \Exp_{x}(u + v), \Exp_{y}(\Gamma_{x}^y v) ) \leq \min\{ \| u\|, \|v \| \} C_\kappa(\| u \| + \| v \|)$,
where $\gX$ has curvature upper bounded by $\kappa$ in magnitude and $C_\kappa(r) \coloneqq \cosh(\sqrt{\kappa}r ) - {\sinh(\sqrt{\kappa} r)}/({\sqrt{\kappa} r})$.
\end{lemma}

\begin{lemma}[\cite{ahn2020nesterov,karcher1977riemannian,mangoubi2018rapid,sun2019escaping}]
\label{lemma_metric_distort}
For a compact subset $\gX \subseteq \M$ with unique geodesic, there exists constants $C_0 > 0$, $C_1, C_2 \geq 1$ that depend on the curvature and diameter of $\gX$ such that for all $x,y,z \in \gX$, $u \in T_x\M$ we have 
\begin{enumerate}[(1).]
    \item $\| \Gamma_{y}^z \Gamma_{x}^y u - \Gamma_{x}^z u \|_{z} \leq C_0 d(x,y) d(y,z) \| u\|_x$ 
    
    \item $C_1^{-1} d(x,y) \leq \| \Exp_z^{-1}(x) - \Exp_z^{-1}(y) \|_{z} \leq C_2 d(x,y)$ 
    
    \item $d\big(\Exp_{x}(u), \Exp_y (\Gamma_{x}^y u) \big) \leq C_3 d(x,y)$.
\end{enumerate}
\end{lemma}

\section{Riemannian nonlinear acceleration}
\label{riemna_main_sect}

We consider a generalization of nonlinear acceleration for Riemannian optimization via a weighted Riemannian average on the manifold. Specifically, for a set of weights $\{ c_i\}_{i=0}^k$ and a set of points $\{x_i \}_{i=0}^k$, we define the weighted Riemannian average $\bar{x}_{c, x}$ as computed according to the following recursion:
\begin{equation}
    \bar{x}_{c, x} = \tilde{x}_k, \qquad \tilde{x}_{i} = \Exp_{\tilde{x}_{i-1}} \Big(  \frac{c_i}{\sum_{j=0}^i c_j} \Exp^{-1}_{\tilde{x}_{i-1}} \big( x_i  \big) \Big), \label{mfd_recursive_average} \tag{Avg.1}
\end{equation}
with $\tilde{x}_{-1} = x_0$, for $i = 0,...,k$. In the Euclidean space, which is a special case of Riemannian manifold, \eqref{mfd_recursive_average} recovers the weighted mean as $\bar{x}_{c,x} = \sum_{i=0}^k c_i x_i$ (see Lemma \ref{waverage_recursion_lemma} in Appendix \ref{euclidean_average_appendix} for more details). The coefficients $\{c_i\}_{i=0}^k$ are determined by minimizing a weighted combination of the residuals $\Exp^{-1}_{x_i}(x_{i+1}) \in T_{x_i}\M, i = 0,..., k$. Given the residuals lie on different tangent spaces, we use parallel transport $\Gamma_{x_i}^{x_k}$ to map them to a fixed tangent space, i.e., $r_i = \Gamma_{x_i}^{x_k} \Exp^{-1}_{x_i}(x_{i+1}) \in T_{x_k}\M$. Following the work \cite{scieur2020regularized}, we also consider a regularization in the coefficients, which leads to the following optimization problem:
\begin{equation}
    \min_{c\in \sR^{k+1}: c^\top 1 = 1} \| \sum_{i=0}^k c_i r_i \|^2_{x_k} + \lambda \|c \|_2^2. \label{coeff_problem_main}
\end{equation}
We show in Proposition \ref{prop_cstar_derivation} that the optimal coefficients $c^*$ has a closed-form solution.


\begin{proposition}
\label{prop_cstar_derivation}
Let $R = [\langle r_i,  r_j\rangle_{x_k}]_{i,j} \in \sR^{(k+1) \times (k+1)}$ collects all pairwise inner products. Then
the solution $c^* = \argmin_{c \in \sR^{k+1} : c^\top 1 = 1} \| \sum_{i=0}^k c_i r_i \|^2_{x_k} + \lambda  \| c \|^2_2$ is explicitly derived as $c^* = \frac{(R + \lambda  I)^{-1} 1}{1^\top (R + \lambda  I)^{-1} 1}$.
\end{proposition}

The Riemannian nonlinear acceleration (RiemNA) strategy is presented in Algorithm \ref{Rna_algorithm}, which takes a sequence of non-diverging iterates from any solver as input and constructs an extrapolation using weights that solve \eqref{coeff_problem_main}. The extrapolation is performed in parallel to the update of the iterate sequence. It is worth mentioning that when the manifold is the Euclidean space, Algorithm \ref{Rna_algorithm} exactly recovers the nonlinear acceleration algorithm in \cite{scieur2020regularized}.

\begin{algorithm}[t]
 \caption{Riemannian nonlinear acceleration (RiemNA)}
 \label{Rna_algorithm}
 \begin{algorithmic}[1]
  \STATE \textbf{Input:} A sequence of iterates $x_0, ..., x_{k+1}$. Regularization parameter $\lambda$.
  \STATE Compute $r_i  = \Gamma_{x_i}^{x_k} {\rm Exp}_{x_i}^{-1}(x_{i+1}) \in T_{x_k}\M, i = 0,...,k$
  \STATE Solve $c^* = \argmin_{c \in \sR^{k+1} : c^\top 1 = 1} \| \sum_{i=0}^k c_i r_i \|^2_{x_k} + \lambda  \| c \|_2^2$
  \STATE \textbf{Output:} $\bar{x}_{c^*,x} = \tilde{x}_k$ computed from $\tilde{x}_{i} = \Exp_{\tilde{x}_{i-1}} \Big(  \frac{c^*_i}{\sum_{j=0}^i c^*_j} \Exp^{-1}_{\tilde{x}_{i-1}} \big( x_i  \big) \Big)$, with $\tilde{x}_{-1}= x_0$. 
 \end{algorithmic} 
\end{algorithm}

\paragraph{Computational comparison to Riemannian Nesterov acceleration.}

Based on Algorithm \ref{Rna_algorithm}, each extrapolation step requires evaluation of $2k$ times the inverse exponential map, $k$ times the parallel transport and exponential map as well as solving a linear system in dimension $k+1$. Nevertheless, in practice, we often consider a limited-memory version of Algorithm \ref{Rna_algorithm} where only the most recent $m$ iterates are used and the extrapolation only happens every $m$ iterations (see more details in Section \ref{experiment_sect}). In addition, if the iterates are given by Riemannian gradient descent, we only require evaluation of parallel transport, exponential map and inverse exponential map once per iteration on average. 
This is more efficient than even the most practical implementation of the Riemannian Nesterov accelerated gradient methods \cite{zhang2018estimate,kim2022accelerated} (with algorithms given in Appendix \ref{appendix_rnag}), which require at least twice the number of exponential and inverse exponential map per iteration. Such claims are also verified empirically in Section \ref{experiment_sect}.


\section{Convergence acceleration for Riemannian gradient descent}
\label{sect_convergence_rimena}

This section analyzes the convergence acceleration of RiemNA (Algorithm \ref{Rna_algorithm}) for the iterates generated by the Riemannian gradient descent (RGD) method \cite{absil2009optimization,boumal2020introduction}. In particular, we show that the extrapolated point (output of Algorithm \ref{Rna_algorithm}) eventually becomes a good estimate of the optimal point with more iterations. We start by making the following assumption throughout the paper.
\begin{assumption}
\label{assump_normal_nei}
Let $x^* \in \mathcal{M}$ be a local minimizer of $f$. The iterates generated, i.e., $x_0, x_1 ,...$ stay within a neighbourhood $\gX$ around $x^*$ with unique geodesic. Furthermore, the sequence of iterates is non-divergent, i.e., $d(x_k, x^*) = O(d(x_0, x^*))$ for all $k \geq 0$. 
\end{assumption}
The former condition in Assumption \ref{assump_normal_nei} ensures the exponential map is invertible and is standard for analyzing accelerated gradient methods on manifolds \cite{ahn2020nesterov,jin2022understanding,kim2022accelerated}. This condition is satisfied for any non-positively curved manifolds, such as symmetric positive definite (SPD) manifold with the affine-invariant metric \cite{bhatia2009positive}. In addition, this also holds true for any sufficiently small subset $\gX$ of any manifold. 

\paragraph{Linear iterates and error decomposition in the Euclidean space.} 
First, we recall that the convergence analysis for the Euclidean nonlinear acceleration \cite{scieur2020regularized} relies critically on a sequence of linear fixed-point iterates that satisfy $\hat{x}_{i} - x^* = G (\hat{x}_{i-1} - x^*)$ for some positive semi-definite and contractive matrix $G$. The main idea is to show the algorithm converges optimally on $\hat{x}_i$ and then bound the deviation arising from the nonlinearity. Particularly, let $\{x_i\}_{i=0}^k$ be the given iterates and $\{\hat{x}_i\}_{i=0}^k$ be the linear iterates. Consider $c^*, \hat{c}^*$ as the coefficients solving \eqref{coeff_problem_main} in the Euclidean setup using $\{x_i\}_{i=0}^k$, $\{\hat{x}_i\}_{i=0}^k$ respectively. The convergence analysis in \cite{scieur2020regularized} aims to bound each term from the following error decomposition:
\begin{equation*}
     \sum_{i=0}^k c_i^* x_i - x^* = \underbrace{\sum_{i=0}^k  \hat{c}^*_i \hat{x}_i - x^* }_{\rm Linear \, term} + \underbrace{\sum_{i=0}^k ( c^*_i - \hat{c}^*_i ) \hat{x}_i}_{\rm Stability} + \underbrace{\sum_{i=0}^k c_i^* (x_i - \hat{x}_i) }_{\rm Nonlinearity}.
\end{equation*}

\paragraph{From linearized iterates to iterates on manifolds.}
On general Riemannian manifolds, due to the curved geometry, it becomes nontrivial to generalize such error decomposition. To mitigate this, we start with identification of linearized iterates on manifolds in the tangent space of $x^*$. For notational convenience, we denote $\Delta_{x} \coloneqq \Exp^{-1}_{x^*}(x)$ for any $x \in \gX$. We now consider the linearized iterates $\hat{x}_i$ are produced by the following progression as 
\begin{equation}
    \Delta_{\hat{x}_i} = G [\Delta_{\hat{x}_{i-1}}], \label{lin_iter_mfd}
\end{equation}
for some $G: T_{x^*}\M \xrightarrow{} T_{x^*}\M$ as a self-adjoint, positive semi-definite operator with $\| G\|_{x^*} \leq \sigma < 1$, where we denote $\| A \|_{x^*}$ as the operator norm for any linear operator $A$ on the tangent space $T_{x^*}\M$. In fact, we show in the Lemma \ref{lin_iter_approx_lemma} that the progression of iterates from the Riemannian gradient descent method is locally linear on the tangent space of the local minimizer $x^*$. This requires the following regularity assumption on the objective function $f$. 
\begin{assumption}
\label{assump_regularity_f}
Function $f$ has geodesic Lipschitz gradient and Lipschitz Hessian.
\end{assumption} 

\begin{remark}
Assumption \ref{assump_regularity_f} is used to ensure sufficient smoothness of the function such that the Riemannian gradient and Hessian are bounded at optimality. 
\end{remark}

\begin{lemma}
\label{lin_iter_approx_lemma}
Under Assumptions \ref{assump_normal_nei}, \ref{assump_regularity_f}, suppose the iterates, generated by the Riemannian gradient descent method, are $x_{i+1} = \Exp_{x_i}(- \eta \,\grad f(x_i))$. 
Then, we have 
\begin{equation*}
    \Delta_{x_i} = \big( \id - \eta \, \hess f(x^*) \big) [\Delta_{x_{i-1}}] + \varepsilon_i
\end{equation*}
where $\id$ denotes the identity operator and $\| \varepsilon_i \|_{x^*} = O( d^2(x_i, x^*) )$ and $\varepsilon_0 = 0$.
\end{lemma}

Lemma \ref{lin_iter_approx_lemma} suggests that it is reasonable to consider the linearized iterates $\{\hat{x}_k\}$ defined in \eqref{lin_iter_mfd} where $G = \id - \eta \, \hess f(x^*)$. It is clear that for a local minimizer $x^*$, there exists $\mu, L > 0$ such that $\mu \, \id \preceq \hess f(x^*) \preceq L \, \id$.
This is irrespective of whether the function $f$ is geodesic strongly convex or has geodesic Lipschitz gradient. Thus, for proper choices of $\eta$, we can always ensure $G$ is positive semi-definite and contractive. 

In this paper, the convergence analysis focus on the case when $G = \id - \eta \, \hess f(x^*)$ and $\{x_i\}$ are given by Riemannian gradient descent to simplify the bounds. However, we highlight that most of the analysis holds for more general and symmetric $G$.

Now, we approach the convergence of $\bar{x}_{c^*, x}$ computed from manifold weighted average recursion in \eqref{mfd_recursive_average} by considering the following error bound decomposition:
\begin{equation*}
    d(\bar{x}_{c^*, x}, x^*) \leq  \underbrace{d(\bar{x}_{\hat{c}^*, \hat{x}}, x^*)}_{\rm Linear \, term} + \underbrace{d(\bar{x}_{\hat{c}^*,\hat{x}}, \bar{x}_{c^*, \hat{x}})}_{\rm Stability} + \underbrace{d(\bar{x}_{c^*, \hat{x}}, \bar{x}_{c^*, x})}_{\rm Nonlinearity},
\end{equation*}
where we denote $\hat{c}^*$ as the coefficients solving \eqref{coeff_problem_main} with the residuals $\hat{r}_i = \Delta_{\hat{x}_{i+1}} - \Delta_{\hat{x}_i}$ from the linearized iterates $\{\hat{x}_i\}$ in \eqref{lin_iter_mfd} and $\bar{x}_{\hat{c}^*, \hat{x}}$, $\bar{x}_{{c}^*, \hat{x}}$ as weighted average computed using pairs $\{ (\hat{c}^*_i, \hat{x}_i) \}_{i=0}^k$ and $\{ ({c}^*_i, \hat{x}_i) \}_{i=0}^k$ respectively. Before we bound each of the error term, we first present a lemma relating the averaging on manifolds to averaging on the tangent space. 

\begin{lemma}
\label{lemma_deviation_recur_tange_avera}
Under Assumption \ref{assump_normal_nei}, for some coefficients $\{c_i\}_{i=0}^k$ with $\sum_{i=0}^k c_i = 1$ and any iterate sequence $\{x_i\}_{i=0}^k$, consider $\bar{x}_{c, x}$ computed from \eqref{mfd_recursive_average} via the given coefficients and the iterates. Then, we have 
\begin{equation*}
    \Delta_{\bar{x}_{c, x}} = \sum_{i = 0}^k c_i \Delta_{x_i} + e,
\end{equation*}
where $\|e \|_{x^*} = O(d^3(x_0, x^*))$.
\end{lemma}

\begin{remark}
Lemma \ref{lemma_deviation_recur_tange_avera} shows that the error between the averaging on the manifold and averaging on the tangent space is on the order of $O(d^3(x_0, x^*))$. This relies heavily on the metric distortion bound given in Lemma \ref{lemma_move_vectorsum}, \ref{lemma_metric_distort}, which only holds for the case of exponential map and parallel transport. Nevertheless, we highlight that when the general retraction and vector transport are used, we can follow the idea of \cite[Lemma 12]{tripuraneni2018averaging} to show the error is on the order of $O(d^2(x_0, x^*))$. See Proposition \ref{general_bounded_deviation_weighted_aver} and Section \ref{retr_conv_sect} for more details where we discuss the convergence under the more general setup. 
\end{remark}

\paragraph{Error bound from the linear term.}
As in the Euclidean space, we see that the extrapolation using the linearized iterates converges in a near-optimal rate, via the regularized Chebyshev polynomial. 

\begin{definition}[Regularized Chebyshev polynomial \cite{scieur2020regularized}]
The regularized Chebyshev polynomial of degree $k$, in the range of $[0,\sigma]$ with a regularization parameter $\alpha$, denoted as $C^{[0,\sigma]}_{k,\alpha}(x)$ is defined as $C^{[0,\sigma]}_{k,\alpha}(x) = \argmin_{p \in \gP_k^1} \max_{x \in [0,\sigma]} p^2(x) + \alpha \|p \|^2_2$,
where we denote $\gP_k^1 \coloneqq \{ p \in \sR[x] : \deg(p) = k, p(1) = 1 \}$ as the set of polynomials of degree $k$ with coefficients summing to $1$ and $\| p\|_2$ is the Euclidean norm of the coefficients of the polynomial $p$. We write the maximum valued as $S_{k,\alpha}^{[0,\sigma]} \coloneqq \sqrt{\max_{x \in [0,\sigma]}  (C^{[0,\sigma]}_{k,\alpha}(x))^2 + \alpha \| C^{[0,\sigma]}_{k,\alpha}(x)\|^2_2}$.
\end{definition}

Next, we present the error bound coming from the linear term in Lemma \ref{error_lin_term}, which follows from the definition of regularized Chebyshev polynomial and Lemma \ref{lemma_deviation_recur_tange_avera}. We notice that due to the curvature of the manifold, we see an additional error term $\epsilon_1$ compared to the Euclidean counterpart.

\begin{lemma}[Error from the linear term]
\label{error_lin_term}
Under Assumption \ref{assump_normal_nei},
let $\bar{x}_{\hat{c}^*, \hat{x}}$ be computed from \eqref{mfd_recursive_average} using $\{ (\hat{c}_i^*, \hat{x}_i) \}_{i=0}^k$. Then,
$$d(\bar{x}_{\hat{c}^*, \hat{x}}, x^*) \leq \frac{d(x_0, x^*)}{1-\sigma} \sqrt{ (S^{[0,\sigma]}_{k,\bar{\lambda}})^2 - \frac{\lambda}{d^2(x_0, x^*)} \| \hat{c}^* \|_2^2  } + \epsilon_1,$$ 
where $\bar{\lambda} = \lambda/d^2(x_0, x^*)$ and $\epsilon_1 = O(d^3(x_0, x^*))$.
\end{lemma}

\paragraph{Error bound from coefficient stability.}
Now we come to bound the deviation between the optimal coefficients computed via the Riemannian gradient descent iterates $\{ x_i \}$ and the linearized iterates $\{ \hat{x}_i \}$. To this end, we require the following results on the coefficients. 

\begin{lemma}[Bound on norm of coefficients]
\label{coeff_bound}
Under Assumptions~\ref{assump_normal_nei}, \ref{assump_regularity_f},
let the coefficients $c^*, \hat{c}^*$ be solved from \eqref{coeff_problem_main} using $\{x_i \}, \{\hat{x}_i \}$ respectively, where $\{x_i \}$ are given by the Riemannian gradient descent and $\{ \hat{x}_i\}$ satisfy \eqref{lin_iter_mfd}. Then, we have $\| c^* \|_2 \leq \sqrt{\frac{ \sum_{i=0}^k d^2(x_i, x_{i+1}) + \lambda  }{(k+1) \lambda } }$ and $\| c^* - \hat{c}^* \|_2 \leq \frac{1}{\lambda } \Big(  \frac{2 d(x_0, x^*)}{1-\sigma} \psi  + (\psi)^2 \Big) \| \hat{c}^*\|_2$ 
for some $\psi = O(d^2(x_0, x^*))$.
\end{lemma}
It should be noted that in the Euclidean space, $\psi = \sum_{i=0}^k \| \Delta_{x_i} - \Delta_{\hat{x}_i} \|_2 = \| x_i - \hat{x}_i \|_2$ and also can be shown to have an order of $O(d^2(x_0,x^*))$ under some Lipschitz conditions on the function (See \cite[Proposition 3.8]{scieur2020regularized}). On manifolds, such error also suffers from additional distortion, which is on the order of $O(d^2(x_0,x^*))$. 

Next, based on Lemma \ref{coeff_bound}, we show that the error from coefficient stability can be bounded as follows. The proof follows from linearizing the weighted average on the tangent space $T_{x^*}\M$ where we bound the deviation arising from the coefficients. Hence, an extra error $\epsilon_2$ appears in the bound. 

\begin{lemma}[Error from coefficient estimation]
\label{error_coeff_lemma}
Under the same settings as in Lemma \ref{coeff_bound},
let $\bar{x}_{\hat{c}^*, \hat{x}}$, $\bar{x}_{{c}^*, \hat{x}}$ be computed from \eqref{mfd_recursive_average} using $\{ (\hat{c}^*_i, \hat{x}_i) \}_{i=0}^k$ and $\{ ({c}^*_i, \hat{x}_i) \}_{i=0}^k$ respectively. Then,
$$d(\bar{x}_{\hat{c}^*, \hat{x}}, \bar{x}_{c^*, \hat{x}}) \leq  \frac{C_1 }{\lambda (1-\sigma)} \Big(  \frac{2 d^2(x_0, x^*)}{1-\sigma} \psi +  d(x_0, x^*) (\psi)^2 \Big) \| \hat{c}^*\|_2 + \epsilon_2,$$
for some $\psi = O(d^2(x_0, x^*)), \epsilon_2 = O(d^3(x_0, x^*))$. 
\end{lemma}

\paragraph{Error bound from nonlinearity.}
Next, we show that the nonlinearity term can be bounded in Lemma \ref{error_nonlinear}, which follows a similar idea of linearization on a fixed tangent space. Additional error $\epsilon_3$ is again due to the curvature of the manifold, which  vanishes in the case $\M$ is the Euclidean space.

\begin{lemma}[Error from the nonlinearity]
\label{error_nonlinear}
Under the same settings as in Lemma \ref{coeff_bound}, we have
$$d(\bar{x}_{c^*, \hat{x}}, \bar{x}_{c^*, x}) \leq C_1\sqrt{\frac{ \sum_{i=0}^k d^2(x_i, x_{i+1}) + \lambda }{(k+1) \lambda } } \Big( \sum_{i=0}^k \sum_{j=0}^i \| \varepsilon_j \|_{x^*} \Big) + \epsilon_3,$$
where $\|\varepsilon_j\|_{x^*} = O(d^2(x_j, x^*))$ is defined in Lemma \ref{lin_iter_approx_lemma} and $\epsilon_3 = O(d^3(x_0, x^*))$.
\end{lemma}

Finally, we combine Lemma \ref{error_lin_term}, \ref{error_coeff_lemma} and \ref{error_nonlinear} to obtain the following convergence result for Algorithm \ref{Rna_algorithm} applying on the iterates generated from the Riemannian gradient descent (RGD).

\begin{theorem}[Convergence of RiemNA with RGD iterates]
\label{main_convergence_rna1_theorem}
Under Assumptions \ref{assump_normal_nei}, \ref{assump_regularity_f}, let $\{ x_i\}_{i=0}^k$ be given by the Riemannian gradient descent method, i.e., $x_{i+1} = \Exp_{x_{i}} ( -\eta \, \grad f(x_{i}) )$ and $\{\hat{x}_i\}_{i=0}^k$ be the linearized iterates satisfying $\Delta_{\hat{x}_i} = G[\Delta_{\hat{x}_{i-1}}]$ with $G = \id - \eta \, \hess f(x^*)$, satisfying $\| G \|_{x^*} \leq \sigma < 1$. Then, Algorithm \ref{Rna_algorithm} with regularization parameter $\lambda$ produces $\bar{x}_{c^*, x^*}$ that satisfies
\begin{align*}
    d(\bar{x}_{c^*, x}, x^*) &\leq d(x_0,x^*) \frac{S^{[0,\sigma]}_{k,\bar{\lambda}}}{1-\sigma} \sqrt{1 + \frac{C_1^2 d^2(x_0, x^*) \Big(  \frac{2 d(x_0, x^*)}{1-\sigma}  \psi + (\psi)^2 \Big)^2}{\lambda^3 } } \\
    &\quad + C_1\sqrt{\frac{ \sum_{i=0}^k d^2(x_i, x_{i+1}) + \lambda }{(k+1) \lambda } } \Big( \sum_{i=0}^k \sum_{j=0}^i \| \varepsilon_j \|_{x^*} \Big)  + \epsilon_1 + \epsilon_2 +  \epsilon_3,
\end{align*}
where $\psi = O(d^2(x_0, x^*)), \epsilon_1, \epsilon_2, \epsilon_3 = O(d^3(x_0, x^*))$ and $\varepsilon_i = O(d^2(x_i, x^*))$ is defined in Lemma \ref{lin_iter_approx_lemma}.
\end{theorem}

We next show that even with additional distortion from the curved geometry of the manifold, the asymptotic optimal convergence is still guaranteed. 
This is mainly due to the fact that all errors incurred by the metric distortion, i.e., $\epsilon_1, \epsilon_2, \epsilon_3$ are on the order of at least $O(d^2(x_0, x^*))$, which is primarily attributed to Lemma \ref{lemma_deviation_recur_tange_avera}.

\begin{proposition}[Asymptotic optimal convergence rate of RiemNA with RGD iterates]
\label{asymp_optimal_convergence_prop}
Under the same settings as in Theorem \ref{main_convergence_rna1_theorem}, set $\lambda = O(d^s(x_0, x^*))$ for $s \in (2, \frac{8}{3})$. Then 
\begin{equation*}
    \lim_{d(x_0, x^*) \xrightarrow{} 0}  \frac{d(\bar{x}_{c^*, x}, x^* )}{d(x_0, x^*)} \leq \frac{1}{1-\sigma} \frac{2 }{\beta^{-k} + \beta^k}, \qquad \beta = \frac{1 - \sqrt{1-\sigma}}{1 + \sqrt{1-\sigma}}.
\end{equation*}
\end{proposition}

\begin{remark}
The asymptotic optimal convergence rate holds as long as $\epsilon_1, \epsilon_2, \epsilon_3$ are on the order of at least $O(d^2(x_0, x^*))$ such that $\lim_{d(x_0, x^*) \xrightarrow{} 0} \frac{1}{d(x_0, x^*)} (\epsilon_1 + \epsilon_2 + \epsilon_3) = 0$.
\end{remark}

\begin{remark}
Suppose at the local minimizer, we have $0\prec \mu \, \id \preceq \hess f(x^*) \preceq L \, \id$, which is satisfied beyond geodesic $\mu$-strongly convex and geodesic $L$-smooth function. Then, we see that by choosing $\eta = \frac{1}{L}$, $\sigma = 1 - \mu/L$. This corresponds to the optimal convergence rate obtained by Nesterov acceleration \cite{nesterov2003introductory} and its Riemannian extensions such as \cite{liu2017accelerated,ahn2020nesterov,kim2022accelerated} for (geodesic) strongly convex functions.
\end{remark}

\section{Alternative averaging schemes on manifolds}\label{sec:alternative_averaging}

In this section, we propose alternative averaging schemes on manifolds used for extrapolation. For the iterates obtained from the Riemannian gradient descent method, we show the schemes ensure the same asymptotically optimal convergence rate obtained in Proposition \ref{asymp_optimal_convergence_prop}.

The first scheme we consider is based on the following equality in the Euclidean space for the weighted mean, i.e.,
\begin{equation*}
    \sum_{i=0}^k c_i x_i = x_k -  (\sum_{i=0}^{k-1} c_i) (x_k - x_{k-1}) - (\sum_{i=0}^{k-2} c_i) (x_{k-1} - x_{k-2}) - \cdots - c_0 (x_1 - x_0).
\end{equation*}
Accordingly, let $\theta_i = \sum_{j=0}^i c_j, i= 0,...,k-1$. We define an alternative weighted averaging as 
\begin{equation}
    \bar{x}_{c, x} = \Exp_{x_k}\Big( - \sum_{i=0}^{k-1} \theta_i \Gamma_{x_i}^{x_k} \Exp_{x_i}^{-1}(x_{i+1}) \Big). \label{averg_mfd_alternative} \tag{Avg.2}
\end{equation}
Based on the earlier analysis, to show the convergence of $\bar{x}_{c,x}$ defined in \eqref{averg_mfd_alternative}, we only require to show that Lemma \ref{lemma_deviation_recur_tange_avera} holds for the new scheme, with an error of order at least $O(d^2(x_0, x^*))$. We formalize this claim in the next lemma and show the error is on the order of $O(d^3(x_0, x^*))$.

\begin{lemma}
\label{lemma_alternative_aver_deviation}
Under Assumption \ref{assump_normal_nei}, for some coefficients $\{c_i\}_{i=0}^k$ with $\sum_{i=0}^k c_i = 1$ and iterates $\{x_i\}_{i=0}^k$, consider $\bar{x}_{c, x} = \Exp_{x_k}\big( - \sum_{i=0}^{k-1} \theta_i \Gamma_{x_i}^{x_k} \Exp_{x_i}^{-1}(x_{i+1}) \big)$, $\theta_i = \sum_{j=0}^i c_j$. Then, we have $\|\Delta_{\bar{x}_{c, x}} - \sum_{i = 0}^k c_i \Delta_{x_i} \|_{x^*} = O(d^3(x_0, x^*))$.
\end{lemma}

In addition, we may consider the weighted Fr\'echet mean, which can also be used in place of the two aforementioned averaging schemes. The Fr\'echet mean is defined as the solution to the following optimization problem
\begin{equation}
    \bar{x}_{c,x} = \argmin_{x \in \gX} \sum_{i=0}^k c_i d^2(x, x_i).   \label{frechet_mean} \tag{Avg.3}
\end{equation}
Nevertheless, for general manifolds, it is not guaranteed the existence and uniqueness of the solution. In fact, one can ensure the uniqueness of the solution when the function $\frac{1}{2}d^2(x, x')$ is geodesic $\tau$-strongly convex in $x$. From \cite[Lemma 2]{alimisis2020continuous}, we see that the geodesic strong convexity of problem \eqref{frechet_mean} holds for sufficiently small $\gX$ on any manifold as well as for any non-positively curved manifold. Specifically, when $\M$ is non-positively curved, we have $\tau = 1$. While for other manifolds, let $D$ be the diameter of $\gX$ and $\kappa^+ >0$ be the upper curvature bound. Then, geodesic strong convexity is satisfied with $\tau < 1$ when $D < \frac{\pi}{2 \sqrt{\kappa^+}}$.

\begin{lemma}
\label{lemma_frechetmean_bound}
Under Assumption \ref{assump_normal_nei}, suppose $x \mapsto \frac{1}{2} d^2(x, x')$ is geodesic $\tau$-strongly convex in $x$ for any $x'\in \gX$. Consider $\bar{x}_{c,x} = \argmin_{x \in \gX} \sum_{i=0}^k c_i d^2(x, x_i)$. Then $d(\bar{x}_{c, x} , x^*) \leq \tau \| \sum_{i=0}^k c_i \Delta_{x_i} \|_{x^*}$ and $\|\Delta_{\bar{x}_{c, x}} - \sum_{i = 0}^k c_i \Delta_{x_i} \|_{x^*} = O(d^3(x_0, x^*))$.
\end{lemma}

Under the additional assumption of geodesic strong convexity, Lemma \ref{lemma_frechetmean_bound} shows an extra tighter bound on $d(\bar{x}_{c,x}, x^*)$, i.e., $d(\bar{x}_{c, x} , x^*) \leq \tau \| \sum_{i=0}^k c_i \Delta_{x_i} \|_{x^*}$. Thus, we see the error from the linear term does not suffer from metric distortion ($\epsilon_1 = 0$). The error bound from coefficient stability and nonlinearity terms however, still incur additional errors as the previous two averaging schemes.

Lemma \ref{lemma_alternative_aver_deviation} and \ref{lemma_frechetmean_bound} allows convergence under the two averaging schemes to be established by exactly following the same steps as before. This is sufficient to show the same convergence bound holds (i.e., Theorem \ref{main_convergence_rna1_theorem} and Proposition \ref{asymp_optimal_convergence_prop}).


\section{Convergence under general retraction and vector transport}
\label{retr_conv_sect}
It is noticed that nearly all studies on Riemannian accelerated gradient methods analyze convergence only under the exponential map and parallel transport. In this section, for the proposed Riemannian nonlinear acceleration on iterates from Riemannian gradient descent, we highlight that similar analysis and the same convergence bounds can be derived for general retraction and vector transport under additional assumptions. The main assumptions include bounding the deviation between retraction and exponential map as well as between vector transport and parallel transport. In addition, the Lipschitz gradient and Hessian need to be compatible with retraction and vector transport (defined in Appendix \ref{appendix_retr_lipschitz}). 

\begin{assumption}
\label{basic_retr_neighb_assump}
The neighbourhood $\gX$ is totally retractive where retraction has a smooth inverse. Function $f$ has retraction Lipschitz gradient and Lipschitz Hessian. 
\end{assumption}


\begin{assumption}
\label{retr_exp_bound}
There exists constants $a_0, a_1, a_2, \delta_{a_0, a_1} > 0$ such that for all $x, y, z \in \gX$, $\|\Retr^{-1}_x(y) \|_x \allowbreak \leq \delta_{a_0, a_1}$, we have
\begin{enumerate}[(1).]
    \item $a_0 d(x,y) \leq \| \Retr_x^{-1}(y) \|_x \leq a_1 d(x,y)$.
    
    \item $\| \Exp_{x}^{-1}(z) - \Retr_x^{-1}(z) \|_x \leq a_2 \| \Retr_x^{-1}(z) \|^2_x$.
\end{enumerate}
\end{assumption}

\begin{assumption}
\label{vector_transport_bound}
The vector transport $\gT_x^y$ is isometric and there exists a constant $a_3 > 0$ such that for all $x, y \in \gX$,  $\| \gT_{x}^y u - \Gamma_x^y u \|_y \leq a_3 \| \Retr^{-1}_x(y) \|_x \| u\|_x$.
\end{assumption}

\begin{remark}
Assumption \ref{retr_exp_bound} is required to bound the deviation from the retraction to the exponential map, which can be considered natural given retraction approximates the exponential map to the first-order. In fact, Assumption \ref{retr_exp_bound} has been commonly used in \cite{sato2019riemannian,kasai2018riemannian,han2021improved} for analyzing Riemannian first-order algorithms using retraction and can be satisfied for a sufficiently small neighbourhood (see for example \cite{ring2012optimization,huang2015riemannian}). Similarly, Assumption \ref{vector_transport_bound} is used to bound the deviation between the vector transport to parallel transport, which is also standard in \cite{huang2015broyden,kasai2018riemannian,han2021improved}. One can follow the procedures in \cite{huang2015broyden} to construct isometric vector transport that satisfies such condition for common manifolds like SPD manifold \cite{huang2015broyden}, Stiefel and Grassmann manifold \cite{huang2013optimization}. 
\end{remark}

In this section, we only show convergence under the recursive weighted average computation for extrapolation, i.e.,
\begin{equation}
    \bar{x}_{c,x} = \tilde{x}_k, \qquad \tilde{x}_i =  \Retr_{\tilde{x}_{i-1}} \Big( \frac{c_i}{\sum_{j=0}^i c_j}  \Retr_{\tilde{x}_{i-1}}^{-1}(x_i) \Big) \label{retr_average}.
\end{equation}
Nevertheless, similar analysis can be also performed on the alternative two averaging schemes discussed in Section \ref{sec:alternative_averaging}.

The next theorem shows that using retraction and vector transport can also achieve the asymptotic optimal convergence rate. This proof follows similarly as the case for exponential map and parallel transport while using the Assumptions \ref{retr_exp_bound}, \ref{vector_transport_bound}. In particular, both these two assumptions ensure the deviations between retraction and exponential map, vector transport and parallel transport are on the order of $O(d^2(x_0, x^*))$. Thus, the additional error terms $\epsilon_1, \epsilon_2, \epsilon_3 = O(d^2(x_0, x^*))$.


\begin{theorem}[Convergence of RiemNA with RGD iterates under general retraction and vector transport]
\label{theorem_retr_convergence}
Under Assumption \ref{assump_normal_nei}, \ref{basic_retr_neighb_assump}, \ref{retr_exp_bound} and \ref{vector_transport_bound}, let $\{ x_i \}_{i=0}^k$ be given by Riemannian gradient descent via retraction, i.e., $x_i = \Retr_{x_{i-1}}(- \eta \, \grad f(x_{i-1}))$ and $\{ \hat{x}_{i}\}_{i=0}^k$ be the linearized iterates satisfying $\Retr_{x^*}^{-1}(\hat{x}_i) = G [\Retr_{x^*}^{-1}(\hat{x}_{i-1})]$ with $G = \id - \eta \, \hess f(x^*)$, satisfying $\| G \|_{x^*} \leq \sigma < 1$. Then, using retraction and vector transport in Algorithm \ref{Rna_algorithm}  and letting $\bar{x}_{c,x}$ be computed from \eqref{retr_average}, it satisfies that
\begin{align*}
    d(\bar{x}_{c^*, x}, x^*) &\leq  \|\Retr_{x^*}^{-1}(x_0) \|_{x^*} \frac{S^{[0,\sigma]}_{k, \bar{\lambda}}}{1-\sigma}  \sqrt{\frac{1}{a_0^2} +  \frac{C_1^2 \|\Retr_{x^*}^{-1}(x_0) \|_{x^*}^2 \big( \frac{2 \psi}{1-\sigma} \|\Retr_{x^*}^{-1}(x_0) \|_{x^*}  + \psi^2 \big)^2 }{\lambda^3}} \\
    &\qquad + C_1 \sqrt{\frac{\sum_{i=0}^k  \|\Retr_{x_i}^{-1}(x_{i+1}) \|_{x_i}^2 + \lambda}{(k+1)\lambda}}\Big( \sum_{i=0}^k \sum_{j=0}^i \| \varepsilon_j \|_{x^*} \Big)  +\epsilon_1 + \epsilon_2 +  \epsilon_3,
\end{align*}
where $\psi = O(d^2(x_0, x^*))$, $\epsilon_1, \epsilon_2, \epsilon_3 = O(d^2(x_0, x^*))$ and $\varepsilon_i = O(d^2(x_i, x^*))$.
Under the same choice of $\lambda = O(d^s(x_0, x^*))$, $s \in (2, \frac{8}{3})$, the same asymptotic optimal convergence rate (Proposition \ref{asymp_optimal_convergence_prop}) holds.
\end{theorem}

\begin{algorithm}[t]
 \caption{RGD+RiemNA}
 \label{RGDplusRiermNA_algorithm}
 \begin{algorithmic}[1]
  \STATE \textbf{Input:} Initialization $x_0$, regularization parameter $\lambda$, and memory depth $m$.
  \STATE Set $t = 0$.
  \WHILE{$t \leq T$}
  \FOR{$i = 1,..., m$}
  \STATE $x_i = \Retr_{x_{i-1}}(- \eta \, \grad f(x_{i-1}))$
  \STATE $t = t + 1$
  \ENDFOR
   \STATE Compute $r_i  = -\eta \, \gT_{x_i}^{x_{m-1}}  \grad f(x_i), \,\, i = 0,...,m-1$.

   \STATE Solve $c^* = \argmin_{c \in \sR^{m} : c^\top 1 = 1} \| \sum_{i=0}^k c_i r_i \|^2_{x_{m-1}} + \lambda  \| c \|_2^2$.
   \STATE Compute $\bar{x}_{c^*,x} = \tilde{x}_{m-1}$ computed from $\tilde{x}_{i} = \Retr_{\tilde{x}_{i-1}} \Big(  \frac{c^*_i}{\sum_{j=0}^i c^*_j} \Retr^{-1}_{\tilde{x}_{i-1}} \big( x_i  \big) \Big)$, with $\tilde{x}_{-1}= x_0$. 
  \STATE Restart with $x_0 = \bar{x}_{c^*, x}$.
  \ENDWHILE
 \end{algorithmic} 
\end{algorithm}

\section{Experiments}
\label{experiment_sect}

In this section, we test the Riemannian nonlinear acceleration (RiemNA) strategy for several applications on various manifolds, including the sphere, Stiefel, Grassmann, and symmetric positive definite manifolds. To this end, we apply the proposed strategy RiemNA (Algorithm \ref{Rna_algorithm}) for accelerating iterates from the Riemannian gradient descent method with fixed stepsize (RGD), which we call RGD+RiemNA. The concrete steps are shown in Algorithm \ref{RGDplusRiermNA_algorithm}.
Specifically, we run the Riemannian gradient descent to produce $x_0, \ldots, x_{m-1}$, where $m$ is the memory depth. Then, we produce $\bar{x}_{c^*, x}$ with these iterates by Algorithm \ref{Rna_algorithm}. We only consider the recursive weighted average in \eqref{mfd_recursive_average} for the experiments. We then restart Riemannian gradient descent with $x_0 = \bar{x}_{c^*, x}$ for the next epoch. Such an implementation has also been considered in \cite{scieur2020regularized} for the Euclidean case. Based on Section \ref{retr_conv_sect}, we can apply general retraction and vector transport to implement Algorithm \ref{RGDplusRiermNA_algorithm}.

We compare the proposed RGD+RiemNA (Algorithm \ref{RGDplusRiermNA_algorithm}) with the recent development of the Riemannian Nesterov accelerated gradient (RNAG) methods \cite{kim2022accelerated}, which correspond to the current state-of-the-art. We also include RAGD (a variant of Nesterov acceleration on manifolds proposed in \cite{zhang2018estimate}) and RGD as baselines. In particular, we compare with RNAG-C \cite{kim2022accelerated} (designed for geodesic convex functions) and RNAG-SC \cite{kim2022accelerated} and RAGD \cite{zhang2018estimate} (designed for geodesic strongly convex functions) regardless of whether the objective is of the particular class. We include more details of the algorithms in Appendix \ref{appendix_rnag}. 

According to \cite{kim2022accelerated}, RNAG-C, RNAG-SC, and RAGD require the knowledge of geodesic Lipschitz constant $L$. Further, RNAG-SC and RAGD require the geodesic strong convexity parameter $\mu$. In particular, the stepsize of RNAG-C, RNAG-SC and RAGD should be set as $1/L$. If such constants are available, we set them accordingly. Otherwise, we tune over the parameters $L, \mu$ for RNAG-C, RNAG-SC to obtain the best results and set the same parameters for RAGD for comparability. Following \cite{kim2022accelerated}, the additional parameters $\xi, \zeta$ are fixed to be $1$ for RNAG-C, RNAG-SC and $\beta = \sqrt{\mu/L}/5$ for RAGD.
We set stepsize of RGD to be $1/L$ if available and tune the stepsize otherwise. For the proposed RiemNA, we fix $\lambda = 10^{-8}$ and choose memory depth over $\{ 5, 10\}$. 

For fair comparisons, we use exponential map, inverse exponential map, and parallel transport for all the algorithms whenever such operations are properly defined. For other cases, we use retraction, inverse retraction, and vector transport. It should be noted that we maintain consistency of the use of such operations across all the algorithms.

All the algorithms are stopped once gradient norm reaches below $10^{-6}$. All experiments are coded in Matlab using the Manopt toolbox \cite{boumal2014manopt}. The codes can be found on \url{https://github.com/andyjm3}.

\paragraph{Leading eigenvector computation.}
First, we consider the problem of computing the leading eigenvector of a symmetric matrix $A$ of size $d\times d$, by solving $\min_{x \in \gS^{d-1}} \{ f(x) \coloneqq - \frac{1}{2} x^\top A x \}$,
where $\gS^{d-1} \coloneqq \{ x \in \sR^{d} : \| x\|_2 = 1\}$ denotes the sphere manifold of intrinsic dimension $d-1$. For the experiment, we generate a positive definite matrix $A$ with condition number $10^3$ and exponentially decaying eigenvalues in dimension $d = 10^3$. As shown in \cite[Proposition 7.1]{kim2022accelerated}, the problem has geodesic $L$-Lipschitz gradient with $L$ to be the eigengap of matrix $A$, i.e., the difference between maximum and minimum eigenvalues of $A$. 

The stepsize is thus set as $1/L$ for all methods. For RNAG-SC and RAGD, we set $\mu = 10$. For RiemNA, we set memory depth to be $m = 10$. We use exponential and inverse exponential map as well as projection-type vector transport for all algorithms including RGD+RiemNA. The operations are given in Appendix \ref{manifolds_geometry_appendix}. 

In Figure \ref{pca_figure}, we compare the algorithms in optimality gap measured as $f(x_t) - f(x^*)$ against both iteration number and runtime. We notice that RiemNA performs competitively with Nesterov acceleration in terms of iterations and significantly outperforms all baselines in runtime. This is due to the light computational cost compared to Nesterov type of acceleration on manifolds (according to the discussion in Section \ref{riemna_main_sect}). Finally we remark that in the initial phase, RiemNA does not necessarily ensure descent in the objective, only in the later phase where acceleration takes place. This is in accordance with our local convergence analysis.

In Figure \ref{ablation_figure}, we test the sensitivity of the choices of regularization parameter $\lambda$ (on the left with $m = 10$ fixed) and memory depth $m$ (on the right with $\lambda = 10^{-8}$ fixed). The results show robustness of RiemNA under various choices of regularization parameter $\lambda$ and memory depth $m$.

\begin{figure*}[!t]
  \begin{subfigure}[t]{0.5\linewidth}
    \centering
    \includegraphics[width=0.49\linewidth]{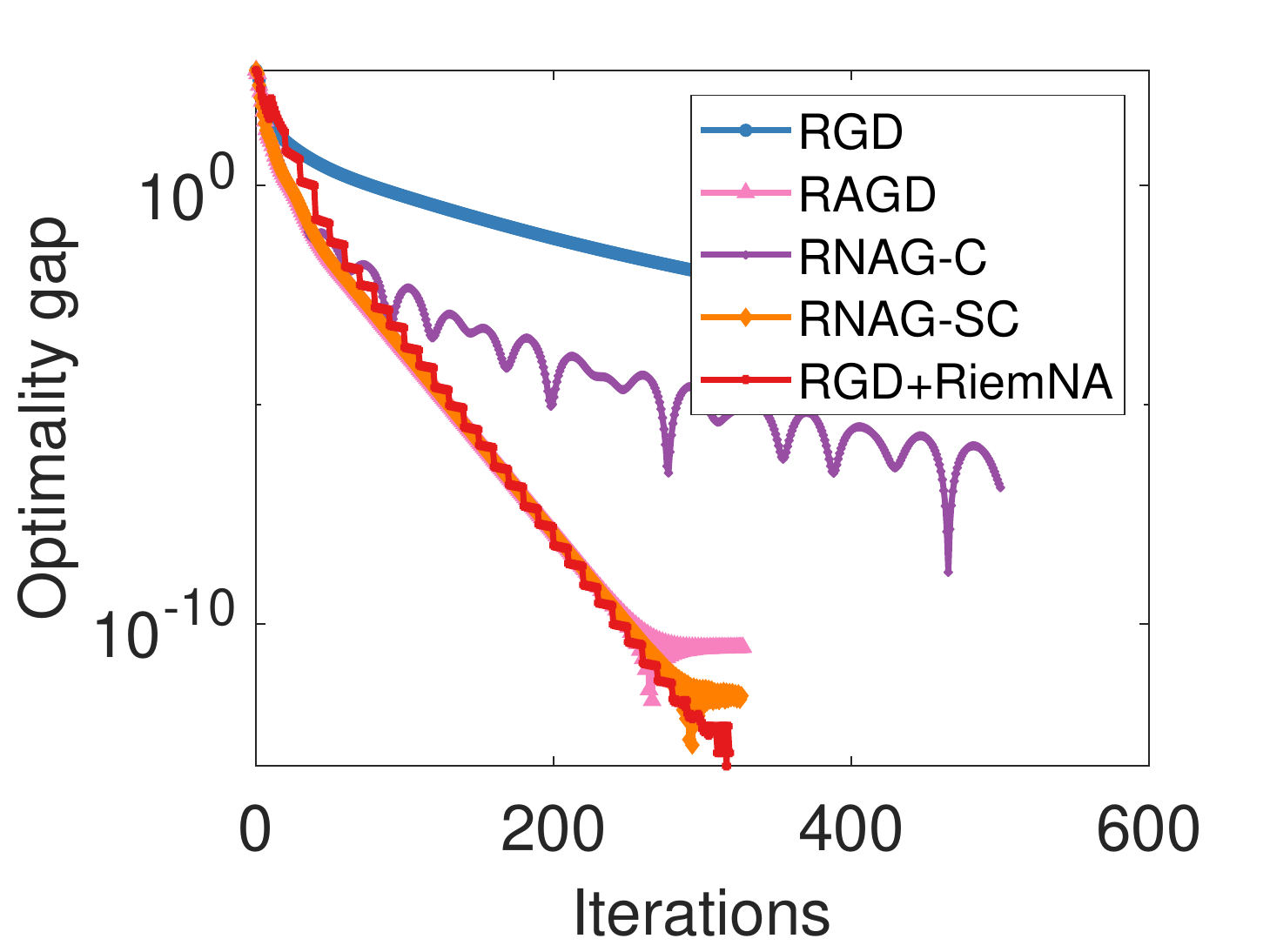}
    \includegraphics[width=0.49\linewidth]{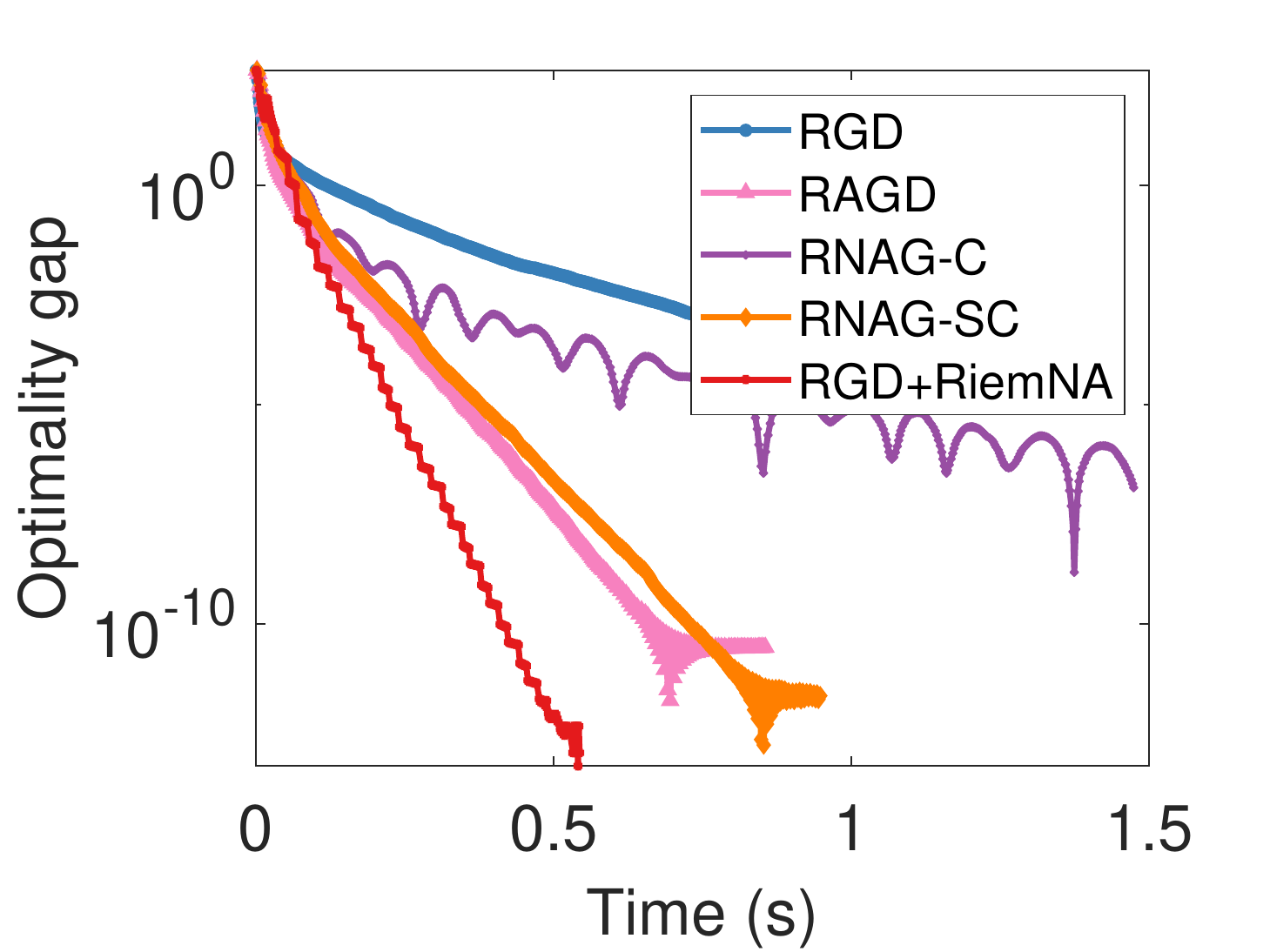}
    \caption{Sphere: Leading eigenvector}\label{pca_figure}
  \end{subfigure}
\hfill
  \begin{subfigure}[t]{.5\linewidth}
    \centering
    \includegraphics[width=0.49\textwidth]{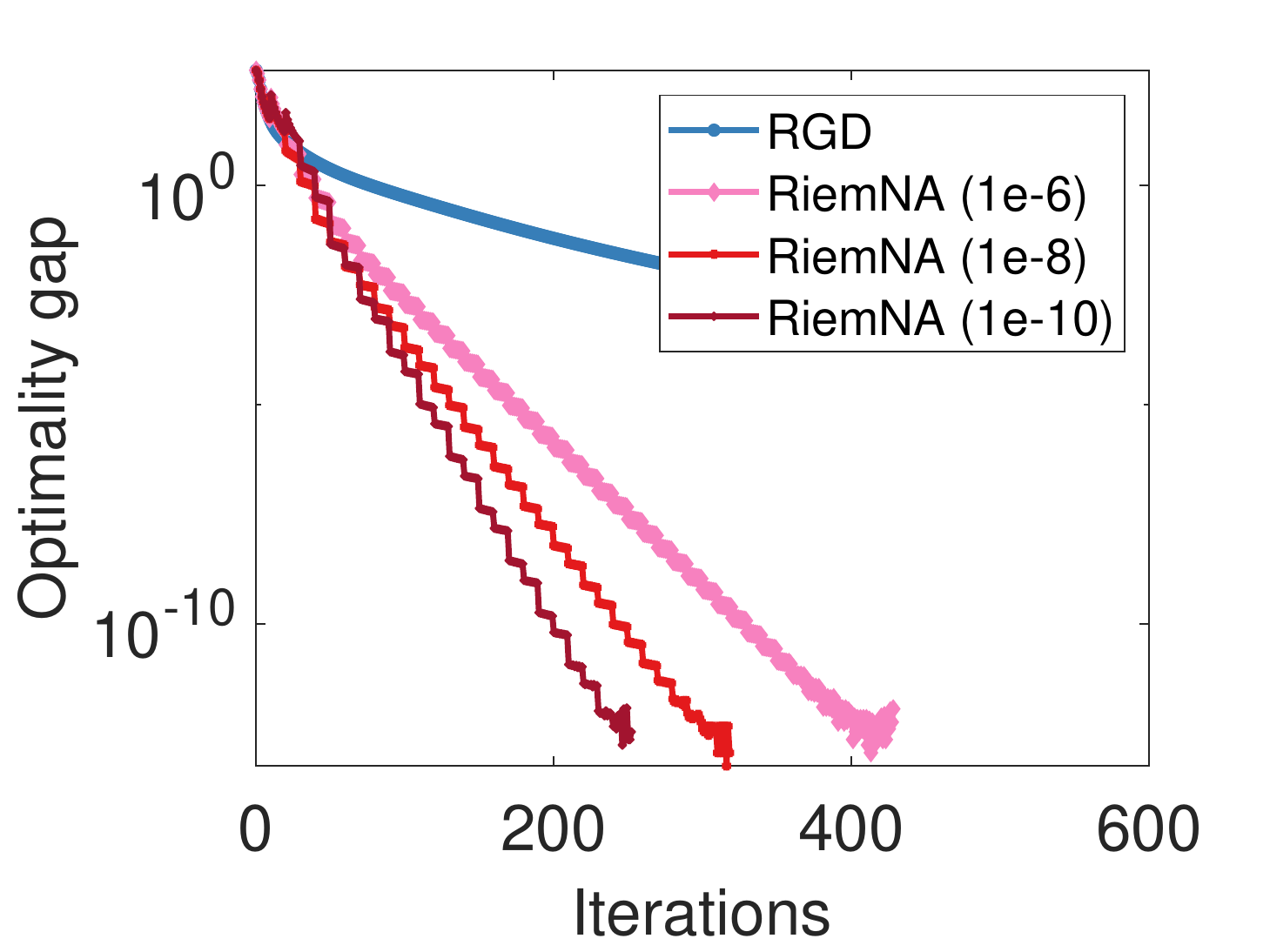}
    \includegraphics[width=0.49\textwidth]{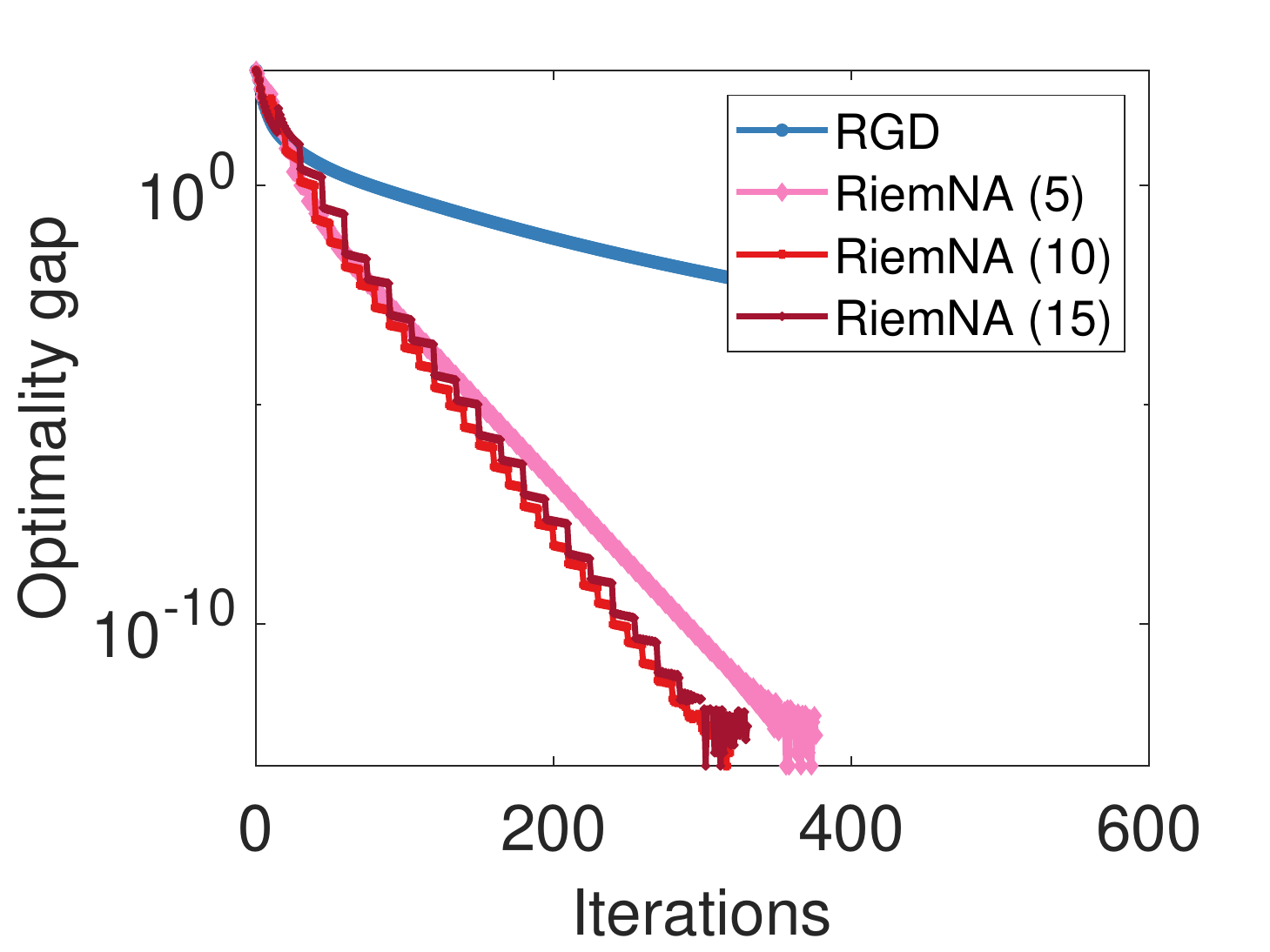}
    \caption{Parameter sensitivity (Leading eigenvector)}\label{ablation_figure}
  \end{subfigure}\\[5pt]
  \begin{subfigure}[t]{.5\linewidth}
    \centering
    \includegraphics[width=0.49\textwidth]{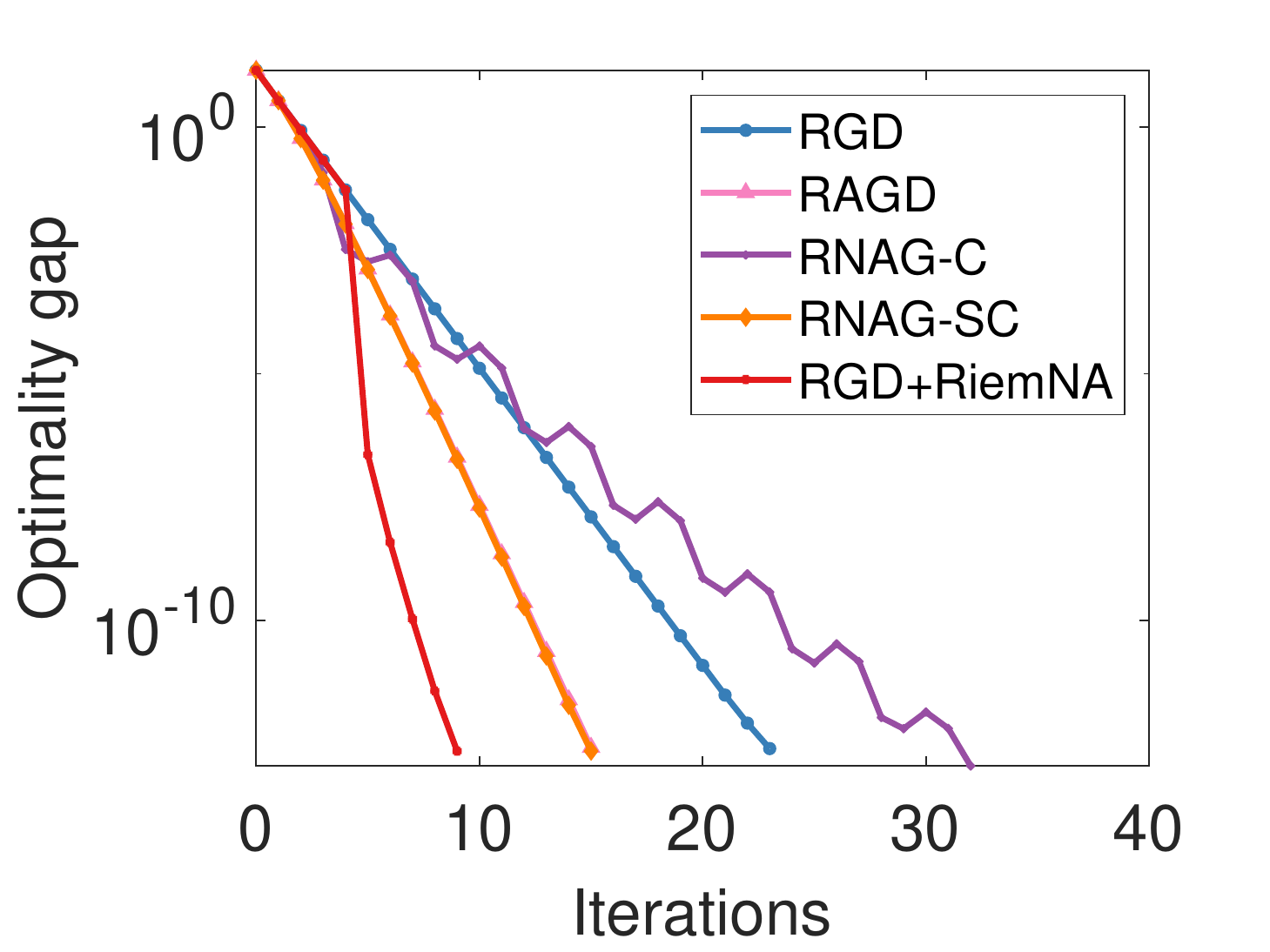}
    \includegraphics[width=0.49\textwidth]{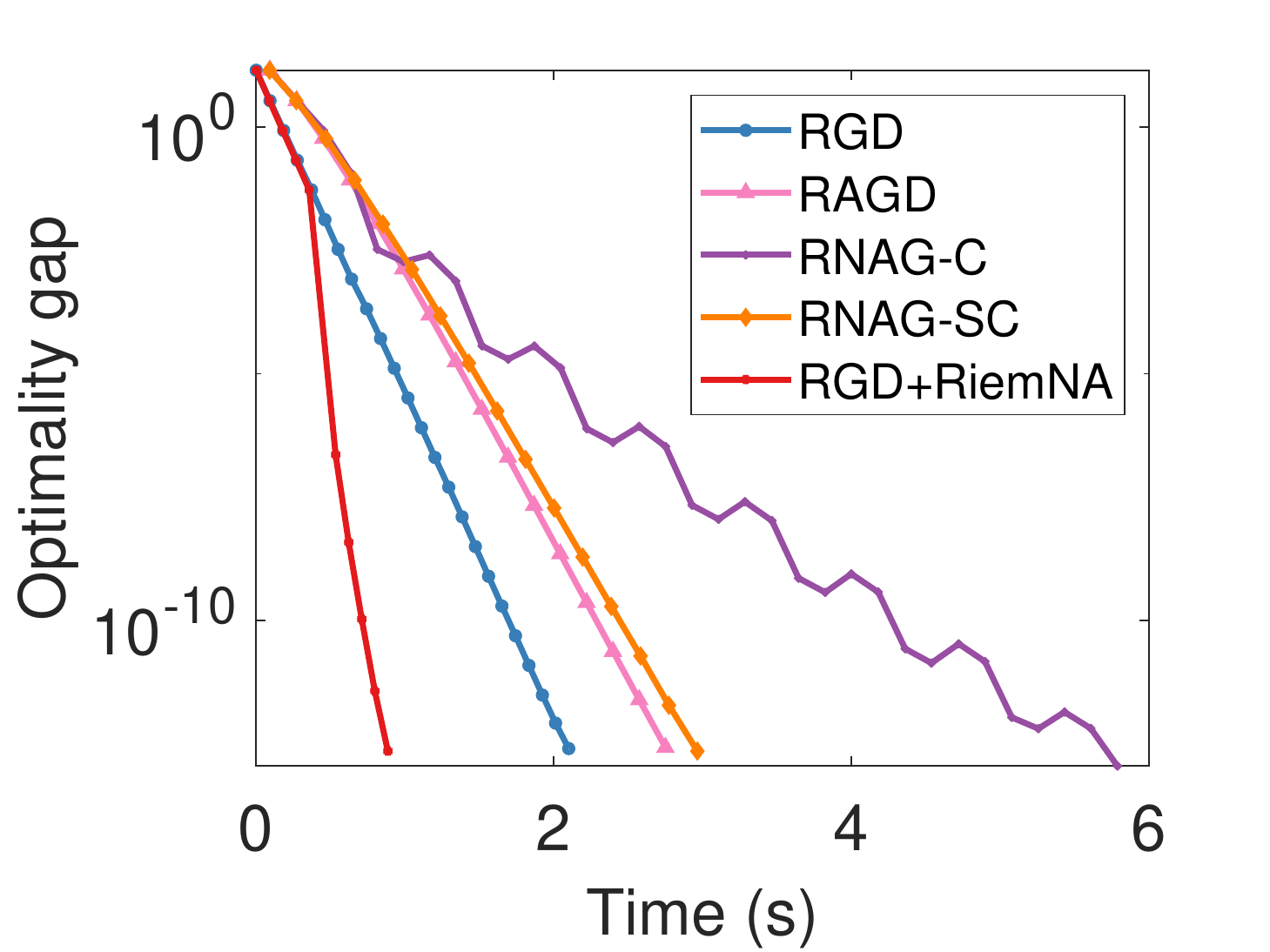}
    \caption{SPD: Fr\'echet mean}\label{spdfm_figure}
  \end{subfigure}
  \hfill
  \begin{subfigure}[t]{.5\linewidth}
    \centering
    \includegraphics[width=0.49\textwidth]{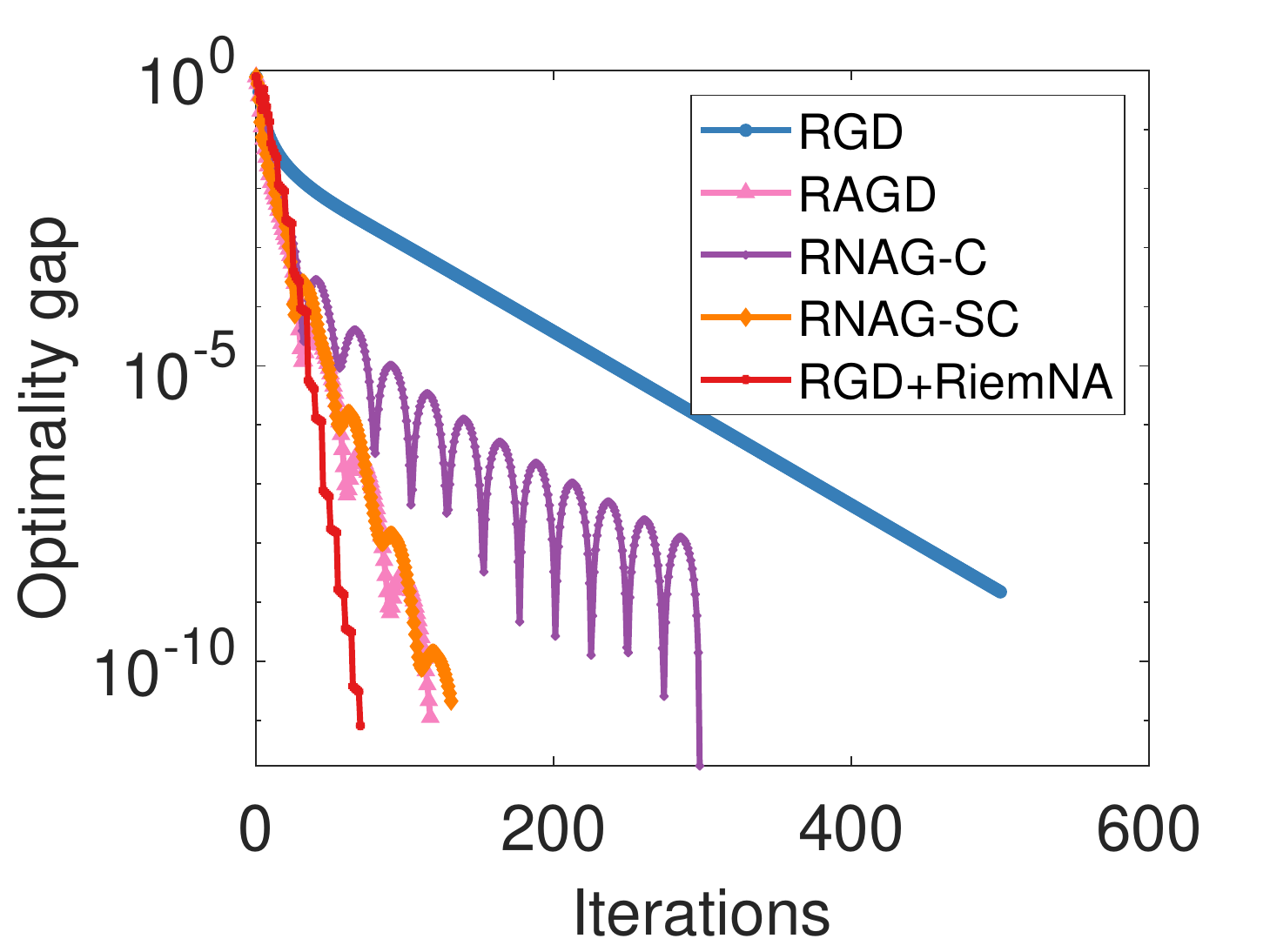}
    \includegraphics[width=0.49\textwidth]{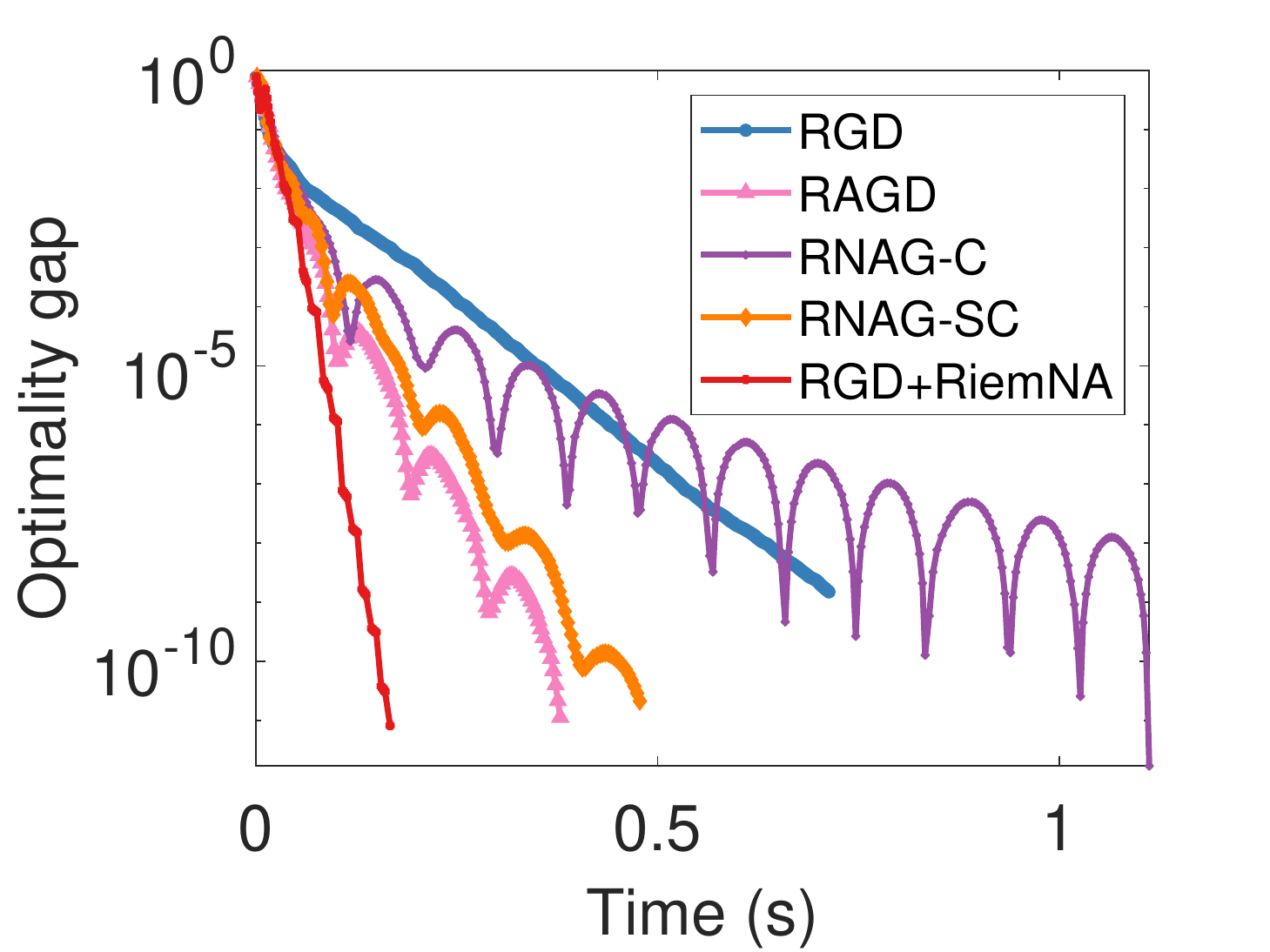}
    \caption{Stiefel: Orthogonal procrustes}\label{ortho_figure}
  \end{subfigure}\\[5pt]
  \begin{subfigure}[c]{\linewidth}
    \centering
    \includegraphics[width=0.248\textwidth]{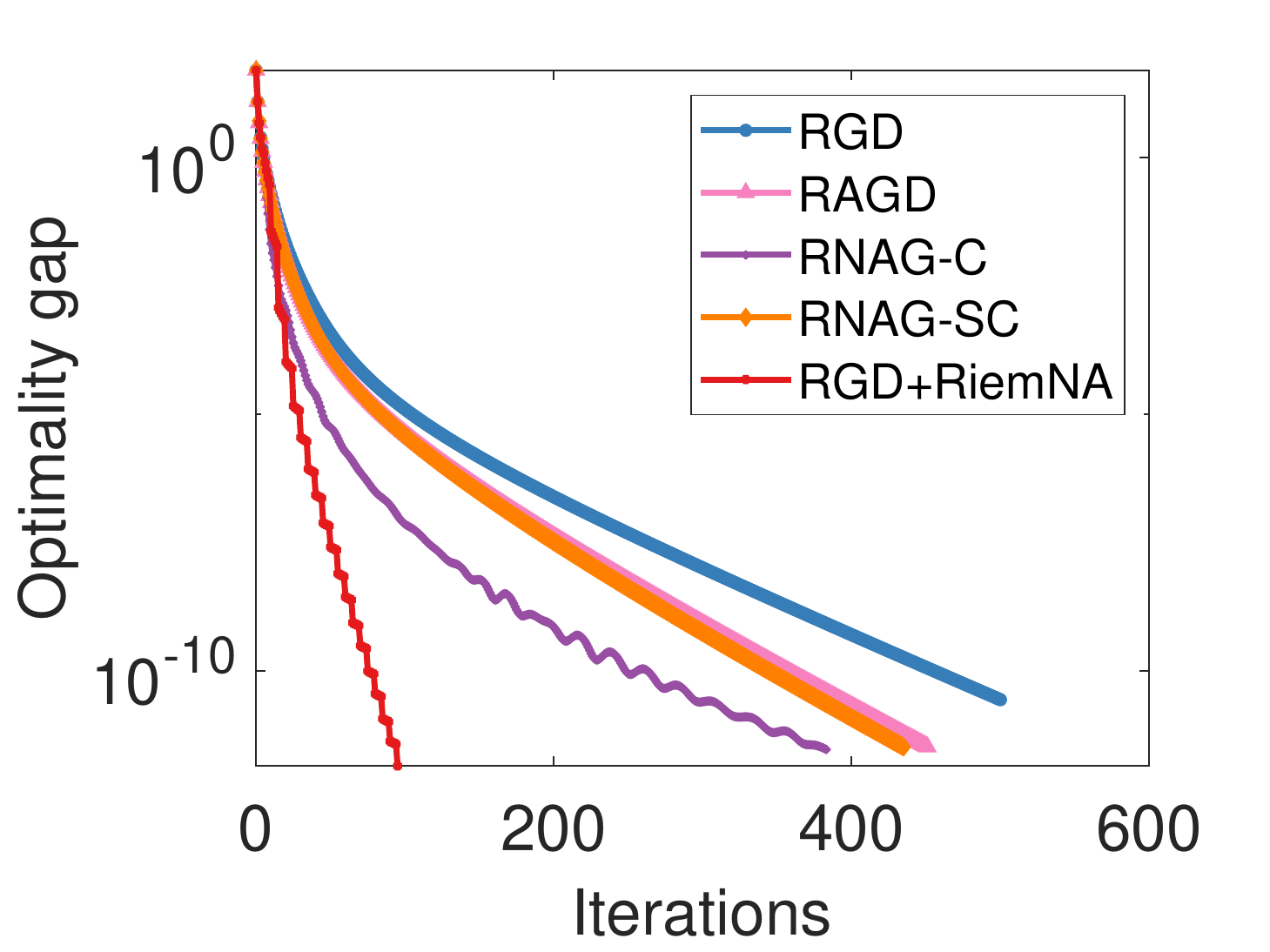}
    \includegraphics[width=0.248\textwidth]{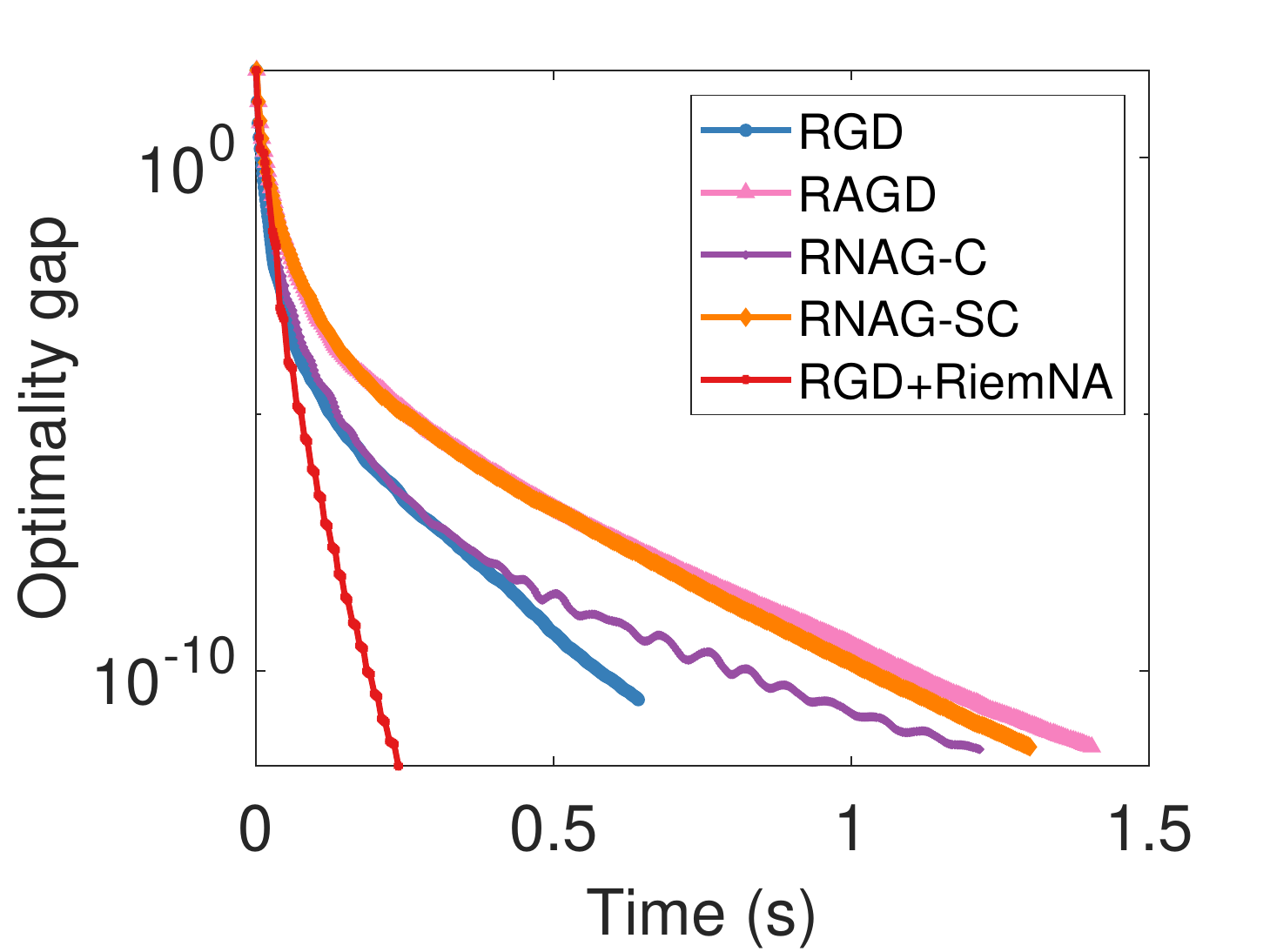}
    \caption{Grassmann: Nonlinear eigenspace}\label{nleig_figure}
  \end{subfigure}
  \caption{Experiment results on leading eigenvector computation \eqref{pca_figure}, Fr\'echet mean \eqref{spdfm_figure}, orthogonal procrustes \eqref{ortho_figure}, and nonlinear eigenspace \eqref{nleig_figure} problems where we observe that the proposed RGD+RiemNA outperforms all the baselines. Parameter sensitivity on leading eigenvector problem is in \eqref{ablation_figure} where on the left, we vary $\lambda$ by fixing $m = 10$ and on the right, we vary $m$ by fixing $\lambda = 10^{-8}$.}
\end{figure*}

\paragraph{Fr\'echet mean of SPD matrices.}
The second example is to compute the Fr\'echet mean of several symmetric positive definite (SPD) matrices under the affine-invariant metric \cite{bhatia2009positive}. Specifically, given a set of SPD matrices $\{A_i \}_{i = 1}^N$ of size $d \times d$, the aim is to $\min_{X \in \sS_{++}^d} \frac{1}{2N} \sum_{i = 1}^N \| \logm(X^{-1/2} A_i X^{-1/2}) \|_{\rm F}^2$, where we denote $\sS_{++}^d$ as the set of SPD matrices of size $d \times d$, $\| \cdot \|_{\rm F}$ as the Frobenius norm, and $\logm(\cdot)$ is the matrix logarithm. The geometry of the SPD manifold is given in Appendix \ref{manifolds_geometry_appendix}. For the experiments, we use exponential map and its inverse as well as the parallel transport for all the algorithms. As commented previously, the geometry is negatively curved, and thus, the Fr\'echet mean problem is geodesic $1$-strongly convex. Hence $\mu = 1$. For this problem, we generate random $N = 100$ SPD matrices of dimension $d = 10$. The stepsize for all methods are tuned and set to be $0.5$. For RiemNA, we set memory depth $m = 5$.  In Figure \ref{spdfm_figure}, we observe superiority of RiemNA in both iterations and runtime. We particularly notice in this case, the extrapolation step leads to a significant acceleration of convergence. 


\paragraph{Orthogonal procrustes problem.}
We also consider the orthogonal procrustes problem on Stiefel manifold. Suppose we are given $A \in \sR^{r\times r}, B \in \sR^{p \times r}$, the objective is $\min_{X \in {\rm St}(p, r)} \| X A - B \|_{\rm F}^2$ where  ${\rm St}(p, r) \coloneqq \{ X \in \sR^{p \times r} : X^\top X = I \}$ is the set of column orthonormal matrices, which forms the so-called Stiefel manifold with the canonical metric. To implement the algorithms, we use QR-based retraction and inverse retraction as well as projection-type vector transport. The definitions are given in Appendix \ref{manifolds_geometry_appendix}. We generate random matrices $A, B$ where the entries are normal distributed. We set $p = 100, r = 5$. For this problem, both $L$ and $\mu$ are unknown. Hence, we tune and set stepsize to be $1$ for all methods. For RNAG-SC and RAGD, we select $\mu = 0.005$ and for RiemNA we set memory depth $m = 5$. Figure \ref{ortho_figure} shows similar patterns as previous two experiments where RiemNA consistently converges faster than all the baselines. 


\paragraph{Nonlinear eigenspace problem.}
Finally, the problem of computing nonlinear eigenspace arises as the total energy minimization on Grassmann manifold \cite{zhao2015riemannian}, i.e., $\min_{X \in {\rm Gr}(p, r)} \allowbreak \frac{1}{2} \trace(X^\top L X) + \frac{1}{4} \rho(X)^\top L^{-1} \rho(X)$ where $\rho(X) \coloneqq {\rm diag}(X X^\top)$ and $L$ is a discrete Laplacian operator. The Grassmann manifold ${\rm Gr}(p, r)$ is the set of all $r$-dimensional subspaces in $\sR^p$, which can be identified as an equivalence class of Stiefel manifold. For experiment, we implement the algorithms with QR-based retraction and inverse retraction as well as projection-based vector transport similar to Stiefel manifold (details given in Appendix \ref{manifolds_geometry_appendix}). We generate $L$ as a tridiagonal matrix with main diagonal entries to be $2$ and sub- and super-diagonal entries to be $-1$. The stepsize is tuned and set to be $0.1$ for all methods and for RNAG-SC, RAGD, $\mu = 5$ and for RiemNA, $m = 5$. We show the convergence results in Figure \ref{nleig_figure} where we also observe that RiemNA significantly outperforms other methods. 


\section{Concluding remarks and perspectives}
In this paper, we introduce a simple strategy for accelerating first-order Riemannian optimization algorithms on manifolds, which is based on the idea of extrapolation and nonlinear acceleration. The extrapolation step is performed via novel intrinsic weighted averaging schemes on manifolds. By carefully bounding the metric distortion, we show that the Riemannian acceleration scheme achieves the same optimal convergence rate as the Euclidean counterparts. When applied to the iterates from the Riemannian gradient descent method, the new acceleration scheme asymptotically matches the convergence guarantee obtained by the Nesterov type of acceleration algorithms on manifolds, while being more practically favourable. We verify such claims for a variety of applications. 

For simplicity, this work only analyze convergence of the acceleration with the Riemannian gradient descent method. This is because locally on the tangent space at optimality, Riemannian gradient descent leads to a linear progression given by a symmetric operator $G$ as shown in Lemma \ref{lin_iter_approx_lemma}. However, for general solvers, including conjugate gradient and even Nesterov accelerated gradient methods on manifolds, Lemma \ref{lin_iter_approx_lemma} may not hold. It would be interesting to explore nonlinear acceleration for more advanced solvers and also to consider stochastic settings.





\bibliographystyle{plain}

\newpage
\appendix

\section{Extensions}

In this section, we consider various extensions to the proposed nonlinear acceleration on manifolds.



\subsection{Online Riemannian nonlinear acceleration}
Following \cite{scieur2018online,bollapragada2022nonlinear}, we can extend Algorithm \ref{Rna_algorithm} to the online setting, where the extrapolated point $\bar{x}_{c,x}$ is used to update the iterate sequence. The idea is to add a mixing step by updating $\bar{x}_{c,x}$ in the direction of the weighted average of the gradients, i.e., $\overline{\grad}f(\bar{x}_{c,x}) = \sum_{i=0}^k c_i \Gamma_{x_i}^{\bar{x}_{c,x}} \grad f(x_i)$. For the averaging schemes \eqref{mfd_recursive_average}, \eqref{frechet_mean}, the next iteration starts with $\Exp_{\bar{x}_{c,x}}(- \delta \, \overline{\grad}f(\bar{x}_{c,x}) )$ for some mixing parameter $\delta > 0$. Particularly for the tangent space averaging scheme \eqref{averg_mfd_alternative}, we show a more efficient strategy of mixing, which we focus in this paper. The averaging and mixing steps are both performed on the same tangent space. Specifically, let $x_{-1} = x_0$, we define the following progression of the online nonlinear acceleration on manifolds. 
\begin{align*}
    x_{k+1} &= \Exp_{x_k} \Big( - \sum_{i=0}^{k-1} \theta_i \Gamma_{x_i}^{x_k} \Exp_{x_i}^{-1}(x_{i+1}) - \delta \sum_{i=0}^k c_i \Gamma_{x_i}^{x_k} \grad f(x_i) \Big) \\
    &= \Exp_{x_k} \Big( - \delta c_k \grad f(x_k) -  \sum_{i=0}^{k-1} \Gamma_{x_i}^{x_k} \big( \theta_i \Exp_{x_i}^{-1}(x_{i+1}) + \delta c_i \grad f(x_i)  \big) \Big). 
\end{align*}
The complete procedures are presented in Algorithm \ref{Rnaonline_algorithm}.

\begin{algorithm}[!h]
 \caption{Riemannian nonlinear acceleration (RiemNA-online)}
 \label{Rnaonline_algorithm}
 \begin{algorithmic}[1]
  \STATE \textbf{Input:} Initialization $x_0$. Regularization parameter $\lambda$. Mixing parameter $\delta$. 
  \FOR{$k = 0,...$}
  \STATE Compute $r_i  = \Gamma_{x_i}^{x_k} {\rm Exp}_{x_i}^{-1}(x_{i+1}) \in T_{x_k}\M, i = 0,...,k$
  \STATE Solve $c^* = \argmin_{c \in \sR^{k+1} : c^\top 1 = 1} \| \sum_{i=0}^k c_i r_i \|^2_{x_k} + \lambda  \| c \|_2^2$.
  \STATE Compute $x_{k+1} = \Exp_{x_k} \big( - \delta c_k^* \grad f(x_k) -  \sum_{i=0}^{k-1} \Gamma_{x_i}^{x_k} \big( \theta_i^* \Exp_{x_i}^{-1}(x_{i+1}) + \delta c_i^* \grad f(x_i)  \big) \big)$, where $\theta_i^* = \sum_{j = 0}^i c_j^*$. 
  \ENDFOR
  \STATE \textbf{Output:} $x_{k+1}$. 
 \end{algorithmic} 
\end{algorithm}

\subsection{Practical considerations}
Here are some practical considerations to use nonlinear acceleration on manifolds.

\paragraph{Iterates from Riemannian gradient descent with line-search.}
Suppose the iterates $\{x_i\}_{i=0}^k$ are generated from $x_i = \Exp_{x_i}(- \eta_i \grad f(x_{i-1}))$ where the stepsize is determined from a line-search procedure (such as backtracking line-search \cite{boumal2019global}) and thus varies across iterations. Nevertheless, Lemma \ref{lin_iter_approx_lemma} still holds with $G_i = \id - \eta_i \hess f(x^*)$. Suppose the stepsize is chosen such that $\|G_i \| \leq \sigma < 1$. Then the analysis still holds under this setting.

\paragraph{Safeguarding decrease.}
Due to the curved geometry of the manifold and nonlinearity of the objective function, it is not guaranteed that $f(\bar{x}_{c,x})$ will decrease. In the main text, we only show local convergence of the acceleration strategy. A typical globalization technique is to only keep the extrapolated point if it shows sufficient decrease compared to previous iterates, i.e., $f(\bar{x}_{c,x}) \leq \tau \min_{i=0,...,k} f(x_i)$ for some $\tau < 1$. In \cite{scieur2020regularized}, an adaptive regularization strategy has been proposed to select regularization parameter $\lambda$. Here we adapt the same strategy on manifolds, which we show in Algorithm \ref{AdaRna_algorithm}. As noticed in \cite{scieur2020regularized}, a higher value of $\lambda$ pushes the weights close to uniform and thus stays closer to $x_0$. Thus the line-search over $t$ tries to enhance the progress compared to the initialization. In addition, for online Riemannian nonlinear acceleration specifically, we may consider performing a line-search over the parameter $\delta$ to ensure a sufficient descent condition is met.

\paragraph{Limited-memory and extrapolation frequency.}
Rather than keeping all the previous iterates for extrapolation, we can set a memory depth of $m$ and using only the most recent $m$ iterates to compute the extrapolated point. In practice, $m$ is usually set to be less than $10$. In addition, we notice that compared to the Euclidean version, the computational cost for the Riemannian nonlinear acceleration can be high due to the use of parallel transport. Hence to mitigate this issue, we may only compute the extrapolated point every $m$ iteration.

\paragraph{Efficient update of the residual matrix $R$.}
Recall for each application of Riemannian nonlinear acceleration, we need to compute $R = [\langle r_i, r_j \rangle_{x_k}]_{0 \leq i,j \leq k}$, where $r_i = \gT_{x_i}^{x_k} \Retr_{x_i}^{-1}(x_{i+1})$, where we write using (isometric) vector transport and general retraction. This includes parallel transport and exponential map as special cases. By isometry, in the next iteration when we receive $r_{k+1}$, the update of $R$ only requires computing $\langle \Gamma_{x_k}^{x_{k+1}} r_i, r_{k+1}\rangle_{x_{k+1}}$, $i = 0,...,k+1$. Denote the vector $r_+ \coloneqq [\langle \Gamma_{x_k}^{x_{k+1}} r_i, r_{k+1} \rangle_{x_{k+1}}]_{0\leq i \leq k}$. Then the updated residual matrix is
\begin{equation*}
    R_+ = 
    \begin{bmatrix}
        R & r_+ \\
        r_+^\top & \|  r_{k+1} \|^2_{x_{k+1}}
    \end{bmatrix}.
\end{equation*}

\begin{algorithm}[!h]
 \caption{Adaptive regularized Riemannian nonlinear acceleration (AdaRiemNA)}
 \label{AdaRna_algorithm}
 \begin{algorithmic}[1]
  \STATE \textbf{Input:} A sequence of iterates $x_0, ..., x_{k+1}$. Tentative regularization parameters $\{ \lambda_j \}_{j=1}^k$.
  \STATE Compute $r_i  = \Gamma_{x_i}^{x_k} {\rm Exp}_{x_i}^{-1}(x_{i+1}) \in T_{x_k}\M, i = 0,...,k$
  \FOR{$j = 1,...,k$}
  \STATE Solve $c^*(\lambda_j) = \argmin_{c \in \sR^{k+1} : c^\top 1 = 1} \| \sum_{i=0}^k c_i r_i \|^2_{x_k} + \lambda_j  \| c \|_2^2$.
  \STATE Compute $\bar{x}(\lambda_j) = \bar{x}_{c,x}$ using $c^*(\lambda_j)$. 
  \ENDFOR
  \STATE Set $\bar{x}^* = \argmin_{j=1,...,k} f(\bar{x}(\lambda_j))$.
  \STATE Compute $u = \Exp^{-1}_{x_0}(\bar{x}^*)$ and set $ t = 1$.
  \WHILE{$f(\Exp_{x_0}( 2t u )) < f(\Exp_{x_0}(tu))$}
  \STATE Update $t = 2t$.
  \ENDWHILE
  \STATE \textbf{Output:} $\Exp_{x_0}(t u)$. 
 \end{algorithmic} 
\end{algorithm}

\section{From Euclidean averaging to Riemannian averaging}
\label{euclidean_average_appendix}

To extend the idea of weighted average to manifolds, we first rewrite the weighted average on the Euclidean space as follows.

\begin{lemma}[Weighted average recursion]
\label{waverage_recursion_lemma}
Given a set of coefficients $\{ c_i \}_{i=0}^k$ with $\sum_{i=0}^k c_i = 1$ and a set of iterates $\{x_i \}_{i=0}^k$. Let the streaming weighted average be defined as $\tilde{x}_{i} = \tilde{x}_{i-1} + \gamma_i  (x_i - \tilde{x}_{i-1})$ where $\gamma_i = \frac{c_i}{\sum_{j=0}^i c_j}$ for $i = 0, ..., k$ and $\tilde{x}_{-1} = x_0$. Then $\tilde{x}_{k} = \sum_{i=0}^k c_i x_i$. 
\end{lemma}
\begin{proof}
For some $\gamma_1, ..., \gamma_k$, the streaming weighted average is defined as $\tilde{x}_i = \tilde{x}_{i-1} + \gamma_i (x_i - \tilde{x}_{i-1})$ for $i \in [k]$.
We first show the streaming weighted average has the form 
\begin{equation*}
    \tilde{x}_i = \prod_{j=1}^i (1 - \gamma_j) x_0 + \gamma_1 \prod_{j=2}^i (1 - \gamma_j) x_1 + \cdots +\gamma_i x_i, \quad \forall i \in [k].
\end{equation*}
We prove such argument by induction. For $i = 1$, it is clear that $\tilde{x}_1 = (1-\gamma_1) x_0 + \gamma_1 x_1$ and satisfies the form. Suppose at $i = k'$, the equality is satisfied, then for $i = k' + 1$, we have 
\begin{equation*}
    \tilde{x}_{k'+1} = \tilde{x}_{k'} + \gamma_{k'+1} (x_{k'+1} - \tilde{x}_{k'}) = (1 - \gamma_{k' +1}) \tilde{x}_{k'} + \gamma_{k'+1} x_{k'+1}
\end{equation*}
which satisfies the equality. Hence this argument holds for all $i \in [k]$. Finally, at $i = k$, we see that the choice that $\gamma_i = \frac{c_i}{\sum_{j=0}^i c_i}$ leads to the matching coefficients.
\end{proof}

\section{Geometry of specific Riemannian manifolds}
\label{manifolds_geometry_appendix}

\paragraph{Sphere manifold.}
The sphere manifold $\gS^{d-1}$ is an embedded submanifold of $\sR^d$ with the tangent space identified as $T_x\gS^{d-1} = \{ u \in \sR^d : x^\top u = 0\}$. The Riemannian metric is given by $\langle u, v\rangle = \langle u,v \rangle_2$ for $u,v \in T_x\gS^{d-1}$. We use the exponential map derived as $\Exp_x(u) = \cos(\| u \|_2) x + \sin(\| u\|_2) \frac{u}{\| u \|_2}$ and the inverse exponential map as $\Exp_{x}^{-1}(y) = \arccos(x^\top y) \frac{{\rm Proj}_{x}(y-x)}{\|{\rm Proj}_{x}(y-x) \|_2}$ where ${\rm Proj}_{x}(v) = v - (x^\top v)x$ is the orthogonal projection of any $v \in \sR^d$ to the tangent space $T_x\gS^{d-1}$. The vector transport is given by the projection operation, i.e., $\gT_x^y u = {\rm Proj}_y(u)$.

\paragraph{Symmetric positive definite (SPD) manifold.}
The SPD manifold of dimension $d$ is denoted as $\sS_{++}^d \coloneqq \{ X \in \sR^{d \times d} : X^\top = X, X \succ 0\}$. The tangent space $T_X\M$ is the set of symmetric matrices. The affine-invariant Riemannian metric is given by $\langle U, V \rangle_X = \trace(X^{-1} U X^{-1} V)$ for any $U, V \in T_X\sS_{++}^d$. We make use of the exponential map, which is $\Exp_X(U) = X \expm(X^{-1} U)$ where $\expm(\cdot)$ is the matrix exponential. The inverse exponential map is derived as $\Exp^{-1}_X(Y) = X \logm(X^{-1} Y)$ for any $X, Y \in \sS_{++}^d$. We consider the parallel transport given by $\Gamma_X^Y U = E U E^\top$ with $E = (Y X^{-1})^{1/2}$.

\paragraph{Stiefel manifold.}
The Stiefel manifold of dimension $p \times r$ is written as ${\rm St}(p,r) \coloneqq \{ X \in \sR^{p \times r} : X^\top X = I \}$. The Riemannian metric is the Euclidean inner product defined as $\langle U, V\rangle_X = \langle U, V \rangle_2$. We consider the QR-based retraction $\Retr_X(U) = {\rm qf}(X+U)$ where ${\rm qf}(\cdot)$ returns the Q-factor from the QR decomposition. The inverse retraction is derived as for $X, Y \in \mathcal{O}(d)$ $\Retr_{X}^{-1}(Y) = Y R - X$, where $R$ is solved from the system $X^\top Y R + R^\top Y^\top X = 2 I$. The vector transport is given by the orthogonal projection, which is $\gT_{X}^Y = U - Y \{Y^\top U\}_{\rm S}$ where $\{A  \}_{\rm S} \coloneqq (A + A^\top)/2$.

\paragraph{Grassmann manifold.}
The Grassmann manifold of dimension $p \times r$, denoted as ${\rm Gr}(p, r)$, is the set of all $r$ dimensional subspaces in $\sR^p$ ($p \geq r$). Each point on the Grassmann manifold can be identified as a column orthonormal matrices $X \in \sR^{p \times r}, X^\top X = I$ and two points $X, Y \in {\rm Gr}(p, r)$ are equivalent if $X = Y O$ for some $O \in \mathcal{O}(r)$, the $r \times r$ orthogonal matrix. Hence Grassmann manifold is a quotient manifold of the Stiefel manifold. We consider the popular QR-based retraction, i.e. $R_{X}(U) = {\rm qf}(X+U)$ where for simplicity, we let $X$ to represent the equivalence class and $U$ represents the horizontal lift of the tangent vector. The inverse retraction is also based on QR factorization, i.e. $R^{-1}_X(Y) = Y (X^\top Y )^{-1} - X$. Vector transport is $\gT_{X}^Y U = U -  X X^\top U$.

\section{Riemannian Nesterov accelerated gradient methods}
\label{appendix_rnag}

Here we include the practical algorithms of Riemannian Nesterov accelerated gradient methods presented in \cite{zhang2018estimate,kim2022accelerated}. It is worth noticing that all the algorithms are analyzed under the exponential map and inverse exponential map and parallel transport. In contrast, the proposed RiemNA applies to general retraction and vector transport.

We first present the (constant-stepsize) RAGD method in \cite[Algorithm 2]{zhang2018estimate}, which is included in Algorithm \ref{ragd_algorithm}. We see the algorithm requires three times evaluation of the exponential map and two times the inverse exponential map.

\begin{algorithm}[!h]
 \caption{RAGD \cite{zhang2018estimate}}
 \label{ragd_algorithm}
 \begin{algorithmic}[1]
  \STATE \textbf{Input:} Initialization $x_0$, parameter $\beta > 0$, stepsize $h \leq \frac{1}{L}$, strong convexity parameter $\mu > 0$. 
  \STATE Initialize $v_0 = x_0$.
  \STATE Set $\alpha = \frac{\sqrt{\beta^2 + 4 (1 + \beta) \mu h} - \beta}{2}, \gamma = \frac{\sqrt{\beta^2 + 4 (1 + \beta) \mu h} - \beta}{\sqrt{\beta^2 + 4 (1 + \beta) \mu h} + \beta} \, \mu$, $\bar{\gamma} = (1 + \beta)\gamma$.
  \FOR{$k = 0, ..., K-1$}
    \STATE Compute $\alpha_k \in (0,1)$ from the equation $\alpha^2_k = h_k ((1- \alpha_k) \gamma_k + \alpha_k \mu)$.
    \STATE $y_k = \Exp_{x_k} \big( \frac{\alpha \gamma}{\gamma + \alpha \mu} \Exp^{-1}_{x_k}(v_k) \big)$
    \STATE $x_{k+1} = \Exp_{y_k}(- h \, \grad f(y_k))$ 
    \STATE $v_{k+1} = \Exp_{y_k} \big( \frac{(1-\alpha)\gamma}{\bar{\gamma}} \Exp^{-1}_{y_k}(v_k) - \frac{\alpha}{\bar{\gamma}} \grad f(y_k) \big)$
  \ENDFOR
  \STATE \textbf{Output:} $x_K$
 \end{algorithmic} 
\end{algorithm}

Below we present RNAG-C (Algorithm \ref{rnag_c_algorithm}), which is designed for geodesic convex functions and RNAG-SC (Algorithm \ref{rnag_sc_algorithm}) which is for geodesic strongly convex functions in \cite{kim2022accelerated}. We observe the algorithms require two times evaluation of the exponential map, inverse exponential map as well as parallel transport.

\begin{algorithm}[!th]
 \caption{RNAG-C \cite{kim2022accelerated}}
 \label{rnag_c_algorithm}
 \begin{algorithmic}[1]
  \STATE \textbf{Input:} Initialization $x_0$, parameters $\xi, T > 0$, stepsize $s \leq \frac{1}{L}$.
  \STATE Initialize $\bar{v}_0 = 0 \in T_{x_0}\M$.
  \STATE Set $\lambda_k = \frac{k + 2 \xi + T}{2}$.
  \FOR{$k = 0, ..., K-1$}
    \STATE $y_k = \Exp_{x_k} \big( \frac{\xi}{\lambda_k + \xi - 1} \bar{v}_k \big)$
    \STATE $x_{k+1} = \Exp_{y_k}(- s \, \grad f(y_k))$ 
    \STATE $v_k = \Gamma_{x_k}^{y_k} \big( \bar{v}_k - \Exp^{-1}_{x_k}(y_k) \big)$ 
    \STATE $\bar{\bar{v}}_{k+1} = v_k - \frac{s \lambda_k}{\xi} \grad f(y_k)$
    \STATE $\bar{v}_{k+1} = \Gamma_{y_k}^{x_{k+1}} \big( \bar{\bar{v}}_{k+1} - \Exp^{-1}_{y_k}(x_{k+1}) \big)$
  \ENDFOR
  \STATE \textbf{Output:} $x_K$
 \end{algorithmic} 
\end{algorithm}

\begin{algorithm}[!th]
 \caption{RNAG-SC \cite{kim2022accelerated}}
 \label{rnag_sc_algorithm}
 \begin{algorithmic}[1]
  \STATE \textbf{Input:} Initialization $x_0$, parameter $\xi$, stepsize $s \leq \frac{1}{L}$, strong convexity parameter $\mu$.
  \STATE Set $q = \mu s$.
  \FOR{$k = 0, ..., K-1$}
    \STATE $y_k = \Exp_{x_k} \big( \frac{\sqrt{\xi q}}{1 + \sqrt{\xi q}} \bar{v}_k\big)$ 
    \STATE $x_{k+1} = \Exp_{y_k} \big( - s \, \grad f(y_k))$
    \STATE $v_k = \Gamma_{x_k}^{y_k} \big( \bar{v}_k - \Exp^{-1}_{x_k}(y_k) \big)$
    \STATE $\bar{\bar{v}}_{k+1} = \big( 1 - \sqrt{\frac{q}{\xi}} \big) v_k + \sqrt{\frac{q}{\xi}} \big( - \frac{1}{\mu} \grad f(y_k)\big)$
    \STATE $\bar{v}_{k+1} = \Gamma_{y_k}^{x_{k+1}} \big( \bar{\bar{v}}_{k+1} - \Exp^{-1}_{y_k}(x_{k+1}) \big)$
  \ENDFOR
  \STATE \textbf{Output:} $x_K$
 \end{algorithmic} 
\end{algorithm}

\section{Function classes on Riemannian manifolds}
\label{function_class_appendix}
This section briefly reviews the various functions classes on Riemannian manifolds. 

\subsection{Geodesic gradient Lipschitzness and Hessian Lipschitzness}
\label{appendix_geodesic_grad_hess}
First we provide several equivalent characterizations for the gradient and Hessian Lipschitzness. For proof and more detailed discussions, see \cite[Section 10.4]{boumal2020introduction}.
\begin{lemma}[Geodesic gradient Lipschitzness and function smoothness]
\label{gradient_lipschitz_prop}
A function $f: \M \xrightarrow{} \sR$ has geodesic $L$-Lipschitz gradient in $\gX \subseteq$ if for all $x, y = \Exp_x(u) \in \gX$ in the domain of the exponential map, we have 
\begin{equation*}
    \| \Gamma_{\gamma(t)}^x \grad f(\gamma(t)) - \grad f(x) \|_{x} \leq L \| tu \|_x,
\end{equation*}
for all $t \in [0,1]$ and $\gamma(t) \coloneqq \Exp_x(tu)$. This is equivalent to function having bounded Hessian as $\| \hess f(x) \|_x \coloneqq \max_{u\in T_x\M : \| u \|_x = 1} \| \hess f(x) [u]\|_x \leq L$, where $\| \hess f(x) \|_x$ denotes the operator norm of Riemannian Hessian. If function $f$ has geodesic $L$-Lipschitz gradient, then function $f$ is geodesic $L$-smooth, which satisfies
\begin{equation*}
    |f(y) - f(x) - \langle \grad f(x), u\rangle_{x}| \leq \frac{L}{2} \| u\|^2_x. 
\end{equation*}
\end{lemma}

\begin{lemma}[Geodesic Hessian Lipschitzness]
\label{hessian_lipschitz_prop}
A function $f$ has geodesic $\rho$-Lipschitz Hessian in $\gX \subseteq \M$ if for all $x, y = \Exp_x(u) \in \gX$ in the domain of the exponential map, we have
\begin{equation*}
    \| \Gamma_{\gamma(t)}^x \circ \hess f(\gamma(t)) \circ \Gamma_x^{\gamma(t)} - \hess f(x) \|_x \leq \rho \| t u\|^3_x,
\end{equation*}
for all $t \in [0,1]$ and $\gamma(t) \coloneqq \Exp_x(tu)$. If function $f$ has geodesic $\rho$-Lipschitz Hessian, then function $f$ satisfies 
\begin{align*}
    &|f(y) - f(x) - \langle \grad f(x), u \rangle_{x} - \frac{1}{2} \langle u, \hess f(x)[u]\rangle_{x} | \leq \frac{\rho}{6} \| u \|_x^3 \\
    &\| \Gamma_y^x \grad f(y) - \grad f(x) - \hess f(x)[u] \|_x \leq \frac{\rho}{2} \| u \|^2.
\end{align*}
\end{lemma}

\subsection{Retraction gradient Lipschitzness and Hessian Lipschitzness}
\label{appendix_retr_lipschitz}

In this section, we define the gradient and Hessian Lipschitzness with respect to a retraction, which generalizes the definitions in Section \ref{appendix_geodesic_grad_hess}.

\begin{definition}[Retraction gradient Lipschitzness]
A function $f: \M \xrightarrow{} \sR$ has retraction $L$-Lipschitz gradient in $\gX \subseteq \M$ if for all $x,y =R_{x}(u) \in \gX$, we have 
\begin{equation*}
    \| \Gamma_{c(t)}^x \grad f(c(t)) - \grad f(x)  \|_{x} \leq L \| t u \|_{x}
\end{equation*}
where we denote $c(t) = R_{x}(tu)$. 
\end{definition}

\begin{definition}[Retraction Hessian Lipschitzness]
A function $f: \M \xrightarrow{} \sR$ has retraction $\rho$-Lipschitz Hessian in $\gX \subseteq \M$ if for all $x, y \in R_{x}(u) \in \gX$ in the domain of the retraction, we have 
\begin{equation*}
    \| \Gamma_{c(t)}^x \circ \hess f(c(t)) \circ \Gamma_x^{c(t)} - \hess f(x)  \|_{x} \leq \rho \| t u \|_x^3,
\end{equation*}
where we denote $c(t) = R_{x}(tu)$. 
\end{definition}

\subsection{Geodesic convexity and strong convexity}

We start with the notion of \textit{geodesic convex set}. A subset $\gX \subseteq \M$ is called geodesic convex if for any two points in the set, there exists a geodesic joining them that lies entirely within the set. 

\begin{definition}[Geodesic (strong) convexity]
A function $f: \gX \xrightarrow{} \sR$ is geodesic convex in a geodesic convex set $\gX$ if for any geodesic $\gamma: [0,1] \xrightarrow{} \gX$, we have $f(\gamma(t)) \leq (1-t) f(x) + t f(y)$ where we let $x = \gamma(0), y = \gamma(1)$. The function is geodesic $\mu$-strongly convex if $(f \circ \gamma)''(t) \geq \mu d^2(x, y)$ for all $t \in [0,1]$. This is equivalent to $\hess f(x) \succeq \mu \,  \id$ for all $x \in \gX$.
\end{definition}

A similar notion of convexity with respect to retraction also exists by replacing the geodesic curve $\gamma(t)$ with retraction curve $c(t) = \Retr_x(tu)$. See \cite{huang2015broyden} for more details.

\section{Main proofs}

Before we proceed with the proofs of the results in the paper, we introduce a lemma that will be used often in the course of the proof. 

\begin{lemma}
\label{diff_bound_gammatransport}
Under Assumption \ref{assump_normal_nei}, for any $w, x, y, z \in \gX$, we have $\| \Gamma_{w}^x \Gamma_{y}^w \Exp^{-1}_y(z) - \big( \Exp_x^{-1}(z) - \Exp_{x}^{-1}(y)  \big) \|_x \leq C_0 d(y,w) d(w,x) d(y,z) + C_2 \min\{ d(y,z), d(x,y) \} C_\kappa\big( d(y,z) + d(x,y) \big)$.
\end{lemma}
\begin{proof}[Proof of Lemma \ref{diff_bound_gammatransport}]
\begin{align*}
    &\| \Gamma_{w}^x \Gamma_{y}^w \Exp^{-1}_y(z) - \big( \Exp_x^{-1}(z) - \Exp_{x}^{-1}(y) \big)  \|_x \\
    &\leq \| \Gamma_{w}^x \Gamma_{y}^w \Exp^{-1}_y(z) - \Gamma_{y}^x \Exp^{-1}_y(z) \|_x + \| \Gamma_{y}^x \Exp^{-1}_y(z) - \big( \Exp_x^{-1}(z) -\Exp_{x}^{-1}(y)  \big)  \|_x \\
    &\leq C_0 d(y,w) d(w,x) d(y,z) + C_2 d\Big( \Exp_x \big( \Gamma_y^x \Exp^{-1}_y(z) + \Exp^{-1}_x(y) \big), z \Big) \\
    &\leq C_0 d(y,w) d(w,x) d(y,z) + C_2 d\Big( \Exp_x \big( \Gamma_y^x \Exp^{-1}_y(z) + \Exp^{-1}_x(y) \big), \Exp_y(\Exp_y^{-1}(z)) \Big) \\
    &\leq C_0 d(y,w) d(w,x) d(y,z) + C_2 \min\{ d(y,z), d(x,y) \} C_\kappa\big( d(y,z) + d(x,y) \big). 
\end{align*}
where we apply Lemma \ref{lemma_move_vectorsum} and \ref{lemma_metric_distort}.
\end{proof}

\subsection{Proof of Proposition \ref{prop_cstar_derivation}}

\begin{proof}[Proof of Proposition \ref{prop_cstar_derivation}]
Let $\mu \in \sR$ be the dual variable. Then we have $c^*, \mu^*$ satisfy the KKT system:
\begin{align*}
    \begin{bmatrix}
        2 (R + \lambda I)  & 1 \\
        1^\top & 0
    \end{bmatrix}
    \begin{bmatrix}
        c^* \\ \mu^*
    \end{bmatrix}
    = 
    \begin{bmatrix}
        0 \\ 1
    \end{bmatrix}
\end{align*}
Solving the system yields the desired result. 
\end{proof}

\subsection{Proof of Lemma \ref{lin_iter_approx_lemma}}

\begin{proof}[Proof of Lemma \ref{lin_iter_approx_lemma}]
First, we consider the pushforward operator $\Exp_x^y: T_x\M \xrightarrow{} T_y\M$ for any $x, y \in \M$, defined as $\Exp_x^y(u) \coloneqq \Exp_y^{-1} (\Exp_x(u))$ for any $u \in T_x\M$. The differential of $\Exp_x^y$ at $0$ along $u \in T_x\M$ is derived as 
\begin{align*}
    \D \Exp_x^y(0)[u] = \D \Exp^{-1}_y (\Exp_x(0 )) [ \D \Exp_x(0) [u] ] = \D \Exp^{-1}_y (x) [u] &= [ \D \Exp_y (\Exp^{-1}_y(x)) ]^{-1} [u] \\
    &= ( T_y^x )^{-1}[u]
\end{align*}
where we denote $T_x^y (v) = \D \Exp_x (\Exp^{-1}_x(y))[v] \in T_y\M$ for $v \in T_x\M$. The second equality is due to $\Exp_x(0) = 0, \D\Exp_x(0) = \id$ and the third equality follows from the inverse function theorem. Then by Taylor's theorem for $\Exp_{x_i}^{x^*}$ around $0$, we have 
\begin{align}
    \Exp_{x_*}^{-1}(x_{i+1})  &= \Exp_{x_i}^{x^*}( \Exp^{-1}_{x_i}(x_{i+1}) ) \nonumber\\
    &= \Exp_{x_i}^{x^*}(0) + \D \Exp_{x_i}^{x^*}(0)[ \Exp^{-1}_{x_i}(x_{i+1}) ]  + \frac{1}{2} \D^2 \Exp_{x_i}^{x^*} (\zeta_i)[ \Exp^{-1}_{x_i}(x_{i+1}),\Exp^{-1}_{x_i}(x_{i+1}) ] \nonumber\\
    &= \Exp_{x^*}^{-1}(x_i) - \eta (T_{x^*}^{x_i})^{-1}[ \grad f(x_i) ] + \frac{\eta^2}{2} \D^2 \Exp_{x_i}^{x^*} (\zeta_i) [\grad f(x_i), \grad f(x_i)] \nonumber\\
    &= \Exp_{x^*}^{-1}(x_i) - \eta (T_{x^*}^{x_i})^{-1}[ \grad f(x_i) ] + \frac{\eta^2}{2} \epsilon_i \label{lin_recur_temp1}
\end{align}
for some $\zeta_i = s \Exp^{-1}_{x_i}(x_{i+1})$, $s \in (0,1)$. We let $\epsilon_i \coloneqq \D^2 \Exp_{x_i}^{x^*} (\zeta_i) [\grad f(x_i), \grad f(x_i)]$ with $\| \epsilon_i \|_{x^*} = O \big( \| \grad f(x_i) \|^2_{x_i} \big)$. Then by Hessian Lipschitzness (Lemma \ref{hessian_lipschitz_prop}), we have around $x^*$
\begin{equation}
    e_i \coloneqq \Gamma_{x_i}^{x^*} \grad f(x_i) - \hess f(x^*)[ \Exp^{-1}_{x^*}(x_i) ] \leq \frac{\rho}{2} \| \Exp^{-1}_{x^*}(x_i) \|^2_{x^*}. \label{hess_lip_ei}
\end{equation}
Combining \eqref{lin_recur_temp1} with \eqref{hess_lip_ei} yields
\begin{align}
    \Exp_{x_*}^{-1}(x_{i+1}) - \Exp_{x^*}^{-1}(x_i) &= - \eta (\Gamma_{x_i}^{x^*} T_{x^*}^{x_i})^{-1}[ \Gamma_{x_i}^{x^*} \grad f(x_i) ] + \frac{\eta^2}{2} \epsilon_i \nonumber\\
    &= - \eta (\Gamma_{x_i}^{x^*} T_{x^*}^{x_i})^{-1} [\hess f(x^*)[ \Exp^{-1}_{x^*}(x_i) ] + e_i] + \frac{\eta^2}{2} \epsilon_i. \label{temp_delta_diff}
\end{align}
To show the desired result, it remains to show the operator $(\Gamma_{x_i}^{x^*} T_{x^*}^{x_i})^{-1}$ is locally identity. This is verified in \cite[Lemma 6]{tripuraneni2018averaging} for general retraction. We restate here and adapt to the case of exponential map. 

Consider the function $H(u) \coloneqq (\Gamma_x^{\Exp_x(u)})^{-1} T_{x}^{\Exp_x(u)} : T_x\M \xrightarrow{} L(T_x\M)$, where $L(T_x\M)$ denotes the set of linear maps on $T_x\M$. Let $\gamma(t) = \Exp_x(tu)$. Then we have 
\begin{align*}
    \frac{d}{dt} H(tu) |_{t = 0} = \frac{d}{dt} (\Gamma_x^{\gamma(t)})^{-1} T_{x}^{\gamma(t)} |_{t=0} = \Big( (\Gamma_x^{\gamma(t)})^{-1} \frac{D}{dt} T_{x}^{\gamma(t)} \Big)|_{t=0} &= \big( \frac{D}{dt} \D \Exp_x(tu) \big)|_{t=0} \\
    &= \frac{D^2}{dt^2} \Exp_x(tu) |_{t=0} = 0.
\end{align*}
where the second equality is due to the property of parallel transport (see for example \cite[Proposition 10.37]{boumal2020introduction}). In addition, from \cite[Theorem A.2.9]{waldmann2012geometric}, we see the second order derivative of $H$ is given by $\frac{d^2}{dt^2} H(tu) |_{t=0} = \frac{1}{6} \Riem_x(u, \cdot) u$ where we denote $\Riem_x$ as the Riemann curvature tensor evaluated at $x$. We notice that $\Riem_x(u, \cdot)u: T_x\M \xrightarrow{} T_x\M$ is symmetric with respect to the Riemannian metric (see for example \cite{andrews2010ricci}).
 
For any $v \in T_x\M$, $H(u)[v] \in T_x\M$, we apply the Taylor's theorem for $H$ up to second order, which yields
\begin{equation*}
    H(u)[v] = v + \frac{1}{6} \Riem_x(u,v)u + O(\| u\|^3),
\end{equation*}
Let $x = x^*$ and $u = \Exp_{x^*}^{-1}(x_i) = \Delta_{x_i}$. Then we obtain for any $v \in T_{x^*}\M$, $H(u)[v] \in T_{x^*}\M$
\begin{equation}
    \Gamma_{x_i}^{x^*} T_{x^*}^{x_i}[v]  = v + \frac{1}{6} \Riem_{x^*} \big(\Delta_{x_i}, v \big)\Delta_{x_i} + O(\| \Delta_{x_i} \|^3). 
\end{equation}

It satisfies that $(\Gamma_{x_i}^{x^*} T_{x^*}^{x_i})^{-1} = \id - \frac{1}{6} \Riem_{x^*}\big( \Delta_{x_i}, \cdot \big) \Delta_{x_i} + O(\| \Delta_{x_i} \|^3)$. 
Substituting this result into \eqref{temp_delta_diff}, we obtain 
\begin{align*}
    &\Delta_{x_{i+1}} - \Delta_{x_{i}} \\
    &= - \eta \Big( \id - \frac{1}{6} \Riem_{x^*}\big( \Delta_{x_i}, \cdot \big) \Delta_{x_i} + O(\| \Delta_{x_i} \|^3) \Big) \big[\hess f(x^*)[ \Delta_{x_i} ] + e_i \big] + \frac{\eta^2}{2} \epsilon_i \\
    &= -\eta \, \hess f(x^*)[ \Delta_{x_i} ] - \eta e_i + \frac{\eta}{6} \Riem_{x^*} (\Delta_{x_i}, \hess f(x^*)[\Delta_{x_i}] + e_i) \Delta_{x_i}  + \frac{\eta^2}{2} \epsilon_i + O(\| \Delta_{x_i} \|^3).
\end{align*}
Let $\varepsilon_i = - \eta e_i +\frac{\eta}{6} \Riem_{x^*} (\Delta_{x_i}, \hess f(x^*)[\Delta_{x_i}] + e_i) \Delta_{x_i}  + \frac{\eta^2}{2} \epsilon_i + O(\| \Delta_{x_i} \|^3)$. We can bound the error term as follows.
\begin{equation*}
    \| \varepsilon_i \|^2_{x^*} = O( \| e_i\|^2_{x^*} +  \| \Delta_{x_i} \|_{x^*}^4 \| \grad f(x_i) \|^2_{x_i} + \| \epsilon_i \|^2_{x^*} + \| \Delta_{x_i} \|^6_{x^*} ) = O(\| \Delta_{x_i} \|^4),
\end{equation*}
where we use the bounds on $\| e_i\|_{x^*}, \| \epsilon_i\|_{x^*}$ as well as $\hess f(x^*)[\Delta_{x_i}] + e_i = \Gamma_{x_i}^{x^*} \grad f(x_i)$ and geodesic gradient Lipschitzness (Lemma \ref{gradient_lipschitz_prop}) such that $\| \grad f(x_i) \|^2 \leq L \| \Delta_i \|^2_{x^*}$. 
\end{proof}

\subsection{Proof of Lemma \ref{lemma_deviation_recur_tange_avera}}

\begin{proof}[Proof of Lemma \ref{lemma_deviation_recur_tange_avera}]
The proof is by induction. Let $\gamma_i = \frac{c_i}{\sum_{j=0}^i c_j}$ and first we rewrite the averaging on tangent space as following the recursion defined as $\widetilde{\Delta}_{x_i} = \widetilde{\Delta}_{x_{i-1}} + \gamma_i (\Delta_{x_i} - \widetilde{\Delta}_{x_{i-1}})$. As we have shown in Lemma \ref{waverage_recursion_lemma}, $\sum_{i=0}^k c_i \Delta_{x_i} = \widetilde{\Delta}_{x_k}$. To show the difference between $\Delta_{\bar{x}_{c,x}}$, based on Lemma \ref{lemma_metric_distort}, it suffices to show the distance between $\tilde{x}_k = \bar{x}_{c,x}$ and $\Exp_{x^*}(\widetilde{\Delta}_{x_k})$ is bounded. 

To this end, we first notice that $\tilde{x}_0 = x_0 = \Exp_{x^*}(\widetilde{\Delta}_{x_0})$ and we consider bounding the difference between $\tilde{x}_1$ and $\Exp_{x^*}(\widetilde{\Delta}_{x_1})$. To derive the bound, we first observe that by Lemma \ref{lemma_move_vectorsum},
\begin{align}
    &d \Big(\Exp_{x^*}(\widetilde{\Delta}_{x_1}), \Exp_{x_0} \big( \Gamma_{x^*}^{x_0} \gamma_1 (\Delta_{x_1} - \widetilde{\Delta}_{x_0}) \big) \Big) \nonumber\\
    &= d\Big(\Exp_{x^*}(\widetilde{\Delta}_{x_0} + \gamma_1 (\Delta_{x_1} - \widetilde{\Delta}_{x_0}) ), \Exp_{x_0} \big( \Gamma_{x^*}^{x_0} \gamma_1 (\Delta_{x_1} - \widetilde{\Delta}_{x_0}) \big) \Big) \nonumber\\
    &\leq d(x_0, x^*) C_\kappa \Big( \| \widetilde{\Delta}_{x_0} \|_{x^*} + \gamma_1 \| \Delta_{x_1} - \widetilde{\Delta}_{x_0} \|_{x^*} \Big), \label{temp_reprove_01}
\end{align}
where we see $x_0 = \Exp_{x^*}(\widetilde{\Delta}_{x_0})$ with $\widetilde{\Delta}_{x_0} = \Delta_{x_0}$. In addition,
\begin{align}
    d \Big( \tilde{x}_1 , \Exp_{x_0} \big( \Gamma_{x^*}^{x_0} \gamma_1 (\Delta_{x_1} - \widetilde{\Delta}_{x_0}) \big)\Big) &= d \Big( \Exp_{x_0}\big( \gamma_1 \Exp^{-1}_{x_0}(x_1) \big),  \Exp_{x_0} \big( \Gamma_{x^*}^{x_0} \gamma_1 (\Delta_{x_1} - \widetilde{\Delta}_{x_0}) \big) \Big) \nonumber\\
    &\leq \gamma_1 C_1 \| \Exp_{x_0}^{-1}(x_1) - \Gamma_{x^*}^{x_0} (\Delta_{x_1} - \Delta_{x_0})  \|_{x_0} \nonumber\\
    &\leq \gamma_1 C_1 C_2 d(x_0, x^*) C_\kappa \big( d(x_0, x_1) + d(x_0, x^*) \big). \label{temp_reprove_02}
\end{align}
where the last inequality is from the proof of Lemma \ref{diff_bound_gammatransport}. Thus combining \eqref{temp_reprove_01}, \eqref{temp_reprove_02} leads to
\begin{align*}
    &d \big(\tilde{x}_1, \Exp_{x^*}(\tilde{\Delta}_{x_1}) \big) \\
    &\quad \leq d(x_0, x^*) C_\kappa \Big( \| \widetilde{\Delta}_{x_0} \|_{x^*} + \gamma_1 \| \Delta_{x_1} - \widetilde{\Delta}_{x_0} \|_{x^*} \Big) + \gamma_1 C_1 C_2 d(x_0, x^*) C_\kappa \big( d(x_0, x_1) + d(x_0, x^*) \big).
\end{align*}
By noticing $C_\kappa(x) = O(x^2)$, we see $d(\tilde{x}_1, \Exp_{x^*}(\tilde{\Delta}_{x_1})) = O(d^3(x_0, x^*))$. 

Now suppose at $i \leq k-1$, we have $d(\tilde{x}_i, \Exp_{x^*}(\widetilde{\Delta}_{x_i})) = O(d^3(x_0, x^*))$ and we wish to show $d(\tilde{x}_{i+1}, \Exp_{x^*}(\widetilde{\Delta}_{x_{i+1}})) = O(d^3(x_0, x^*))$. To this end, we first see $\Exp_{x^*}(\widetilde{\Delta}_{x_{i+1}}) = \Exp_{x^*} \big(\widetilde{\Delta}_{x_i} + \gamma_{i+1} (\Delta_{x_{i+1}} - \widetilde{\Delta}_{x_i}) \big)$ and by Lemma \ref{lemma_move_vectorsum}
\begin{align*}
    &d \Big(\Exp_{x^*}\big(\widetilde{\Delta}_{x_i} + \gamma_{i+1} (\Delta_{x_{i+1}} - \widetilde{\Delta}_{x_i}) \big), \Exp_{\Exp_{x^*}(\widetilde{\Delta}_{x_i})} \big( \Gamma_{x^*}^{\Exp_{x^*}(\widetilde{\Delta}_{x_i})} \gamma_{i+1} (\Delta_{x_{i+1}} - \widetilde{\Delta}_{x_i}) \big) \Big) \nonumber\\
    &\leq \| \widetilde{\Delta}_{x_i} \|_{x^*} C_\kappa \big( \|  \widetilde{\Delta}_{x_i} \|_{x^*} + \gamma_{i+1} \| \Delta_{x_{i+1}} - \widetilde{\Delta}_{x_i} \|_{x^*} \big) \\
    &= O(d^3(x_0, x^*)),
\end{align*}
where the order of $O(d^3(x_0, x^*))$ is due to $C_\kappa(x) = O(x^2)$ and $\| \widetilde{\Delta}_{x_i} \|_{x^*} = O(d(x_0, x^*))$, which can be shown by induction.  

Further, noticing $\tilde{x}_{i+1} = \Exp_{\tilde{x}_i} \big( \gamma_{i+1} \Exp_{\tilde{x}_i}^{-1}(x_{i+1}) \big)$, we can show 
\begin{align}
    &d \Big(\tilde{x}_{i+1}, \Exp_{\Exp_{x^*}(\widetilde{\Delta}_{x_i})} \big( \Gamma_{x^*}^{\Exp_{x^*}(\widetilde{\Delta}_{x_i})} \gamma_{i+1} (\Delta_{x_{i+1}} - \widetilde{\Delta}_{x_i}) \big) \Big) \nonumber\\
    &\leq d \Big( \Exp_{\tilde{x}_i} \big( \gamma_{i+1} \Exp_{\tilde{x}_i}^{-1}(x_{i+1}) \big), \Exp_{\tilde{x}_i} \big( \Gamma_{x^*}^{\tilde{x}_i} \gamma_{i+1} (\Delta_{x_{i+1}} - \widetilde{\Delta}_{x_i} ) \big)  \Big) 
    \nonumber\\
    &\qquad + d \Big( \Exp_{\tilde{x}_i} \big( \Gamma_{x^*}^{\tilde{x}_i} \gamma_{i+1} (\Delta_{x_{i+1}} - \widetilde{\Delta}_{x_i} ) \big),  \Exp_{\Exp_{x^*}(\widetilde{\Delta}_{x_i})} \big( \Gamma_{x^*}^{\Exp_{x^*}(\widetilde{\Delta}_{x_i})} \gamma_{i+1} (\Delta_{x_{i+1}} - \widetilde{\Delta}_{x_i}) \big)   \Big). \label{temp_reprove_03}
\end{align}
The first term on the right of \eqref{temp_reprove_03} can be bounded as 
\begin{align}
    &d\Big( \Exp_{\tilde{x}_i} \big( \gamma_{i+1} \Exp_{\tilde{x}_i}^{-1}(x_{i+1}) \big), \Exp_{\tilde{x}_i} \big( \Gamma_{x^*}^{\tilde{x}_i} \gamma_{i+1} (\Delta_{x_{i+1}} - \widetilde{\Delta}_{x_i} ) \big)  \Big) \nonumber\\
    &\leq \gamma_{i+1} \| \Exp_{\tilde{x}_i}^{-1}(x_{i+1}) - \Gamma_{x^*}^{\tilde{x}_i} (\Delta_{x_{i+1}}  - \widetilde{\Delta}_{x_i})   \|_{\tilde{x}_i} \nonumber\\
    &= \gamma_{i+1} \| \Gamma_{\tilde{x}_i}^{x^*} \Exp_{\tilde{x}_i}^{-1}(x_{i+1}) - (\Delta_{x_{i+1}} - \Delta_{\tilde{x}_i} ) + (\widetilde{\Delta}_{x_i} - \Delta_{\tilde{x}_i} ) \|_{x^*} \nonumber\\
    &\leq \gamma_{i+1} \| \Gamma_{\tilde{x}_i}^{x^*} \Exp_{\tilde{x}_i}^{-1}(x_{i+1}) - (\Delta_{x_{i+1}} - \Delta_{\tilde{x}_i} ) \|_{x^*} + \gamma_{i+1}\| \widetilde{\Delta}_{x_i} - \Delta_{\tilde{x}_i}  \|_{x^*} \nonumber\\
    &\leq \gamma_{i+1} C_2 d(\tilde{x}_i, x^*) C_\kappa\big( d(x_{i+1}, \tilde{x}_i) + d(\tilde{x}_i, x^*) \big) + \gamma_{i+1} C_2 d(\tilde{x}_i, \Exp_{x^*}(\widetilde{\Delta}_{x_i})), \label{temp_reprove_04}
\end{align}
where we again use the result from the proof of Lemma \ref{diff_bound_gammatransport}. To see \eqref{temp_reprove_04} is on the order of $O(d^3(x_0, x^*))$, we only need to show $\| \Delta_{\tilde{x}_i} \|^2_{x^*} = d^2(\tilde{x}_i, x^*) =  O(d^2(x_0, x^*))$, which can be seen by a simple induction argument. First, it is clear that $\| \Delta_{\tilde{x}_0} \|^2_{x^*} = d^2(x_0, x^*)$. Then suppose for any $i < k$, we have $d(\tilde{x}_i, x^*) = O(d(x_0, x^*))$. Then from Lemma \ref{lemma_metric_distort}, we have
\begin{align*}
    d(\tilde{x}_{i+1}, x^*) &\leq C_1 \| \Exp^{-1}_{\tilde{x}_i}(\tilde{x}_{i+1}) - \Exp_{\tilde{x}_i}^{-1}(x^*)  \|_{\tilde{x}_i} \leq \frac{C_1 c_{i+1}}{\sum_{j=0}^{i+1} c_j} d( \tilde{x}_i, x_{i+1}) + d(\tilde{x}_i, x^*) \\
    &\leq \big(  \frac{C_1 c_{i+1}}{\sum_{j=0}^{i+1} c_j} + 1 \big) d(\tilde{x}_i, x^*) + d(x_{i+1}, x^*) = O(d(x_0, x^*)). 
\end{align*}
Thus, using $d(\tilde{x}_i, \Exp_{x^*}(\widetilde{\Delta}_{x_i})) = O(d^3(x_0, x^*))$, we see \eqref{temp_reprove_04} is on the order of $O(d^3(x_0, x^*))$. 

Now we bound the second term on the right of \eqref{temp_reprove_03}. Particularly, 
\begin{align}
    &d \Big( \Exp_{\tilde{x}_i} \big( \Gamma_{x^*}^{\tilde{x}_i} \gamma_{i+1} (\Delta_{x_{i+1}} - \widetilde{\Delta}_{x_i} ) \big),  \Exp_{\Exp_{x^*}(\widetilde{\Delta}_{x_i})} \big( \Gamma_{x^*}^{\Exp_{x^*}(\widetilde{\Delta}_{x_i})} \gamma_{i+1} (\Delta_{x_{i+1}} - \widetilde{\Delta}_{x_i}) \big)   \Big) \nonumber\\
    &\leq d \Big( \Exp_{\tilde{x}_i} \big( \Gamma_{x^*}^{\tilde{x}_i} \gamma_{i+1} (\Delta_{x_{i+1}} - \widetilde{\Delta}_{x_i} ) \big),  
    \Exp_{\tilde{x}_i} \big( \Gamma^{\tilde{x}_i}_{\Exp_{x^*}(\widetilde{\Delta}_{x_i})} \Gamma_{x^*}^{\Exp_{x^*}(\widetilde{\Delta}_{x_i})}  \gamma_{i+1}(\Delta_{x_{i+1}} - \widetilde{\Delta}_{x_i}) \big) 
    \Big) \nonumber\\
    &\qquad + d\Big( \Exp_{\tilde{x}_i} \big( \Gamma^{\tilde{x}_i}_{\Exp_{x^*}(\widetilde{\Delta}_{x_i})} \Gamma_{x^*}^{\Exp_{x^*}(\widetilde{\Delta}_{x_i})}  \gamma_{i+1}(\Delta_{x_{i+1}} - \widetilde{\Delta}_{x_i}), \Exp_{\Exp_{x^*}(\widetilde{\Delta}_{x_i})} \big( \Gamma_{x^*}^{\Exp_{x^*}(\widetilde{\Delta}_{x_i})} \gamma_{i+1} (\Delta_{x_{i+1}} - \widetilde{\Delta}_{x_i}) \big)  \Big) \nonumber\\
    &\leq  \gamma_{i+1}C_1 C_0 \| \widetilde{\Delta}_{x_i} \|_{x^*} d(\tilde{x}_i, \Exp_{x^*}\big(\widetilde{\Delta}_{x_i}) \big) \| \Delta_{x_{i+1}} - \widetilde{\Delta}_{x_i} \|_{x^*} + C_3 d(\tilde{x}_i, \Exp_{x^*}(\widetilde{\Delta}_{x_i})) \nonumber\\
    &= O(d^3(x_0, x^*)), \nonumber
\end{align}
where we apply Lemma \ref{lemma_metric_distort} multiple times. Combining the previous results, we see 
\begin{align*}
    &d(\tilde{x}_{i+1}, \Exp_{x^*}(\widetilde{\Delta}_{x_{i+1}})) \\
    &\leq  d \Big(\tilde{x}_{i+1}, \Exp_{\Exp_{x^*}(\widetilde{\Delta}_{x_i})} \big( \Gamma_{x^*}^{\Exp_{x^*}(\widetilde{\Delta}_{x_i})} \gamma_{i+1} (\Delta_{x_{i+1}} - \widetilde{\Delta}_{x_i}) \big) \Big) \\
    &\qquad + d \Big(\Exp_{x^*}\big(\widetilde{\Delta}_{x_i} + \gamma_{i+1} (\Delta_{x_{i+1}} - \widetilde{\Delta}_{x_i}) \big), \Exp_{\Exp_{x^*}(\widetilde{\Delta}_{x_i})} \big( \Gamma_{x^*}^{\Exp_{x^*}(\widetilde{\Delta}_{x_i})} \gamma_{i+1} (\Delta_{x_{i+1}} - \widetilde{\Delta}_{x_i}) \big) \Big) \\
    &= O(d^3(x_0, x^*))
\end{align*}
Now applying Lemma \ref{lemma_metric_distort}, we obtain 
\begin{equation*}
    \| \Delta_{\tilde{x}_{i+1}} - \widetilde{\Delta}_{x_{i+1}} \|_{x^*} \leq C_2 d(\tilde{x}_{i+1}, \Exp_{x^*}(\widetilde{\Delta}_{x_{i+1}})) = O(d^3(x_0, x^*))
\end{equation*}
for all $i \leq k-1$. Let $i = k-1$ we have $\| \Delta_{\tilde{x}_k} - \widetilde{\Delta}_{x_k} \|_{x^*} = \| \Delta_{\bar{x}_{c,x}} - \sum_{i=0}^k c_i \Delta_{x_i} \|_{x^*} = O(d^3(x_0, x^*))$. Thus the proof is complete.
\end{proof}

\subsection{Proof of Lemma \ref{error_lin_term}}

\begin{proof}[Proof of Lemma \ref{error_lin_term}]
Directly combining Lemma \ref{convergence_lin_iterate_theorem} and Lemma \ref{lemma_deviation_recur_tange_avera} gives the result. 
\end{proof}

\begin{lemma}[Convergence of the linearized iterates]
\label{convergence_lin_iterate_theorem}
Consider the linearized iterates $\{ \hat{x}_i \}_{i=0}^k$ satisfying \eqref{lin_iter_mfd} for some $G \succeq 0$ with $\|G \|_{x^*} \leq \sigma < 1$. Let $\hat{r}_i = \Delta_{\hat{x}_{i+1}} - \Delta_{\hat{x}_i}$, $\hat{c}^* = \argmin_{c^\top 1 = 1} \| \sum_{i=0}^k c_i \hat{r}_i \|^2_{x^*} + \lambda  \| c \|^2_2$. 
Then
\begin{equation*}
    \| \sum_{i = 0}^k \hat{c}^*_i \Delta_{\hat{x}_i} \|_{x^*} \leq \frac{d(x_0, x^*)}{1-\sigma} \sqrt{ (S^{[0,\sigma]}_{k,\bar{\lambda}})^2 - \frac{\lambda}{d^2(x_0, x^*)} \| \hat{c}^* \|_2^2  }
\end{equation*}
\end{lemma}
\begin{proof}[Proof of Lemma \ref{convergence_lin_iterate_theorem}]
The proof follows from \cite[Proposition 3.4]{scieur2020regularized} and we include it here for completeness. Denote $\gP_k^1 \coloneqq \{ p \in \sR[x] : \deg(p) = k, p(1) = 1 \}$ as the set of polynomials of degree $k$ with coefficients summing to $1$. Noticing that $\hat{r}_i = \Delta_{\hat{x}_{i+1}} - \Delta_{\hat{x}_i} = \big( G - \id \big) [\Delta_{\hat{x}_i}] = \big( G - \id \big) G^i [\Delta_{x_0}]$, we have 
$\| \sum_{i=0}^k c_i \hat{r}_i \|^2_{x^*} = \| (G - \id) p(G) [\Delta_{x_0}]\|^2_{x^*}$ where $p \in \gP_k^1$ and $\{c_i \}_{i=0}^k$ are the corresponding coefficients.
Then we obtain
\begin{align*}
    \min_{p \in \gP_k^1} \Big\{ \| (G - \id) p(G) [\Delta_{x_0}] \|_{x^*}^2 + \lambda   \| c \|^2_2 \Big\} &\leq  d^2(x_0, x^*) \min_{p \in \gP_k^1} \Big\{ \| p(G) \|^2_{x^*} + \frac{\lambda}{d^2(x_0, x^*)} \| p \|^2_2 \Big\} \\
    &\leq d^2(x_0, x^*) \min_{p \in \gP_k^1} \max_{M: 0 \preceq M \preceq \sigma\id} \Big\{ \| p(M) \|^2_{x^*} + \frac{\lambda}{d^2(x_0, x^*)}  \| p\|^2_2\Big\} \\
    &= d^2(x_0, x^*) \min_{p \in \gP_k^1} \max_{x \in [0,\sigma]} \Big\{  p^2(x)  + \frac{\lambda}{d^2(x_0, x^*)} \| p\|^2_2\Big\} \\
    &= (S^{[0,\sigma]}_{k,\bar{\lambda}})^2 d^2(x_0, x^*),
\end{align*}
where $\bar{\lambda} = \lambda/d^2(x_0, x^*)$ and we use the fact that $\| G - \id \|_{x^*} \leq 1$. Then
\begin{align*}
    \| \sum_{i=0}^k \hat{c}^*_i \Delta_{\hat{x}_i}  \|^2_{x^*} &= \| \sum_{i=0}^k \hat{c}_i^* (G -\id)^{-1} \hat{r}_i \|^2_{x^*} \\
    &\leq \| (G - \id )^{-1}\|_{x^*}^2 \Big( \| \sum_{i=0}^k\hat{c}^*_i \hat{r}_i \|^2_{x^*} + \lambda \| \hat{c}^* \|^2_2 - \lambda  \| \hat{c}^* \|^2_2\Big) \\
    &\leq \frac{d^2(x_0, x^*)}{(1-\sigma)^2} \Big( (S^{[0,\sigma]}_{k,\bar{\lambda}})^2 - \frac{\lambda}{d^2(x_0, x^*)} \| \hat{c}^* \|_2^2  \Big),
\end{align*}
where we see that $\| (G - \id )^{-1}\|_{x^*} \leq \frac{1}{1-\sigma}$. 
\end{proof}

\subsection{Proof of Lemma \ref{coeff_bound}}

\begin{proof}[Proof of Lemma \ref{coeff_bound}]
From Proposition \ref{prop_cstar_derivation} and following \cite[Proposition 3.2]{scieur2020regularized}, we obtain
\begin{align*}
    \| c^* \| \leq \sqrt{\frac{\| R\|_2 + \lambda  }{(k+1) \lambda } }.
\end{align*}
Now we bound $\| R\|_2$. First we see $R$ can be rewritten as $\mathcal{R}^\top \mathcal{G}_{x_k} \mathcal{R}$, where $\mathcal{G}_{x_k} \in \sR^{r \times r}$ is the positive definite metric tensor at $x_k$ and $\mathcal{R} = [\vec{r}_i] \in \sR^{r \times k}$ is the collection of tangent vector in an orthonormal basis and $r$ is the intrinsic dimension of the manifold. Thus we can write Riemannian inner product as $\langle r_i, r_j \rangle_{x_k} = \vec{r}_i^\top \gG_{x_k} \vec{r}_j$ and
\begin{align*}
    \| R\|_2  = \| \gG_{x_k}^{1/2} \gR \|_2^2 \leq \| \gG_{x_k}^{1/2} R\|_{\rm F}^2 = \sum_{i=0}^k \vec{r}_i^\top \gG_{x_k} \vec{r}_i  = \sum_{i=0}^k \|r_i \|_{x_k}^2 = \sum_{i=0}^k d^2(x_i, x_{i+1}).
\end{align*}
On the other hand, denote the perturbation matrix $P = R - \hat{R}$. Then from Proposition \ref{prop_cstar_derivation} and following \cite[Proposition 3.2]{scieur2020regularized}, we have
\begin{align*}
    \| \delta^c \|_2 \leq \frac{\| P \|_2}{\lambda } \| \hat{c}^*\|_2. 
\end{align*}
Now we need to bound $\| P \|_2$. Let $\gE_i = \Delta_{x_i} - \Delta_{\hat{x}_i}$. Then we have 
\begin{align}
    \| \Gamma_{x_k}^{x^*} r_i - \hat{r}_i \|_{x^*} 
    &= \| \Gamma_{x_k}^{x^*} r_i - (\Delta_{x_{i+1}} - \Delta_{x_i}) + (\Delta_{x_{i+1}} - \Delta_{x_i}) - \hat{r}_i \|_{x^*} \nonumber\\
    &\leq \| \Gamma_{x_k}^{x^*} r_i - (\Delta_{x_{i+1}} - \Delta_{x_i}) \|_{x^*} + \| (\Delta_{x_{i+1}} - \Delta_{x_i}) - \hat{r}_i \|_{x^*} \nonumber\\
    &= \| \Gamma_{x_k}^{x^*} r_i - (\Delta_{x_{i+1}} - \Delta_{x_i}) \|_{x^*} + \| \gE_{i+1} - \gE_{i} \|_{x^*} \label{temp_lemma10_01}
\end{align}
where we use Lemma \ref{lemma_metric_distort}. Now we respectively bound each of the two terms on the right. First we see from Lemma \ref{diff_bound_gammatransport},
\begin{align}
    \| \Gamma_{x_k}^{x^*} r_i - (\Delta_{x_{i+1}} - \Delta_{x_i}) \|_{x^*} \leq C_0 d(x_i, x_k) d(x_k, x^*) d(x_i, x_{i+1}) + C_2 d(x_i, x^*) C_\kappa\big( d(x_i, x^*) + d(x_i, x_{i+1}) \big) \label{temp_lemma10_03}
\end{align}

Further, we bound $\| \gE_{i+1} - \gE_i \|_{x^*}$. From Lemma \ref{lin_iter_approx_lemma}, we have $\gE_i = G [\gE_{i-1}] + \varepsilon_i, \gE_0 = 0$ and
\begin{align}
    \| \gE_{i+1} - \gE_{i} \|_{x^*} = \| (G - \id) \gE_i + \varepsilon_{i+1} \|_{x^*} &=  \| (G - \id) \sum_{j=1}^{i} G^{i - j} \varepsilon_j + \varepsilon_{i+1} \|_{x^*} \leq \sum_{j=1}^{i+1} \| \varepsilon_j \|_{x^*}. \label{temp_lemma10_04}
\end{align}
Combining \eqref{temp_lemma10_04}, \eqref{temp_lemma10_03}, \eqref{temp_lemma10_01} leads to
\begin{equation*}
    \| \Gamma_{x_k}^{x^*} r_i - \hat{r}_i \|_{x^*} \leq C_0 d(x_i, x_k) d(x_k, x^*) d(x_i, x_{i+1}) + C_2 d(x_i, x^*) C_\kappa\big( d(x_i, x^*) + d(x_i, x_{i+1}) \big) + \sum_{j=1}^{i+1} \| \varepsilon_j \|_{x^*}.
\end{equation*}
Finally, recall we can write $R = \gR^\top \gG_{x_k} \gR$ and similarly for $\hat{R} = \hat{\gR}^\top \gG_{x^*} \hat{\gR}$ where $\hat{\gR} = [\vec{\hat{r}}_i]$. By isometry of parallel transport, we have $R = \gR_{x^*}^\top \gG_{x^*} \gR_{x^*}$ where $\gR_{x^*} = [\overrightarrow{\Gamma_{x_k}^{x^*} r}_i]$. Let $E = \gG_{x^*}^{1/2} (\gR_{x^*} - \hat{\gR})$. Then 
\begin{align*}
    \| P \|_2 = \|\gR_{x^*}^\top \gG_{x^*} \gR_{x^*} -  \hat{\gR}^\top \gG_{x^*} \hat{\gR}\|_2  \leq 2 \| E \|_2 \| \gG_{x^*}^{1/2} \hat{R}\|_2 + \| E \|_2^2.
\end{align*}
Notice that 
\begin{equation*}
    \| \gG_{x^*}^{1/2} \hat{R} \|_2 \leq \| \gG_{x^*}^{1/2} \hat{R} \|_{\rm F} \leq \sum_{i=0}^k \| \hat{r}_i \|_{x^*} \leq \sum_{i=0}^k \| (G- \id) G^i \hat{r}_0 \|_{x^*} \leq \sum_{i=0}^k \sigma^i \| \hat{r}_0 \|_{x^*} \leq \frac{1-\sigma^{k+1}}{1-\sigma} d(x_0, x^*),
\end{equation*}
Also 
\begin{align*}
    \| E \|_2 &= \| \gG_{x^*}^{1/2} (\gR_{x^*} - \hat{\gR}) \|_2 \leq \sum_{i=0}^k \| \Gamma_{x_k}^{x^*} r_i - \hat{r}_i \|_{x^*} \\
    &\leq d(x_k, x^*) C_0 \sum_{i=0}^k  d(x_i, x_k) d(x_i, x_{i+1}) + C_2 \sum_{i=0}^k d(x_i, x^*) C_\kappa \big( d(x_i, x^*) + d(x_i, x_{i+1}) \big) + \sum_{i=0}^k \sum_{j = 1}^{i+1} \| \varepsilon_j \|_{x^*} \\
    &= O(d^2(x_0, x^*)),
\end{align*}
where we notice that $C_\kappa(d(x_i, x^*) + d(x_i, x_{i+1})) = O(d^2(x_i, x^*))$ and recall that $\| \varepsilon_j \|_{x^*} = O(d^2(x_j, x^*) ) = O(d^2(x_0, x^*))$. Thus $\| P \|_2 \leq 2\psi \frac{1-\sigma^{k+1}}{1-\sigma} d(x_0, x^*) + (\psi)^2$ where $\psi = O(d^2(x_0, x^*))$.
\end{proof}

\subsection{Proof of Lemma \ref{error_coeff_lemma}}

\begin{proof}[Proof of Lemma \ref{error_coeff_lemma}]
From Lemma \ref{lemma_metric_distort}, we first observe that $d(\bar{x}_{\hat{c}^*, \hat{x}}, \bar{x}_{c^*, \hat{x}}) \leq C_1 \|  \Delta_{\bar{x}_{\hat{c}^*, \hat{x}}} - \Delta_{\bar{x}_{c^*, \hat{x}}} \|_{x^*}.$
Now we derive a bound on the term $\| \Delta_{\bar{x}_{\hat{c}^*, \hat{x}}} - \Delta_{\bar{x}_{c^*, \hat{x}}}\|_{x^*}$. Notice that from Lemma \ref{lemma_deviation_recur_tange_avera}, we have
\begin{align*}
   \| \Delta_{\bar{x}_{\hat{c}^*, \hat{x}}} - \Delta_{\bar{x}_{c^*, \hat{x}}}\|_{x^*} &= \| \sum_{i=0}^k (\hat{c}^*_i - c^*_i) \Delta_{\hat{x}_i} + \hat{\epsilon} \|_{x^*} \leq  \| \delta^c \|_2 \big( \sum_{i=0}^k \| \Delta_{\hat{x}_i} \|^2_{x^*} \big)^{1/2} + \| \hat{\epsilon} \|_{x^*} \\
   &\leq \| \delta^c \|_2 (\sum_{i=0}^k \| \Delta_{\hat{x}_i} \|_{x^*}) + \| \hat{\epsilon} \|_{x^*} \\
   &\leq \| \delta^c \|_2 (\sum_{i=0}^k \| G\|^i \| \Delta_{{x}_0} \|_{x^*}) + \| \hat{\epsilon} \|_{x^*} \\
   &\leq \frac{1-\sigma^{k+1}}{1-\sigma}  d(x_0, x^*)  \| \delta^c \|_2  + \| \hat{\epsilon} \|_{x^*} \\
   &\leq \frac{1}{1-\sigma} \frac{d(x_0, x^*)}{\lambda } \Big(  \frac{1}{1-\sigma} 2\psi d(x_0, x^*) + (\psi)^2 \Big) \| \hat{c}^*\|_2 + \| \hat{\epsilon} \|_{x^*}
\end{align*}
for some $\|\hat{\epsilon} \|_{x^*} = O(d^3(x_0, x^*))$ and we denote $\delta^c = c^* - \hat{c}^*$. The bound on $\| \delta^c\|_2$ is from Lemma \ref{coeff_bound}. 
\end{proof}

\subsection{Proof of Lemma \ref{error_nonlinear}}

\begin{proof}[Proof of Lemma \ref{error_nonlinear}]
Similarly to Lemma \ref{error_coeff_lemma}, we first see $d(\bar{x}_{c^*, \hat{x}}, \bar{x}_{c^*, x}) \leq C_1 \| \Delta_{\bar{x}_{c^*, \hat{x}}} - \Delta_{\bar{x}_{c^*, x}} \|_{x^*}$ due to Lemma \ref{lemma_metric_distort}. Again using Lemma \ref{lemma_deviation_recur_tange_avera}, we see 
\begin{align*}
    \| \Delta_{\bar{x}_{c^*, \hat{x}}} - \Delta_{\bar{x}_{c^*, x}} \|_{x^*} = \| \sum_{i=0}^k c^*_i (\Delta_{x_i} - \Delta_{\hat{x}_i}) + \hat{\epsilon}  \|_{x^*} &\leq \| c^*\|_2 (\sum_{i=0}^k \| \gE_i \|^2_{x^*})^{1/2}  + \| \hat{\epsilon} \|_{x^*} \\
    &\leq \| c^*\|_2 (\sum_{i=0}^k \| \gE_i \|_{x^*}) + \| \hat{\epsilon} \|_{x^*} \textbf{}
\end{align*}
where $\|\hat{\epsilon} \|_{x^*} = O(d^3(x_0, x^*))$ and $\gE_i = \Delta_{x_i} - \Delta_{\hat{x}_i}$. From Lemma \ref{lin_iter_approx_lemma}, we have $\gE_i = G [\gE_{i-1}] + \varepsilon_i, \gE_0 = 0$. Thus we can bound 
\begin{equation*}
    \| \gE_i \|_{x^*} = \| \sum_{j=1}^{i} G^{i-j} \varepsilon_j   \|_{x^*} \leq \sum_{j=1}^{i} \|  \varepsilon_j   \|_{x^*}.
\end{equation*}
Then using Lemma \ref{coeff_bound} to bound $\| c^*\|_2$, we obtain 
\begin{equation*}
    \| \Delta_{\bar{x}_{c^*, \hat{x}}} - \Delta_{\bar{x}_{c^*, x}} \|_{x^*} \leq \sqrt{\frac{ \sum_{i=0}^k d^2(x_i, x_{i+1}) + \lambda  }{(k+1) \lambda } } \Big( \sum_{i=0}^k \sum_{j=0}^i \| \varepsilon_j \|_{x^*} \Big) + \epsilon_3,
\end{equation*}
where $\epsilon_3 = O(d^3(x_0, x^*))$. 
\end{proof}

\subsection{Proof of Theorem \ref{main_convergence_rna1_theorem}}

\begin{proof}[Proof of Theorem \ref{main_convergence_rna1_theorem}]
Following the decomposition of error, we show 
\begin{align*}
    &d(\bar{x}_{c^*, x}, x^*) \\
    &\leq d(\bar{x}_{\hat{c}^*, \hat{x}}, x^*) + d(\bar{x}_{\hat{c}^*,\hat{x}}, \bar{x}_{c^*, \hat{x}}) + d(\bar{x}_{c^*, \hat{x}}, \bar{x}_{c^*, x}) \\
    &\leq \frac{d(x_0, x^*)}{1-\sigma} \sqrt{ (S^{[0,\sigma]}_{k,\bar{\lambda}})^2 - \frac{\lambda}{d^2(x_0, x^*)} \| \hat{c}^* \|_2^2} +  \frac{C_1 d(x_0, x^*)}{\lambda (1-\sigma)} \Big(  \frac{2 d(x_0, x^*)}{1-\sigma}  \psi + (\psi)^2 \Big) \| \hat{c}^*\|_2 \\
    &\qquad + C_1\sqrt{\frac{ \sum_{i=0}^k d^2(x_i, x_{i+1}) + \lambda  }{(k+1) \lambda } } \Big( \sum_{i=0}^k \sum_{j=0}^i \| \varepsilon_j \|_{x^*} \Big)  + \epsilon_1 + \epsilon_2 +  \epsilon_3.
\end{align*}
Now we maximize the bound over $\| \hat{c}^*\|$. From \cite[Proposition A.1]{scieur2020regularized}, we see the maximum of a function $g(x) = c \sqrt{a - \bar{\lambda} x^2} + b x$ is $\sqrt{a} \sqrt{c^2 + \frac{b^2}{\bar{\lambda}}}$ where $\bar{\lambda} = \lambda/d^2(x_0, x^*)$. Let $a = (S^{[0,\sigma]}_{k,\bar{\lambda}})^2$, $b = \frac{C_1 d(x_0, x^*)}{\lambda (1-\sigma)} \Big(  \frac{2 d(x_0, x^*)}{1-\sigma}  \psi + (\psi)^2 \Big)$, $c = \frac{d(x_0, x^*)}{1-\sigma}$. We then obtain 
\begin{align*}
    d(\bar{x}_{c^*, x}, x^*) &\leq S^{[0,\sigma]}_{k,\bar{\lambda}} \sqrt{\frac{d^2(x_0, x^*)}{(1-\sigma)^2} + \frac{C_1^2 d^4(x_0, x^*) \Big(  \frac{2 d(x_0, x^*)}{1-\sigma}  \psi + (\psi)^2 \Big)^2}{\lambda^3 (1-\sigma)^2} } \\
    &\quad + C_1\sqrt{\frac{ \sum_{i=0}^k d^2(x_i, x_{i+1}) + \lambda  }{(k+1) \lambda } } \Big( \sum_{i=0}^k \sum_{j=0}^i \| \varepsilon_j \|_{x^*} \Big)  + \epsilon_1 +  \epsilon_2 +  \epsilon_3,
\end{align*}
which completes the proof. 
\end{proof}

\subsection{Proof of Proposition \ref{asymp_optimal_convergence_prop}}

\begin{proof}[Proof of Proposition \ref{asymp_optimal_convergence_prop}]
Dividing the bound from Theorem \ref{main_convergence_rna1_theorem} by $d(x_0, x^*)$ gives 
\begin{align*}
    \frac{d(\bar{x}_{c^*, x}, x^* )}{d(x_0, x^*)} &\leq  \frac{S^{[0,\sigma]}_{k,\bar{\lambda}}}{1-\sigma} \sqrt{1 + { O(d^{(2-3s)}(x_0, x^*) \Big(  \frac{2 d(x_0, x^*)}{1-\sigma}  \psi + (\psi)^2 \Big)^2} } \\
    &\quad + C_1\sqrt{\frac{ \sum_{i=0}^k d^2(x_i, x_{i+1}) }{(k+1) O(d^s(x_0, x^*)) } + \frac{1}{k+1} } \Big( \sum_{i=0}^k \sum_{j=0}^i \| \varepsilon_j \|_{x^*} \Big)  + \frac{1}{d(x_0,x^*)} \big(\epsilon_1 + \epsilon_2 + \epsilon_3 \big).
\end{align*}
By $\psi = O(d^2(x_0, x^*))$, the first term of the bound simplifies to
$\frac{S^{[0,\sigma]}_{k, \bar{\lambda}}}{1-\sigma} \sqrt{1 + O(d^{(8-3s)}(x_0, x^*) )},$
and similarly because $d(x_i, x_{i+1}) = O(d(x_0, x^*))$, $\| \varepsilon_j \|_{x^*} = O(d^2(x_0, x^*))$ under Assumption \ref{assump_normal_nei}, the second term simplifies to $O(\sqrt{d^2(x_0, x^*) + d^{(4-s)}(x_0, x^*)})$ and the last term reduces to $O(d^2( x_0, x^* ))$ as $\epsilon_1, \epsilon_2, \epsilon_3 = O(d^3(x_0, x^*))$. Hence we obtain 
\begin{equation*}
    \frac{d(\bar{x}_{c^*, x}, x^* )}{d(x_0, x^*)} \leq \frac{S^{[0,\sigma]}_{k, \bar{\lambda}}}{1-\sigma} \sqrt{1 + O(d^{(8-3s)}(x_0, x^*) )} + O(\sqrt{d^2(x_0, x^*) + d^{(4-s)}(x_0, x^*)}) + O(d^2( x_0, x^* )).
\end{equation*}
Finally we notice that the last two terms vanishes when $d(x_0, x^*) \xrightarrow{} 0$ for the choice of $s$. For the first term, given that when $d(x_0, x^*) \xrightarrow{} 0$, $\bar{\lambda} = O(d^{(s-2)}(x_0, x^*)) \xrightarrow{} 0$ and $O(d^{(8-3s)}(x_0, x^*)) \xrightarrow{} 0$ for $s \in (2, \frac{8}{3})$, then
\begin{equation*}
    \lim_{d(x_0, x^*) \xrightarrow{} 0} \frac{S^{[0,\sigma]}_{k, \bar{\lambda}}}{1-\sigma} \sqrt{1 + O(d^{(2-3s)}(x_0, x^*) )} = \frac{S_{k, 0}^{[0, \sigma]}}{1-\sigma} = \frac{1}{1-\sigma} \frac{2 }{\beta^{-k} + \beta^k}
\end{equation*}
where $\beta = \frac{1 - \sqrt{1-\sigma}}{1 + \sqrt{1-\sigma}}$. This follows because without regularization, $S^{[0,\sigma]}_{k,0}$ reduces to the rescaled and shifted Chebyshev polynomial. See for example \cite{d2021acceleration}.
\end{proof}

\subsection{Proof of Lemma \ref{lemma_alternative_aver_deviation}}

\begin{proof}[Proof of Lemma \ref{lemma_alternative_aver_deviation}]
First, we write 
\begin{equation*}
    \sum_{i=0}^k c_i \Delta_{x_i} = \Delta_{x_k} - \sum_{i=0}^{k-1} \theta_i (\Delta_{x_{i+1}} - \Delta_{x_i}).
\end{equation*}
By Lemma \ref{lemma_move_vectorsum}, we obtain
\begin{align*}
    &d \Big(\Exp_{x^*} \big(\sum_{i=0}^k c_i \Delta_{x_i} \big), \Exp_{x_k} \big( - \Gamma_{x^*}^{x_k} \sum_{i=0}^{k-1} \theta_i (\Delta_{x_{i+1}} - \Delta_{x_i}) \big)  \Big) \\
    &\qquad \qquad \leq d(x_k, x^*) C_\kappa \Big(  d(x_k, x^*) + \| \sum_{i=0}^{k-1} \theta_i (\Delta_{x_{i+1}} - \Delta_{x_i} )  \|_{x^*} \Big) \\
    &\qquad \qquad \leq d(x_k, x^*) C_\kappa \Big( d(x_k, x^*) + \sum_{i=0}^{k-1} \theta_i (d(x_{i+1}, x^*) + d(x_i, x^*) )\Big),
\end{align*}
where we use the fact that $C_\kappa(x)$ is increasing for $x > 0$. 
In addition, from Lemma \ref{lemma_metric_distort}, 
\begin{align*}
    d \Big( \bar{x}_{c, x}, \Exp_{x_k} \big( - \Gamma_{x^*}^{x_k} \sum_{i=0}^{k-1} \theta_i (\Delta_{x_{i+1}} - \Delta_{x_i}) \big) \Big) &\leq C_1 \| \sum_{i=0}^{k-1} \theta_i \big( \Gamma_{x^*}^{x_k} (\Delta_{x_{i+1}} - \Delta_{x_i}) - \Gamma_{x_i}^{x_k} \Exp_{x_i}^{-1}(x_{i+1}) \big) \|_{x_k} \\
    &\leq C_1 \sum_{i=0}^{k-1} \theta_i \| \Delta_{x_{i+1}} - \Delta_{x_i} - \Gamma_{x_k}^{x^*} \Gamma_{x_i}^{x_k} \Exp_{x_i}^{-1}(x_{i+1}) \|_{x^*}.
\end{align*}
Using Lemma \ref{diff_bound_gammatransport}, we obtain 
\begin{align*}
    \| \Delta_{x_{i+1}} - \Delta_{x_i} - \Gamma_{x_k}^{x^*} \Gamma_{x_i}^{x_k} \Exp_{x_i}^{-1}(x_{i+1}) \|_{x^*} &\leq C_0 d(x_i, x_k) d(x_k, x^*) d(x_i, x_{i+1}) \\
    &\qquad + C_2 d(x_i, x^*) C_\kappa \big( d(x_i, x^*) + d(x_i, x_{i+1}) \big).
\end{align*}
Let $e = \Delta_{\bar{x}_{c,x}} - \sum_{i=0}^k c_i \Delta_{x_i}$. Now combining the above results gives
\begin{align*}
    \| e\|_{x^*} &= \| \Delta_{\bar{x}_{c,x}} - \sum_{i=0}^k c_i \Delta_{x_i}\|_{x^*} \\
    &\leq C_2 d \Big( \bar{x}_{c,x}, \Exp_{x^*} \big( \sum_{i=0}^k c_i \Delta_{x_i} \big) \Big) \\
    &\leq C_2  d \Big( \bar{x}_{c, x}, \Exp_{x_k} \big( - \Gamma_{x^*}^{x_k} \sum_{i=0}^{k-1} \theta_i (\Delta_{x_{i+1}} - \Delta_{x_i}) \big) \Big) \\
    &\qquad + C_2 d \Big( \Exp_{x^*} \big(\sum_{i=0}^k c_i \Delta_{x_i} \big), \Exp_{x_k} \big( - \Gamma_{x^*}^{x_k} \sum_{i=0}^{k-1} \theta_i (\Delta_{x_{i+1}} - \Delta_{x_i}) \big)  \Big) \\
    &\leq C_2 C_1 \sum_{i=0}^{k-1} \theta_i \Big( C_0 d(x_i, x_k) d(x_k, x^*) d(x_i, x_{i+1}) + C_2 d(x_i, x^*) C_\kappa \big( d(x_i, x^*) + d(x_i, x_{i+1})  \Big) \\
    &\qquad + C_2 d(x_k, x^*) C_\kappa \Big( d(x_k, x^*) + \sum_{i=0}^{k-1} \theta_i (d(x_{i+1}, x^*) + d(x_i, x^*) )\Big).
\end{align*}
Under Assumption \ref{assump_normal_nei} and $C_\kappa(x) = O(x^2)$, we see $\| e\|_{x^*} = O(d^3(x_0, x^*))$.
\end{proof}

\subsection{Proof of Lemma \ref{lemma_frechetmean_bound}}

\begin{proof}[Proof of Lemma \ref{lemma_frechetmean_bound}]
Let $D(x) \coloneqq \frac{1}{2} \sum_{i=0}^k c_i d^2(x,x_i)$. Then we can show $\grad D(x) = - \sum_{i=0}^k c_i \Exp^{-1}_{x} (x_i)$. See for example \cite{alimisis2020continuous}. By the first-order stationarity,
\begin{equation*}
    \grad D(\bar{x}_{c,x}) = - \sum_{i=0}^k c_i \Exp_{\bar{x}_{c,x}}^{-1}(x_i) = 0
\end{equation*}
and $\grad D(x^*) = - \sum_{i=0}^k c_i \Exp^{-1}_{x^*} (x_i)$. 

The first claim that $d(\bar{x}_{c, x} , x^*) \leq \| \sum_{i=0}^k c_i \Delta_{x_i} \|_{x^*}$ follows from Lemma \cite[Lemma 10]{tripuraneni2018averaging} and we include here for completeness. Define a real-valued function $g(t) \coloneqq D\big( \Exp_{x^*}(t \eta) \big)$ with $\eta = \frac{\Delta_{\bar{x}_{c,x}}}{\| \Delta_{\bar{x}_{c,x}} \|_{x^*}}$. Under the assumption and definition of geodesic $\mu$-strongly convex, we see $g(t)$ is $\mu$-strongly convex in $t$. Thus, we have $g'(t_0) - g'(0) \geq \mu t_0$ for any $t_0$. Let $t_0 = \| \Delta_{\bar{x}_{c,x}} \|_{x^*}$ and denote the geodesic $\gamma(t) \coloneqq \Exp_{x^*}(t\eta)$. We derive that $g'(t) = \langle \grad D(\Exp_{x^*}(t\eta)), \gamma'(t) \rangle$ by chain rule. Then we have
$g'(t_0) = \langle \grad D(\bar{x}_{c,x}), \gamma'(t_0) \rangle_{\bar{x}_{c,x}} = 0$ and $g'(0) = \langle  \grad D(x^*), \eta \rangle$. Finally, we see
\begin{equation*}
    \| \grad D(x^*) \|^2_{x^*} \geq (g'(0))^2 = (g'(t_0) - g'(0))^2 \geq \mu^2 t_0^2 = \mu^2\| \Delta_{\bar{x}_{c,x}} \|_{x^*}^2,
\end{equation*}
where the first inequality is due to Cauchy–Schwarz inequality. The first claim is proved by noticing $\| \grad D(x^*) \|_{x^*} = \| \sum_{i=0}^k c_i \Delta_{x_i} \|_{x^*}$ and $\| \Delta_{\bar{x}_{c,x}} \|_{x^*} = d(\bar{x}_{c,x}, x^*)$.

For the second claim, we first observe from the proof of Lemma \ref{diff_bound_gammatransport} that 
\begin{align*}
    \| \Exp_{\bar{x}_{c,x}}^{-1}(x_i) - \Gamma_{x^*}^{\bar{x}_{c,x}} \big( \Exp_{x^*}^{-1}(x_i) - \Exp_{x^*}^{-1}(\bar{x}_{c,x})  \big)  \|_{\bar{x}_{c,x}} &\leq C_2 d(\bar{x}_{c,x}, x^*) C_\kappa\big( d(\bar{x}_{c,x}, x^*) + d(\bar{x}_{c,x}, x_i) \big) \\
    &= O(d^3(x_0, x^*)),
\end{align*}
where the order can be seen due to that $d(\bar{x}_{c,x}, x^*) \leq \frac{1}{\mu} \sum_{i=0}^k c_i d(x_i, x^*) = O(x_0, x^*)$ from the first claim. Thus let $\bar{\varepsilon} \coloneqq \Exp_{\bar{x}_{c,x}}^{-1}(x_i) - \Gamma_{x^*}^{\bar{x}_{c,x}} \big( \Exp_{x^*}^{-1}(x_i) - \Exp_{x^*}^{-1}(\bar{x}_{c,x})  \big) $, we have $\| \bar{\varepsilon} \|_{\bar{x}_{c,x}} = O(d^3(x_0, x^*))$. From the first order stationarity, we see
\begin{align*}
    0 = \sum_{i=0}^k c_i \Exp_{\bar{x}_{c,x}}^{-1}(x_i) &= \sum_{i=0}^k c_i \Big( \Gamma_{x^*}^{\bar{x}_{c,x}} \big( \Exp_{x^*}^{-1}(x_i) - \Exp_{x^*}^{-1}(\bar{x}_{c,x})  \big) + \bar{\varepsilon} \Big) \\
    &= \Gamma_{x^*}^{\bar{x}_{c,x}} \Big( \sum_{i=0}^k c_i \Delta_{x_i} - \Delta_{\bar{x}_{c,x}}  \Big) + \bar{\varepsilon}.
\end{align*}
Taking the norm and using the isometry of parallel transport, we obtain the desired result.
\end{proof}

\section{Proofs under general retraction and vector transport}

Here we show that when we use general retraction $\Retr$ in place of the exponential map $\Exp$, thus avoiding the lemma on metric distortion (Lemma \ref{lemma_move_vectorsum}, \ref{lemma_metric_distort}), we can still show a similar result as Lemma \ref{lemma_deviation_recur_tange_avera} but with an error on the order of $O(d^2(x_0, x^*))$ instead of $O(d^3(x_0, x^*))$ as for the case of exponential map. The main idea of proof follows from \cite{tripuraneni2018averaging}. The next proposition formalizes such claim. For this section, we denote $\Delta_{x} = \Retr_{x^*}^{-1}(x)$ for any $x \in \gX$ where the retraction has a smooth inverse. For general retraction, the deviation is on the order of $O(\| \Delta_{x_0} \|^2_{x^*}) = O(d^2(x_0, x^*))$ where we use the fact that retraction approximates the exponential map to the first order.

\begin{proposition}
\label{general_bounded_deviation_weighted_aver}
Suppose all iterates $x_i \in \gX$, a neighbourhood where retraction has a smooth inverse. Consider the weighted average $\bar{x}_{c, x} = \tilde{x}_k$ given by \eqref{mfd_recursive_average} with retraction.  Assume the sequence of iterates is non-divergent in retraction, i.e. $\| \Delta_{x_i}\|_{x^*}, \| \Delta_{\tilde{x}_i}\|_{x^*} = O(\|\Delta_{x_0} \|_{x^*})$. Then we have $\Delta_{\bar{x}_{c, x}} = \sum_{i=0}^k c_i \Delta_{x_i} + e$, with $\| e\|_{x^*} = O(\| \Delta_{x_0} \|_{x^*}^2)$,
\end{proposition}
\begin{proof}
The proof generalize the proof of \cite[Lemma 12]{tripuraneni2018averaging}.
First denote $\Retr_x^y \coloneqq \Retr_y^{-1} \circ \Retr_x$ and we notice that 
\begin{align*}
    {\Delta}_{\tilde{x}_{i+1}} = \Retr_{x^*}^{-1}(\tilde{x}_{i+1}) = \Retr^{-1}_{x^*} \Big( \Retr_{\tilde{x}_{i}} \big(  \gamma_{i+1} \Retr^{-1}_{\tilde{x}_{i}} \big( x_{i+1}  \big) \big) \Big) &= \Retr_{\tilde{x}_i}^{x^*} \Big( \gamma_{i+1} \Retr^{-1}_{\tilde{x}_i} \big( \Retr_{x^*} (\Delta_{x_{i+1}})  \big) \Big) \\
    &= \Retr_{\tilde{x}_i}^{x^*} \Big(  \gamma_{i+1} \big( \Retr_{\tilde{x}_i}^{x^*} \big)^{-1} (\Delta_{x_{i+1}}) \Big) \\
    &= F(\Delta_{x_{i+1}}),
\end{align*}
where we denote $\gamma_i = \frac{c_i}{\sum_{j=0}^i c_j}$ and $F: T_{x^*}\M \xrightarrow{} T_{x^*}\M$ defined as $F(u) = \Retr_{\tilde{x}_i}^{x^*} \Big(  \gamma_{i+1} \big( \Retr_{\tilde{x}_i}^{x^*} \big)^{-1} (u) \Big) $. In addition, it can be verified that $F(\Delta_{\tilde{x}_i}) = \Delta_{\tilde{x}_i}$. 

Now by chain rule, we have 
\begin{align*}
    \D F( u ) &= \D \Retr_{\tilde{x}_i}^{x^*} \Big( \gamma_{i+1} (\Retr_{\tilde{x}_i}^{x^*})^{-1}( u ) \Big) \Big[  \D \gamma_{i+1} (\Retr_{\tilde{x}_i}^{x^*})^{-1} (u)  \Big] \\
    &=  \gamma_{i+1} \D \Big( \frac{1}{\gamma_{i+1}} \Retr_{\tilde{x}_i}^{x^*} \Big)\Big( \gamma_{i+1} (\Retr^{x^*}_{\tilde{x}_i} )^{-1} (u) \Big) \Big[  \D \gamma_{i+1} (\Retr_{\tilde{x}_i}^{x^*})^{-1} (u) \Big] \\
    &= \gamma_{i+1} \Big( \D \gamma_{i+1}(\Retr^{x^*}_{\tilde{x}_i} )^{-1} (u) \Big)^{-1} \Big[ \D \gamma_{i+1}  (\Retr_{\tilde{x}_i}^{x^*})^{-1} (u) \Big] = \gamma_{i+1} \id,
\end{align*}
where the third inequality uses the inverse function theorem. Hence the Taylor expansion of $F$ at $\Delta_{\tilde{x}_i}$ up to second order gives 
\begin{align*}
    \Delta_{\tilde{x}_{i+1}} = F(\Delta_{x_{i+1}}) &= F(\Delta_{\tilde{x}_i}) + \gamma_{i+1} (\Delta_{x_{i+1}}  - \Delta_{\tilde{x}_i}) + \tilde{\epsilon}_i \\
    &= (1- \gamma_{i+1}) \Delta_{\tilde{x}_i} + \gamma_{i+1} \Delta_{x_{i+1}} + \tilde{\epsilon}_i.
\end{align*}
where we let $\tilde{\epsilon}_i = O(\| \Delta_{x_{i+1}} - \Delta_{\tilde{x}_i} \|^2_{x^*})$. From the expansion, it follows that $\Delta_{\tilde{x}_{i+1}} = \frac{\sum_{j=0}^i c_i}{\sum_{j=0}^{i+1} c_j} \Delta_{\tilde{x}_i} + \frac{c_{i+1}}{\sum_{j=0}^{i+1} c_j} \Delta_{x_{i+1}} + \tilde{\epsilon}_i$, which yields
\begin{equation*}
    (\sum_{j=0}^{i+1} c_j) \Delta_{\tilde{x}_{i+1}} = (\sum_{j=0}^i c_j ) \Delta_{\tilde{x}_i} + c_{i+1} \Delta_{x_{i+1}} + (\sum_{j=0}^i c_j ) \tilde{\epsilon}_i = \sum_{j=0}^{i+1} c_j \Delta_{x_j} + \sum_{j = 0}^i (\sum_{\ell = 0}^j c_\ell) \tilde{\epsilon}_j,
\end{equation*}
where the second equality follows by expanding the first equality. Let $i = k-1$, this leads to 
\begin{equation*}
    \Delta_{\bar{x}_{c,x}} = \Delta_{\tilde{x}_{k}} = \sum_{j=0}^k c_j \Delta_{x_j} + e,
\end{equation*}
where we let $e = \sum_{j=0}^{k-1} (\sum_{\ell = 0}^j c_\ell) \tilde{\epsilon}_j = O \big(\sum_{j=0}^{k-1} (\sum_{\ell = 0}^j c_\ell) ( \| \Delta_{x_{j+1}} \|^2_{x^*} + \| \Delta_{\tilde{x}_j} \|^2_{x^*} ) \big)$. We observe that $\| \Delta_{x_{i+1}}\|^2_{x^*} = O(\| \Delta_{x_{0}} \|^2_{x^*})$ and $\| \Delta_{\tilde{x}_j} \|_{x^*}^2 = O(\| \Delta_{x_{0}} \|_{x^*}^2)$ due to the non-divergent assumption. The proof is complete. 
\end{proof}

\subsection{Proof of Theorem \ref{theorem_retr_convergence}}

\begin{proof}[Proof of Theorem \ref{theorem_retr_convergence}]
Here we only provide a sketch of proof because the main idea is exactly the same as the case of exponential map.

Under general retraction and vector transport, an analogue of Lemma \ref{lin_iter_approx_lemma} holds. That is, 
\begin{equation}
    \Retr_{x^*}^{-1}(x_i) = (\id - \eta \hess f(x^*)) [\Retr_{x^*}^{-1}(x_{i-1})] + \varepsilon_i, \label{retr_lemma3}
\end{equation}
where $\| \varepsilon_i \|_{x^*} = O(d^2(x_i, x^*))$. To show \eqref{retr_lemma3}, we follow the exact same steps as the proof for Lemma \ref{lin_iter_approx_lemma} where we replace exponential map with retraction. The only difference is that the second order derivative is no longer the Riemann curvature tensor. In addition, we have shown in Proposition \ref{general_bounded_deviation_weighted_aver} that for retraction, we also have 
\begin{equation}
    \Retr_{x^*}^{-1}({\bar{x}_{c,x}}) = \sum_{i=0}^k c_i \Retr_{x^*}^{-1}(x_i) + e \label{retr_lemma4}
\end{equation}
with $\| e \|_{x^*} = O(d^2(x_0, x^*))$. 

Further, we still consider the same error bound decomposition, i.e.,
\begin{equation*}
    d(\bar{x}_{c^*, x}, x^*) \leq  {d(\bar{x}_{\hat{c}^*, \hat{x}}, x^*)} + {d(\bar{x}_{\hat{c}^*,\hat{x}}, \bar{x}_{c^*, \hat{x}})} + {d(\bar{x}_{c^*, \hat{x}}, \bar{x}_{c^*, x})}. 
\end{equation*}

(I). For the linear term $d(\bar{x}_{\hat{c}^*, \hat{x}}, x^*)$, we first see the linearized iterates $\hat{x}_i$ enjoys the same convergence as in Lemma \ref{convergence_lin_iterate_theorem} that 
\begin{align}
    \| \sum_{i=0}^k \hat{c}_i^* \Retr_{x^*}^{-1}(\hat{x}_i) \|_{x^*} &\leq \frac{\|\Retr_{x^*}^{-1}(x_0) \|_{x^*}}{1-\sigma} \sqrt{(S^{[0,\sigma]}_{k, \bar{\lambda}})^2 - \frac{\lambda}{\|\Retr_{x^*}^{-1}(x_0) \|_{x^*}^2} \| \hat{c}^* \|^2_2}, \label{retr_linear_convg}
\end{align}
where $\bar{\lambda} \coloneqq \lambda/\|\Retr_{x^*}^{-1}(x_0) \|_{x^*}^2$ and we use Assumption \ref{retr_exp_bound}. Combining \eqref{retr_linear_convg} with \eqref{retr_lemma4} yields
\begin{align*}
    d(\bar{x}_{\hat{c}^*, \hat{x}}, x^*) \leq \frac{1}{a_0} \| \Retr_{x^*}^{-1}(\bar{x}_{\hat{c}^*, \hat{x}}) \|_{x^*} &\leq \| \sum_{i=0}^k \hat{c}_i^* \Retr_{x^*}^{-1}(\hat{x}_i) \|_{x^*} + \epsilon_1, \\
    &\leq \frac{\|\Retr_{x^*}^{-1}(x_0) \|_{x^*}}{a_0(1-\sigma)} \sqrt{(S^{[0,\sigma]}_{k, \bar{\lambda}})^2 - \frac{\lambda}{\|\Retr_{x^*}^{-1}(x_0) \|_{x^*}^2} \| \hat{c}^* \|^2_2 } + \epsilon_1,
\end{align*}
with $\epsilon_1 = O(d^2(x_0, x^*))$. 

(II). For the stability term $d(\bar{x}_{\hat{c}^*,\hat{x}}, \bar{x}_{c^*, \hat{x}})$, we first use Assumption \ref{retr_exp_bound} to show 
\begin{align*}
    \| \Delta_{\bar{x}_{\hat{c}^*,\hat{x}}} - \Delta_{\bar{x}_{c^*, \hat{x}}} - \big(\Retr_{x^*}^{-1}(\bar{x}_{\hat{c}^*,\hat{x}}) -  \Retr_{x^*}^{-1}({\bar{x}_{c^*, \hat{x}}}) \big) \|_{x^*} &\leq a_2 \| \Retr_{x^*}^{-1}(\bar{x}_{\hat{c}^*,\hat{x}}) \|^2_{x^*} + a_2 \| \Retr_{x^*}^{-1}({\bar{x}_{c^*, \hat{x}}}) \big)  \|_{x^*}^2  \\
    &\leq a_2a_1^2 \big( d^2(  \bar{x}_{\hat{c}^*,\hat{x}},x^* ) + d^2(\bar{x}_{c^*, \hat{x}}, x^*) \big) \\
    &= O(d^2(x_0, x^*)).
\end{align*}
Let $\epsilon_r \coloneqq \Delta_{\bar{x}_{\hat{c}^*,\hat{x}}} - \Delta_{\bar{x}_{c^*, \hat{x}}} - \big(\Retr_{x^*}^{-1}(\bar{x}_{\hat{c}^*,\hat{x}}) -  \Retr_{x^*}^{-1}({\bar{x}_{c^*, \hat{x}}}) \big)$, we have $\| \epsilon_r\|_{x^*} = O(d^2(x_0, x^*))$. In addition, based on Assumption \ref{vector_transport_bound}, we show 
\begin{align*}
    &\| \gT_{x_k}^{x^*} r_i - \big(  \Retr_{x^*}^{-1}(x_{i+1}) - \Retr_{x^*}^{-1}(x_i) \big) - \Gamma_{x_k}^{x^*} r_i + \big( \Delta_{x_{i+1}} - \Delta_{x_{i}} \big) \|_{x^*} \\
    &\leq \| \gT_{x_k}^{x^*} r_i - \Gamma_{x_k}^{x^*} r_i \|_{x^*} + O(d^2(x_0, x^*)) =  O(d^2(x_0, x^*)),
\end{align*}
where we use Assumption \ref{retr_exp_bound}, \ref{vector_transport_bound} and notice $\| r_i \|_{x_i} = \| \Retr_{x_i}^{-1}(x_{i+1}) \|_{x_i} \leq a_1 d(x_i, x_{i+1}) = O(d(x_0, x^*))$. Let $\epsilon_v \coloneqq \gT_{x_k}^{x^*} r_i - \big(  \Retr_{x^*}^{-1}(x_{i+1}) - \Retr_{x^*}^{-1}(x_i) \big) - \Gamma_{x_k}^{x^*} r_i + \big( \Delta_{x_{i+1}} - \Delta_{x_{i}} \big)$, we have $\| \epsilon_v \|_{x^*} = O(d^2(x_0, x^*))$.

Using Lemma \ref{lemma_metric_distort}, we then obtain
\begin{align*}
    d(\bar{x}_{\hat{c}^*,\hat{x}}, \bar{x}_{c^*, \hat{x}}) \leq C_1 \| \Delta_{\bar{x}_{\hat{c}^*,\hat{x}}} - \Delta_{\bar{x}_{c^*, \hat{x}}} \|_{x^*} &\leq C_1 \| \Retr_{x^*}^{-1}(\bar{x}_{\hat{c}^*,\hat{x}}) -  \Retr_{x^*}^{-1}({\bar{x}_{c^*, \hat{x}}})  \|_{x^*} + C_1 \| \epsilon_r \|_{x^*} \\
    &\leq \frac{C_1 \| \Retr_{x^*}^{-1}(x_0) \|_{x^*}}{1-\sigma} \| c^* - \hat{c}^* \|_2 + O(d^2(x_0, x^*)),
\end{align*}
where we apply \eqref{retr_lemma4}.
Now we proceed to bound $\| c^* - \hat{c}^* \|_2 \leq \frac{\| P \|_2}{\lambda} \| \hat{c}^* \|_2$ in a similar manner as Lemma \ref{coeff_bound} where $P = R - \hat{R}$. From the proof of Lemma \ref{coeff_bound}, we have 
\begin{equation*}
    \| P \|_2 \leq \frac{2}{1-\sigma} \|\Retr_{x^*}^{-1}(x_0) \|_{x^*} \| E \|_2 + \| E\|_2^2,
\end{equation*}
where $\| E \|_2 \leq \sum_{i=0}^k \| \gT_{x_k}^{x^*} r_i - \hat{r}_i \|_{x^*}$. Thus it remains to bound $\| \gT_{x_k}^{x^*} r_i - \hat{r}_i \|_{x^*}$. Similarly, we can show
\begin{align*}
    \| \gT_{x_k}^{x^*} r_i - \hat{r}_i \|_{x^*} &\leq \| \gT_{x_k}^{x^*} r_i - \big(\Retr_{x^*}^{-1}(x_{i+1}) - \Retr_{x^*}^{-1}(x_i) \big) \|_{x^*} + \sum_{j=1}^{i+1} \| \varepsilon_j \|_{x^*} \\
    &\leq \| \Gamma_{x_k}^{x^*} r_i - \big( \Delta_{x_{i+1}} - \Delta_{x_{i}} \big) \|_{x^*} + \| \epsilon_v \|_{x^*} + \sum_{j=1}^{i+1} \| \varepsilon_j \|_{x^*} = O(d^2(x_0, x^*)),
\end{align*}
where $\varepsilon_j$ is defined in \eqref{retr_lemma3} and we use Lemma \ref{diff_bound_gammatransport} for the exponential map. Thus $\| P\|_2 \leq 2 \psi \frac{a_1}{1-\sigma} d(x_0, x^*) + \psi^2$ where $\psi = O(d^2(x_0, x^*))$. This leads to 
\begin{align*}
    d(\bar{x}_{\hat{c}^*,\hat{x}}, \bar{x}_{c^*, \hat{x}}) \leq \frac{C_1 \|\Retr_{x^*}^{-1}(x_0) \|_{x^*}}{\lambda(1-\sigma)} \Big(  \frac{2 \psi}{1-\sigma} \|\Retr_{x^*}^{-1}(x_0) \|_{x^*}  + \psi^2 \Big) \| \hat{c}^* \|_2 + \epsilon_2,
\end{align*}
where $\epsilon_2 = O(d^2(x_0, x^*))$.

(III). Finally for the nonlinearity term ${d(\bar{x}_{c^*, \hat{x}}, \bar{x}_{c^*, x})}$, we show 
\begin{align*}
    d(\bar{x}_{c^*, \hat{x}}, \bar{x}_{c^*, x}) &\leq C_1 \| \Delta_{\bar{x}_{c^*, \hat{x}}} - \Delta_{\bar{x}_{c^*, x}} \|_{x^*} \leq C_1 \| \Retr_{x^*}^{-1}(\bar{x}_{c^*,\hat{x}}) -  \Retr_{x^*}^{-1}({\bar{x}_{c^*, {x}}}) \|_{x^*} + O(d^2(x_0, x^*)) \\
    &\leq C_1 \| c^* \|_2 (\sum_{i=0}^k \| \Retr_{x^*}^{-1}(x_i) - \Retr_{x^*}^{-1}(\hat{x}_i) \|_{x^*}) + O(d^2(x_0, x^*)) \\
    &\leq C_1 \sqrt{\frac{\sum_{i=0}^k  \|\Retr_{x_i}^{-1}(x_{i+1}) \|_{x_i}^2 + \lambda}{(k+1)\lambda}}\Big( \sum_{i=0}^k \sum_{j=0}^i \| \varepsilon_j \|_{x^*} \Big) + \epsilon_3,
\end{align*}
where $\epsilon_3 = O(d^2(x_0, x^*))$ and we follow similar steps as in Lemma \ref{coeff_bound}. 

Finally, combining results from (I), (II), (III), we have
\begin{align*}
    d(\bar{x}_{c^*, x}, x^*) &\leq \frac{\|\Retr_{x^*}^{-1}(x_0) \|_{x^*}}{a_0(1-\sigma)} \sqrt{(S^{[0,\sigma]}_{k, \bar{\lambda}})^2 - \frac{\lambda}{\|\Retr_{x^*}^{-1}(x_0) \|_{x^*}^2} \| \hat{c}^* \|^2_2}  \\
    &\qquad + \frac{C_1 \|\Retr_{x^*}^{-1}(x_0) \|_{x^*}}{\lambda(1-\sigma)} \Big(  \frac{2 \psi}{1-\sigma} \|\Retr_{x^*}^{-1}(x_0) \|_{x^*}  + \psi^2 \Big) \| \hat{c}^* \|_2  \\
    &\qquad + C_1 \sqrt{\frac{\sum_{i=0}^k  \|\Retr_{x_i}^{-1}(x_{i+1}) \|_{x_i}^2 + \lambda}{(k+1)\lambda}}\Big( \sum_{i=0}^k \sum_{j=0}^i \| \varepsilon_j \|_{x^*} \Big)  +\epsilon_1 + \epsilon_2 +  \epsilon_3.
\end{align*}
Maximizing the bound over $\| \hat{c}^* \|_2$ yields
\begin{align*}
    d(\bar{x}_{c^*, x}, x^*) &\leq S^{[0,\sigma]}_{k, \bar{\lambda}} \sqrt{\frac{\|\Retr_{x^*}^{-1}(x_0) \|_{x^*}^2}{a_0^2 (1-\sigma)^2} +  \frac{C_1^2 \|\Retr_{x^*}^{-1}(x_0) \|_{x^*}^4 \big( \frac{2 \psi}{1-\sigma} \|\Retr_{x^*}^{-1}(x_0) \|_{x^*}  + \psi^2 \big)^2 }{\lambda^3(1-\sigma)^2}} \\
    &\qquad + C_1 \sqrt{\frac{\sum_{i=0}^k  \|\Retr_{x_i}^{-1}(x_{i+1}) \|_{x_i}^2 + \lambda}{(k+1)\lambda}}\Big( \sum_{i=0}^k \sum_{j=0}^i \| \varepsilon_j \|_{x^*} \Big)  +\epsilon_1 + \epsilon_2 +  \epsilon_3.
\end{align*}
Finally, to see the asymptotic convergence rate, we notice that $ \|\Retr_{x^*}^{-1}(x_{0}) \|_{x^*} = O(d(x_0, x^*))$ and $\lim_{d(x_0, x^*) \xrightarrow{} 0} \frac{1}{d(x_0, x^*)} (\epsilon_1 + \epsilon_2 + \epsilon_3) = 0$. 
\end{proof}


\begin{thebibliography}{10}

\bibitem{absil2007trust}
P-A Absil, Christopher~G Baker, and Kyle~A Gallivan.
\newblock Trust-region methods on {R}iemannian manifolds.
\newblock {\em Foundations of Computational Mathematics}, 7(3):303--330, 2007.

\bibitem{absil2009optimization}
P-A Absil, Robert Mahony, and Rodolphe Sepulchre.
\newblock Optimization algorithms on matrix manifolds.
\newblock In {\em Optimization Algorithms on Matrix Manifolds}. Princeton
  University Press, 2009.

\bibitem{agarwal2021adaptive}
Naman Agarwal, Nicolas Boumal, Brian Bullins, and Coralia Cartis.
\newblock Adaptive regularization with cubics on manifolds.
\newblock {\em Mathematical Programming}, 188(1):85--134, 2021.

\bibitem{ahn2020nesterov}
Kwangjun Ahn and Suvrit Sra.
\newblock {From Nesterov’s estimate sequence to Riemannian acceleration}.
\newblock In {\em Conference on Learning Theory}, pages 84--118. PMLR, 2020.

\bibitem{aitken_1927}
A.~C. Aitken.
\newblock On {B}ernoulli's numerical solution of algebraic equations.
\newblock {\em Proceedings of the Royal Society of Edinburgh}, 46:289–305,
  1927.

\bibitem{alimisis2020continuous}
Foivos Alimisis, Antonio Orvieto, Gary B{\'e}cigneul, and Aurelien Lucchi.
\newblock A continuous-time perspective for modeling acceleration in
  {R}iemannian optimization.
\newblock In {\em International Conference on Artificial Intelligence and
  Statistics}, pages 1297--1307. PMLR, 2020.

\bibitem{andrews2010ricci}
Ben Andrews and Christopher Hopper.
\newblock {\em {The Ricci flow in Riemannian geometry: a complete proof of the
  differentiable $1/4$-pinching sphere theorem}}.
\newblock springer, 2010.

\bibitem{becigneul2018riemannian}
Gary Becigneul and Octavian-Eugen Ganea.
\newblock Riemannian adaptive optimization methods.
\newblock In {\em International Conference on Learning Representations}, 2018.

\bibitem{bhatia2009positive}
Rajendra Bhatia.
\newblock Positive definite matrices.
\newblock In {\em Positive Definite Matrices}. Princeton university press,
  2009.

\bibitem{bollapragada2022nonlinear}
Raghu Bollapragada, Damien Scieur, and Alexandre d’Aspremont.
\newblock Nonlinear acceleration of momentum and primal-dual algorithms.
\newblock {\em Mathematical Programming}, pages 1--38, 2022.

\bibitem{bonnabel2013stochastic}
Silvere Bonnabel.
\newblock Stochastic gradient descent on {R}iemannian manifolds.
\newblock {\em IEEE Transactions on Automatic Control}, 58(9):2217--2229, 2013.

\bibitem{boumal2020introduction}
Nicolas Boumal.
\newblock An introduction to optimization on smooth manifolds.
\newblock {\em Available online, May}, 3, 2020.

\bibitem{boumal2015low}
Nicolas Boumal and P-A Absil.
\newblock Low-rank matrix completion via preconditioned optimization on the
  {G}rassmann manifold.
\newblock {\em Linear Algebra and its Applications}, 475:200--239, 2015.

\bibitem{boumal2019global}
Nicolas Boumal, Pierre-Antoine Absil, and Coralia Cartis.
\newblock Global rates of convergence for nonconvex optimization on manifolds.
\newblock {\em IMA Journal of Numerical Analysis}, 39(1):1--33, 2019.

\bibitem{boumal2014manopt}
Nicolas Boumal, Bamdev Mishra, P-A Absil, and Rodolphe Sepulchre.
\newblock Manopt, a matlab toolbox for optimization on manifolds.
\newblock {\em The Journal of Machine Learning Research}, 15(1):1455--1459,
  2014.

\bibitem{brezinski2018shanks}
Claude Brezinski, Michela Redivo-Zaglia, and Yousef Saad.
\newblock Shanks sequence transformations and {A}nderson acceleration.
\newblock {\em SIAM Review}, 60(3):646--669, 2018.

\bibitem{cherian2016riemannian}
Anoop Cherian and Suvrit Sra.
\newblock Riemannian dictionary learning and sparse coding for positive
  definite matrices.
\newblock {\em IEEE Transactions on Neural Networks and Learning Systems},
  28(12):2859--2871, 2016.

\bibitem{criscitiello2022accelerated}
Christopher Criscitiello and Nicolas Boumal.
\newblock An accelerated first-order method for non-convex optimization on
  manifolds.
\newblock {\em Foundations of Computational Mathematics}, pages 1--77, 2022.

\bibitem{criscitiello2022negative}
Christopher Criscitiello and Nicolas Boumal.
\newblock Negative curvature obstructs acceleration for strongly geodesically
  convex optimization, even with exact first-order oracles.
\newblock In {\em Conference on Learning Theory}, pages 496--542. PMLR, 2022.

\bibitem{duruisseaux2022variational}
Valentin Duruisseaux and Melvin Leok.
\newblock A variational formulation of accelerated optimization on {R}iemannian
  manifolds.
\newblock {\em SIAM Journal on Mathematics of Data Science}, 4(2):649--674,
  2022.

\bibitem{d2021acceleration}
Alexandre d’Aspremont, Damien Scieur, Adrien Taylor, et~al.
\newblock Acceleration methods.
\newblock {\em Foundations and Trends{\textregistered} in Optimization},
  5(1-2):1--245, 2021.

\bibitem{edelman1998geometry}
Alan Edelman, Tom{\'a}s~A Arias, and Steven~T Smith.
\newblock The geometry of algorithms with orthogonality constraints.
\newblock {\em SIAM Journal on Matrix Analysis and Applications},
  20(2):303--353, 1998.

\bibitem{hamilton2021no}
Linus Hamilton and Ankur Moitra.
\newblock No-go theorem for acceleration in the hyperbolic plane.
\newblock {\em arXiv:2101.05657}, 2021.

\bibitem{han2021improved}
Andi Han and Junbin Gao.
\newblock Improved variance reduction methods for {R}iemannian non-convex
  optimization.
\newblock {\em IEEE Transactions on Pattern Analysis and Machine Intelligence},
  2021.

\bibitem{hanmomentum2021}
Andi Han and Junbin Gao.
\newblock Riemannian stochastic recursive momentum method for non-convex
  optimization.
\newblock In {\em International Joint Conference on Artificial Intelligence},
  pages 2505--2511, 8 2021.

\bibitem{han2022riemannian}
Andi Han, Bamdev Mishra, Pratik Jawanpuria, and Junbin Gao.
\newblock Riemannian block {SPD} coupling manifold and its application to
  optimal transport.
\newblock {\em arXiv:2201.12933}, 2022.

\bibitem{harandi2013dictionary}
Mehrtash Harandi, Conrad Sanderson, Chunhua Shen, and Brian~C Lovell.
\newblock Dictionary learning and sparse coding on {G}rassmann manifolds: An
  extrinsic solution.
\newblock In {\em Proceedings of the IEEE International Conference on Computer
  Vision}, pages 3120--3127, 2013.

\bibitem{huang2013optimization}
Wen Huang.
\newblock {\em {Optimization algorithms on Riemannian manifolds with
  applications}}.
\newblock PhD thesis, The Florida State University, 2013.

\bibitem{huang2015riemannian}
Wen Huang, P-A Absil, and Kyle~A Gallivan.
\newblock A riemannian symmetric rank-one trust-region method.
\newblock {\em Mathematical Programming}, 150(2):179--216, 2015.

\bibitem{huang2015broyden}
Wen Huang, Kyle~A Gallivan, and P-A Absil.
\newblock {A Broyden class of quasi-Newton methods for Riemannian
  optimization}.
\newblock {\em SIAM Journal on Optimization}, 25(3):1660--1685, 2015.

\bibitem{jin2022understanding}
Jikai Jin and Suvrit Sra.
\newblock {Understanding Riemannian acceleration via a proximal extragradient
  framework}.
\newblock In {\em Conference on Learning Theory}, pages 2924--2962. PMLR, 2022.

\bibitem{karcher1977riemannian}
Hermann Karcher.
\newblock Riemannian center of mass and mollifier smoothing.
\newblock {\em Communications on Pure and Applied Mathematics}, 30(5):509--541,
  1977.

\bibitem{kasai2019riemannian}
Hiroyuki Kasai, Pratik Jawanpuria, and Bamdev Mishra.
\newblock Riemannian adaptive stochastic gradient algorithms on matrix
  manifolds.
\newblock In {\em International Conference on Machine Learning}, pages
  3262--3271. PMLR, 2019.

\bibitem{kasai2018riemannian}
Hiroyuki Kasai, Hiroyuki Sato, and Bamdev Mishra.
\newblock Riemannian stochastic recursive gradient algorithm.
\newblock In {\em International Conference on Machine Learning}, pages
  2516--2524. PMLR, 2018.

\bibitem{keshavan2009gradient}
Raghunandan~H Keshavan and Sewoong Oh.
\newblock A gradient descent algorithm on the {G}rassman manifold for matrix
  completion.
\newblock {\em arXiv:0910.5260}, 2009.

\bibitem{kim2022accelerated}
Jungbin Kim and Insoon Yang.
\newblock Accelerated gradient methods for geodesically convex optimization:
  Tractable algorithms and convergence analysis.
\newblock In {\em International Conference on Machine Learning}, pages
  11255--11282. PMLR, 2022.

\bibitem{liu2017accelerated}
Yuanyuan Liu, Fanhua Shang, James Cheng, Hong Cheng, and Licheng Jiao.
\newblock Accelerated first-order methods for geodesically convex optimization
  on {R}iemannian manifolds.
\newblock In {\em Advances in Neural Information Processing Systems},
  volume~30, 2017.

\bibitem{mangoubi2018rapid}
Oren Mangoubi and Aaron Smith.
\newblock Rapid mixing of geodesic walks on manifolds with positive curvature.
\newblock {\em The Annals of Applied Probability}, 28(4):2501--2543, 2018.

\bibitem{martinez2022global}
David Mart{\'\i}nez-Rubio.
\newblock Global {R}iemannian acceleration in hyperbolic and spherical spaces.
\newblock In {\em International Conference on Algorithmic Learning Theory},
  pages 768--826. PMLR, 2022.

\bibitem{mishra2021manifold}
Bamdev Mishra, NTV Satyadev, Hiroyuki Kasai, and Pratik Jawanpuria.
\newblock Manifold optimization for non-linear optimal transport problems.
\newblock {\em arXiv:2103.00902}, 2021.

\bibitem{nesterov2003introductory}
Yurii Nesterov.
\newblock {\em Introductory lectures on convex optimization: A basic course},
  volume~87.
\newblock Springer Science \& Business Media, 2003.

\bibitem{nesterov1983method}
Yurii~E Nesterov.
\newblock A method for solving the convex programming problem with convergence
  rate o (1/k\^{} 2).
\newblock In {\em Dokl. akad. nauk Sssr}, volume 269, pages 543--547, 1983.

\bibitem{qi2010riemannian}
Chunhong Qi, Kyle~A Gallivan, and P-A Absil.
\newblock Riemannian {BFGS} algorithm with applications.
\newblock In {\em Recent Advances in Optimization and its Applications in
  Engineering}, pages 183--192. Springer, 2010.

\bibitem{ring2012optimization}
Wolfgang Ring and Benedikt Wirth.
\newblock Optimization methods on {R}iemannian manifolds and their application
  to shape space.
\newblock {\em SIAM Journal on Optimization}, 22(2):596--627, 2012.

\bibitem{sato2019riemannian}
Hiroyuki Sato, Hiroyuki Kasai, and Bamdev Mishra.
\newblock Riemannian stochastic variance reduced gradient algorithm with
  retraction and vector transport.
\newblock {\em SIAM Journal on Optimization}, 29(2):1444--1472, 2019.

\bibitem{scieur2020regularized}
Damien Scieur, Alexandre d’Aspremont, and Francis Bach.
\newblock Regularized nonlinear acceleration.
\newblock {\em Mathematical Programming}, 179(1):47--83, 2020.

\bibitem{scieur2018online}
Damien Scieur, Edouard Oyallon, Alexandre d'Aspremont, and Francis Bach.
\newblock Online regularized nonlinear acceleration.
\newblock {\em arXiv:1805.09639}, 2018.

\bibitem{shanks1955non}
Daniel Shanks.
\newblock Non-linear transformations of divergent and slowly convergent
  sequences.
\newblock {\em Journal of Mathematics and Physics}, 34(1-4):1--42, 1955.

\bibitem{shi2021coupling}
Dai Shi, Junbin Gao, Xia Hong, ST~Boris~Choy, and Zhiyong Wang.
\newblock Coupling matrix manifolds assisted optimization for optimal transport
  problems.
\newblock {\em Machine Learning}, 110(3):533--558, 2021.

\bibitem{sidi1986acceleration}
Avram Sidi, William~F Ford, and David~A Smith.
\newblock Acceleration of convergence of vector sequences.
\newblock {\em SIAM Journal on Numerical Analysis}, 23(1):178--196, 1986.

\bibitem{siegel2019accelerated}
Jonathan~W Siegel.
\newblock Accelerated optimization with orthogonality constraints.
\newblock {\em arXiv:1903.05204}, 2019.

\bibitem{su2014differential}
Weijie Su, Stephen Boyd, and Emmanuel Candes.
\newblock A differential equation for modeling {N}esterov’s accelerated
  gradient method: theory and insights.
\newblock In {\em Advances in Neural Information Processing Systems},
  volume~27, 2014.

\bibitem{sun2019escaping}
Yue Sun, Nicolas Flammarion, and Maryam Fazel.
\newblock Escaping from saddle points on {R}iemannian manifolds.
\newblock {\em Advances in Neural Information Processing Systems}, 32, 2019.

\bibitem{tripuraneni2018averaging}
Nilesh Tripuraneni, Nicolas Flammarion, Francis Bach, and Michael~I Jordan.
\newblock {Averaging stochastic gradient descent on Riemannian manifolds}.
\newblock In {\em Conference On Learning Theory}, pages 650--687. PMLR, 2018.

\bibitem{udriste2013convex}
Constantin Udriste.
\newblock {\em {Convex functions and optimization methods on Riemannian
  manifolds}}, volume 297.
\newblock Springer Science \& Business Media, 2013.

\bibitem{vandereycken2013low}
Bart Vandereycken.
\newblock Low-rank matrix completion by {R}iemannian optimization.
\newblock {\em SIAM Journal on Optimization}, 23(2):1214--1236, 2013.

\bibitem{waldmann2012geometric}
Stefan Waldmann.
\newblock Geometric wave equations.
\newblock {\em arXiv:1208.4706}, 2012.

\bibitem{walker2011anderson}
Homer~F Walker and Peng Ni.
\newblock Anderson acceleration for fixed-point iterations.
\newblock {\em SIAM Journal on Numerical Analysis}, 49(4):1715--1735, 2011.

\bibitem{wibisono2016variational}
Andre Wibisono, Ashia~C Wilson, and Michael~I Jordan.
\newblock A variational perspective on accelerated methods in optimization.
\newblock {\em Proceedings of the National Academy of Sciences},
  113(47):E7351--E7358, 2016.

\bibitem{wynn1956device}
Peter Wynn.
\newblock On a device for computing the e m (s n) transformation.
\newblock {\em Mathematical Tables and Other Aids to Computation}, pages
  91--96, 1956.

\bibitem{zhang2016riemannian}
Hongyi Zhang, Sashank J~Reddi, and Suvrit Sra.
\newblock Riemannian {SVRG}: Fast stochastic optimization on {R}iemannian
  manifolds.
\newblock In {\em Advances in Neural Information Processing Systems},
  volume~29, 2016.

\bibitem{zhang2016first}
Hongyi Zhang and Suvrit Sra.
\newblock First-order methods for geodesically convex optimization.
\newblock In {\em Conference on Learning Theory}, pages 1617--1638. PMLR, 2016.

\bibitem{zhang2018estimate}
Hongyi Zhang and Suvrit Sra.
\newblock An estimate sequence for geodesically convex optimization.
\newblock In {\em Conference On Learning Theory}, pages 1703--1723. PMLR, 2018.

\bibitem{zhao2015riemannian}
Zhi Zhao, Zheng-Jian Bai, and Xiao-Qing Jin.
\newblock A riemannian newton algorithm for nonlinear eigenvalue problems.
\newblock {\em SIAM Journal on Matrix Analysis and Applications},
  36(2):752--774, 2015.

\end{thebibliography}
\end{document}